\documentclass{amsart}

\usepackage{amscd,amssymb,epsfig, epsf,pinlabel,graphicx}
\usepackage{epic,eepic}
\usepackage[dvipsnames]{xcolor}

\usepackage[all]{xy}
\SelectTips{cm}{}
\allowdisplaybreaks

\numberwithin{equation}{subsection}

\newtheorem{propo}{Proposition}[section]

\newtheorem{theor}[propo]{Theorem}
\newtheorem{lemma}[propo]{Lemma}

\theoremstyle{definition}

\theoremstyle{remark}

\let\oldmarginpar\marginpar
\renewcommand\marginpar[1]{\oldmarginpar{\footnotesize #1}}

\newcommand{\CC}{\mathbb{C}}
\newcommand{\QQ}{\mathbb{Q}}
\newcommand{\ZZ}{\mathbb{Z}}

\newcommand{\RR}{\mathbb{R}}

\newcommand{\kk}{\mathbb{K}}

\newcommand{\g}{\mathfrak{g}}

\newcommand{\Hom}{\operatorname{Hom}}
\newcommand{\Ker}{\operatorname{Ker}}

\newcommand{\ev}{\operatorname{ev}}

\newcommand{\Der}{\operatorname{Der}}

\newcommand{\card}{\operatorname{card}}

\newcommand{\id}{\operatorname{id}}

\newcommand{\GL}{\operatorname{GL}}
\newcommand{\Mat}{\operatorname{Mat}}

\newcommand{\Map}{\operatorname{Map}}
\newcommand{\Homeo}{\operatorname{Homeo}}
\newcommand{\tr}{\operatorname{tr}}
\newcommand{\Lie}{\operatorname{Lie}}
\newcommand{\pr}{\operatorname{pr}}

\newcommand{\Alg}{\mathcal{A}lg}
\newcommand{\CAlg}{\mathcal{C\!A}lg}

\newcommand{\bigdot}{\bullet}
\newcommand{\mediumdot}{{ \displaystyle \mathop{ \ \ }^{\hbox{$\centerdot$}}}}

\newcommand{\bord}{\nu}
\newcommand{\drob}{\overline{\nu}}

\newcommand{\bracket}[2]{\left\{#1,#2\right\}}
\newcommand{\double}[2]{\left\{\!\!\left\{#1,#2\right\}\!\!\right\}}
\newcommand{\triple}[3]{\left\{\!\!\left\{#1,#2,#3\right\}\!\!\right\}}

\begin{document}

\title[Quasi-Poisson structures on representation spaces of surfaces]{
Quasi-Poisson structures on \\ representation spaces of surfaces}

    \author[Gw\'ena\"el Massuyeau]{Gw\'ena\"el Massuyeau}
    \address{
    Gw\'ena\"el Massuyeau \newline
    \indent IRMA,    Universit\'e de Strasbourg \& CNRS \newline
    \indent 7 rue Ren\'e Descartes \newline
    \indent 67084 Strasbourg, France \newline
    \indent $\mathtt{massuyeau@math.unistra.fr}$}

    \author[Vladimir Turaev]{Vladimir Turaev}
    \address{
    Vladimir Turaev \newline
    \indent   Department of Mathematics \newline
    \indent  Indiana University \newline
    \indent Bloomington IN47405, USA\newline
    \indent $\mathtt{vturaev@yahoo.com}$}

\begin{abstract}
Given an oriented surface $\Sigma$ with base point $\ast\in \partial \Sigma$,  
we introduce for all $N\geq 1$, a  canonical quasi-Poisson
bracket on the space of $N $-dimensional linear representations of
$\pi_1(\Sigma, \ast)$.  Our bracket extends the well-known Poisson bracket on   $\GL_N$-invariant   
functions on this space.
 Our main tool is   a natural structure of a quasi-Poisson double algebra (in the sense of M. Van den Bergh) 
 on the group algebra of $\pi_1(\Sigma, \ast)$.
\end{abstract}

\maketitle

\section {Introduction}

 The   representation space    ${\mathcal H} =\Hom(\pi, G) $  consisting   of all homomorphisms from
the fundamental group   $\pi$   of a compact oriented surface
  to a Lie group $G$ is a rich source of   geometry. 
  The group $G$ acts on ${\mathcal H}$ by conjugations and the quotient ${\mathcal H}/G$  can be
identified with a moduli space of flat connections and with a moduli
space of holomorphic vector bundles (for appropriate $G$). For
closed surfaces,   the space ${\mathcal H}/G$ carries
 symplectic geometry. The classical instances are the
Weil--Petersson symplectic structure on the Teichm\"uller space (for
  $G=\operatorname{PSL}(2,\RR)$) and the  Atiyah--Bott symplectic
structure for compact $G$  endowed with a nondegenerate $ {\rm {Ad}}
(G)$-invariant   symmetric  bilinear form on the   corresponding
  Lie algebra. A systematic   approach to the symplectic
structure on ${\mathcal H}/G$ was introduced by W. Goldman
\cite{Go1, Go2} in extension of the work of S. Wolpert \cite{Wo}.
Goldman defined a Lie bracket in the free abelian group generated by
the set of conjugacy classes of elements of $\pi$ and used it to
compute the Poisson structure on ${\mathcal H}/G$ induced by the
symplectic structure.  Surfaces with boundary   have a
canonical Poisson structure on the quotient  ${\mathcal H}/G$.   It was described in
\cite{FR} in terms of ciliated fat graphs and in \cite{GHJW} in
terms of group systems, see also   \cite{AMM},  \cite{La}  and the surveys   \cite{Au},  \cite{Go3}, \cite{Hu}.

In this paper we show that for surfaces with boundary and
$G=\GL_N(\RR)$ with $N\geq 1$, there is a canonical   Poisson-type
structure on
  $\mathcal H$.   More precisely,
consider  a
 compact  oriented surface $\Sigma$  with non-void boundary. Set  $\pi=\pi_1(\Sigma, \ast)$ with $\ast \in
\partial \Sigma$  and  ${\mathcal H} =\Hom(\pi, G)$.
 Since   $\pi$ is a free group of a finite rank, $n$, a choice
of a basis of $\pi$ yields a bijection ${\mathcal H} \cong G^{n}$.
This induces a structure of a smooth manifold on   $\mathcal H$
independent of the choice of the basis. Moreover, the action of $G$
on $\mathcal H$ by conjugations is smooth.  A. Alekseev, Y.
Kosmann-Schwarzbach, and E. Meinrenken \cite{AKsM} introduced a
notion of a quasi-Poisson structure on a smooth manifold   endowed
with a smooth action of a Lie group. We construct here
  a canonical quasi-Poisson structure on
${\mathcal H}$.  One may think of this structure   as of a
skew-symmetric bracket $\{-,-\}$  in the algebra $C^\infty (\mathcal
H)$ of smooth $\RR$-valued functions on $\mathcal H$ satisfying the
Leibniz identity and the modified Jacobi identity
$$
\bracket{f} {\bracket{g }{h}}  + \bracket{g} {\bracket{h }{f}}
+   \bracket {h} {\bracket{f }{g}}   =   \phi(f,g,h)
$$
 for any  $f,g,h\in C^\infty (\mathcal H)$. 
 Here   $\phi \in \Lambda^3 \mathfrak{gl}_N(\RR)$  is the  Cartan trivector
and the action of the Lie algebra $\mathfrak{gl}_N(\RR)$ of $G=\GL_N(\RR)$ on $C^\infty (\mathcal
H)$ is induced
 by the action of $G$ on $\mathcal H$    by
conjugations. Both $G$ and the group $\Homeo(\Sigma, \ast)$ of
isotopy classes of orientation-preserving self-homeomorphisms of the
pair $(\Sigma, \ast)$ act on ${\mathcal H}$ by bracket-preserving
diffeomorphisms.  In particular, the bracket is preserved under the
Dehn twists about simple closed curves in $\Sigma$. 
More generally, an orientation-preserving basepoint-preserving
homeomorphism of surfaces induces a   diffeomorphism of the
corresponding representation spaces commuting with the action of $G$
and preserving the quasi-Poisson bracket.

The  usual Poisson structure on $\mathcal H/G$ is determined by our
quasi-Poisson structure on $\mathcal H$ as follows.  By a
smooth function  on $\mathcal H/G$ we mean  a $G$-invariant smooth
function on $\mathcal H$.  The  subalgebra $C^\infty (\mathcal H)^G
\subset C^\infty (\mathcal H)$ of smooth functions  on $\mathcal
H/G$  is closed under our quasi-Poisson bracket $\{-,-\}$ and the
restriction of $\{-,-\}$ to $C^\infty (\mathcal H)^G$ is a Poisson
bracket.   The  latter bracket  divided by two is    the
usual Poisson bracket in the algebra of smooth functions on
$\mathcal H/G$.

 Quasi-Poisson structures on $\mathcal H\cong G^n$ were first constructed in  \cite{AKsM}. 
The approach of \cite{AKsM} consists in  producing explicit quasi-Poisson structures on   $G$ 
and  on $G \times G$,  then combining copies of these structures 
into a quasi-Poisson structure on $G^n$ by a process called ``fusion'', 
and finally identifying $\mathcal H$ with $ G^n$ 
via a choice of a  basis of $\pi$.
This construction produces a family of quasi-Poisson structures on $\mathcal H $ numerated by certain bases of~$\pi$.
These   structures  a priori are not invariant under the action of  $\Homeo(\Sigma, \ast)$.
Nevertheless,   we show that   our quasi-Poisson structure on $\mathcal H $ 
coincides with that of \cite{AKsM} for a specific choice of a basis of $\pi$.

Our construction of the  quasi-Poisson  structure   on $\mathcal H$
proceeds in two steps.  First,  we introduce an abstract notion of a
quasi-Poisson algebra and show  how to turn the coordinate
algebra of ${\mathcal H}$ (in the sense of algebraic geometry) into
a  quasi-Poisson algebra.   Then  we extend the quasi-Poisson
bracket in the coordinate algebra of ${\mathcal H}$ to all smooth
functions on  $\mathcal H$.

The definitions and results introduced at the first step apply to
both compact and non-compact surfaces and hold  over any commutative
ring $\kk$ rather than~$\RR$.
  To work in this generality, we replace the
coordinate algebras as above with   more general algebras $A_N$
derived from the group ring $A=\kk\pi$ of $\pi$. The key new point
is a relationship between
  Fox pairings in $A$   and the theory of 
quasi-Poisson double
  brackets   due to M.\ Van den Bergh
\cite{VdB}.    Namely,  we show that the   Fox pairing in   $A $
defined in \cite{Tu} induces a quasi-Poisson double bracket in~$A$.
The Van den Bergh theory,   which  provides a version of
Poisson geometry for non-commutative algebras,   produces then a
quasi-Poisson structure on $A_N$.

The present work opens a number of further directions of study:
compute our quasi-Poisson bracket via  local geometry of
representation spaces;  
compute the bracket in terms of  fat graphs
 and compare it to the construction of Fock and Rosly \cite{FR}; 
extend our results   to other Lie groups  or algebraic groups; etc.  
In a sequel to the paper the authors will discuss a
high-dimensional generalization of the quasi-Poisson bracket
 in the context of string  topology, see  Remark \ref{Chas}.4. 

Our exposition  is essentially self-contained and   proceeds as
follows. We define quasi-Poisson algebras in
Section~\ref{sectionQUASIPOISSON}.   In
Section~\ref{sectionalgebrasmaintheorem} we discuss the algebras
$A_N$  and formulate our main   theorem.   The proof of this theorem
occupies Sections~\ref{VdB}--\ref{PMT}. We recall Van den Bergh's
theory of double brackets (Section~\ref{VdB}), discuss Fox pairings
in Hopf algebras (Section~\ref{sectionFOX}), and show how to derive
double brackets from Fox pairings (Section~\ref{FFPTDB}). Then we
recall the homotopy intersection pairing of \cite{Tu} and prove the
main theorem (Section~\ref{PMT}).  For compact surfaces, we derive
from this theorem   a natural quasi-Poisson structure on    the
representation manifold $\mathcal H$ and  explicitly compute this structure 
in  certain coordinates on $\mathcal H$ (Section~\ref{QPMRS}).
In Section~\ref{TBE+} we consider moment maps and discuss surfaces without boundary.  
In Section~\ref{GFG} we extend a part of our constructions to Fuchsian groups.  
 We conclude with two appendices.
In Appendix \ref{appendix_schemes} we discuss in more detail
certain group actions and Lie algebra  actions on $A_N$ introduced
in Section~\ref{sectionalgebrasmaintheorem}.  
In Appendix \ref{appendix_comparison} we compare 
the quasi-Poisson structures on $\mathcal H$ introduced in this paper and in \cite{AKsM}.

Throughout the paper we fix a commutative ring $\kk$. Unless
 stated otherwise, by a module we shall mean a $\kk$-module,  and by a linear map of modules we
mean  a $\kk$-linear map. By an algebra we shall mean an associative unital $\kk$-algebra.\\

 {\it Acknowledgements.} The work of G.\ Massuyeau was
partially supported by the French ANR research project
ANR-08-JCJC-0114-01.  G.M.\ would like to thank his
colleagues in Strasbourg for helpful conversations: P. Baumann, C.
Gasbarri, C. Kassel  and C. Noot--Huyghe. The work of V.\ Turaev was partially supported
by the NSF grant DMS-0904262.  
V.T. would like to thank A. Ramadoss  for   useful discussions.

\section{Quasi-Poisson algebras}\label{sectionQUASIPOISSON}

Inspired by the theory of quasi-Poisson manifolds   \cite{AKsM},
  we introduce in this section quasi-Poisson  algebras
generalizing the familiar Poisson algebras.

  \subsection{Poisson algebras}\label{QPA0}  Recall that a {\it derivation} of an algebra  ${\mathcal A}$
 is a  linear map $d:{\mathcal A}\to {\mathcal A}$ such that
   $d(ab)= d(a) b+ a d(b)$    for all $a,b \in  {\mathcal A}$.  A \emph{Poisson algebra} is an algebra ${\mathcal A}$ endowed with a skew-symmetric
  bilinear form $\{ -, -\}:{\mathcal A}\times {\mathcal A} \to {\mathcal A}$ which is a
derivation in each variable and satisfies the   Jacobi identity: for
all $a,b,c\in {\mathcal A}$,
$$
\bracket{a} {\bracket{b }{c}}  + \bracket{b} {\bracket{c }{a}} + \bracket {c} {\bracket{a }{b}} =   0 .
$$
 The form $\{ -, -\}$ satisfying these conditions is called a
\emph{Poisson bracket}. Note for the record that  for all $a,b ,c
\in {\mathcal A}$,
\begin{equation}\label{ssd}\bracket{b }{a}=-\bracket{a }{b}  \quad   {\rm
{and}} \quad \bracket{ab }{c}  = \bracket{a }{c} b+ a \bracket{b
}{c}.\end{equation}

  We emphasize that   we do not require the self-annihilating relation $\{a,a\}=0$ for $a\in {\mathcal A}$.
Likewise, speaking about Lie brackets we require skew-symmetry and
the Jacobi relation but not the self-annihilating relation. The
reader   uncomfortable with these conventions may assume from now on
that   $2 \in \kk$ is invertible so that the self-annihilating
relation follows from  the skew-symmetry.

    \subsection{Quasi-Poisson $\g$-algebras}\label{QPA2}  The derivations
   of an algebra ${\mathcal A}$ form a Lie algebra, $\Der ({\mathcal A})$, with Lie bracket $[d_1,
   d_2]=d_1d_2-d_2d_1$ for all $d_1, d_2\in \Der({\mathcal A})$.
   A   \emph{(left) action}
   of  a Lie algebra $\g$ on ${\mathcal A}$ is a Lie algebra homomorphism
   $\g\to \Der ({\mathcal A})$. Given such an action, we say that ${\mathcal A}$ is an
   \emph{algebra over $\g$} or, shorter, a \emph{$\g$-algebra}. An element $a\in {\mathcal A}$ is   \emph{$\g$-invariant} if $wa=0$
   for all $w\in \g$. The $\g$-invariant elements of ${\mathcal A}$ form a
   subalgebra of ${\mathcal A}$ denoted ${\mathcal A}^\g$.

An example of a $\g$-algebra is provided by the tensor algebra
$\displaystyle{\oplus_{n\geq 0} \g^{\otimes n}}$ where each $w\in \g$ acts by
\begin{equation}\label{action}
w(w_1\otimes \cdots \otimes w_n)=\sum_{i=1}^n w_1 \otimes \cdots
\otimes w_{i-1} \otimes [w, w_i] \otimes w_{i+1} \otimes\cdots
\otimes w_n
\end{equation}
for any $w_1,\ldots, w_n\in \g$. A vector $a\in \g^{\otimes n}$ is
\emph{skew-symmetric} if any transposition of tensor factors carries
$a$ into $-a$. The  module of  $\g$-invariant skew-symmetric vectors
of $\g^{\otimes n}$ is denoted by $\wedge^n_{\g}$. We shall be
specifically interested in this module for $n=3$.  In
particular, a nonsingular $\g$-invariant symmetric bilinear form
$\cdot:\g \times \g \to \kk$ defines a vector $\phi \in
\wedge^3_{\g}$ which is the skew-symmetric trilinear form
\begin{equation}\label{Cartan}
 \g \otimes \g \otimes \g \longrightarrow \kk, \ (w_1,w_2,w_3) \longmapsto w_1 \cdot [w_2,w_3]
\end{equation}
 viewed as a skew-symmetric element of $(\g^*)^{\otimes
3}\simeq \g^{\otimes 3}$ through the isomorphism $\g \to \g^*, w
\mapsto w\cdot(-)$. We call $\phi$  the \emph{Cartan trivector}
associated to the form $\cdot$.

Fix a Lie algebra $ \g $ and a    vector $\phi\in \wedge^3_{\g}$. A
$ \g$-algebra ${\mathcal A}$ is \emph{quasi-Poisson} if ${\mathcal
A}$ is endowed with a skew-symmetric  bilinear form $\{ -,
-\}:{\mathcal A}\times {\mathcal A} \to {\mathcal A}$ which is a
derivation in each variable and satisfies the following two
identities:
\begin{equation}\label{ssdM--}
w \bracket{a} {b}  = \bracket{wa }{b}  +  \bracket{a}{wb}
\end{equation}
for all $w\in \g$, $a,b\in {\mathcal A}$ and
\begin{equation}\label{ssdM}
\bracket{a} {\bracket{b }{c}}  + \bracket{b} {\bracket{c }{a}} + \bracket {c} {\bracket{a }{b}} =   \phi(a, b, c)
\end{equation}
for all $a,b,c\in {\mathcal A}$. Here    $\phi(a, b, c)\in {\mathcal
A} $ is the image of $\phi\in \wedge^3_\g \subset \g^{\otimes 3}$
under the linear map $\g^{\otimes 3} \to {\mathcal A} $ carrying
$w_1\otimes w_2 \otimes w_3$ to $(w_1 a) (w_2 b) (w_3c)$ for all
$w_1, w_2 , w_3 \in \g$.

   We call the form $\{ -, -\}$ a \emph{quasi-Poisson bracket  in
${\mathcal A}$ associated with $\phi$}. Clearly, $\{{\mathcal A}^\g,
{\mathcal A}^\g\}\subset {\mathcal A}^\g$ and the restriction of $\{
-, -\}$ to ${\mathcal A}^\g$ is a Poisson bracket in ${\mathcal
A}^\g$. Thus, a quasi-Poisson structure in ${\mathcal A}$ induces a
Poisson structure in ${\mathcal A}^\g$. For  $\g=\{0\}$,   a
quasi-Poisson structure in ${\mathcal A}$ is just a
  Poisson structure in ${\mathcal A}$.

 \subsection{Quasi-Poisson  algebras over Lie pairs}\label{QPA3} The
 notion of a  quasi-Poisson algebra over a Lie algebra may be
 generalized bringing a group into the picture. To do it, we
 introduce the following notion:  a
 \emph{Lie pair} is  a pair $(G, \g)$   where $G$ is a group   and
 $\g$ is a  Lie algebra  endowed with a (left) action of $G$ on $\g$ by Lie algebra
 automorphisms. The  action  is denoted by  $ w\mapsto  {}^g\! w$ for $w\in \g$ and $g\in G$.

 For example, a pair consisting of the trivial group $G=\{1\}$ and  any Lie algebra is  a Lie pair.
 The pair consisting of any group and the trivial Lie algebra   $\g=\{0\}$   is a Lie pair.
For $\kk=\RR$,   examples of Lie pairs are provided by pairs
(a Lie group $G$, the Lie algebra  $\g$ of $G$ with the (left) adjoint action of $G$).
Further examples of Lie pairs are given in Section~\ref{QPA5}.

 An \emph{algebra over a Lie pair} $(G,\g)$ or, shorter, a
\emph{$(G,\g)$-algebra} is a $\g$-algebra ${\mathcal A}$ endowed
with a (left) action
   of $G$ by algebra automorphisms
  such that
  $ {}^g \! w(a)=g\, w (g^{-1}  a)  $ for all $w\in \g$, $g\in G$, $a\in {\mathcal A}$.
An element $a\in  {\mathcal A}$ is \emph{$G$-invariant} if $ga=a$
for all $g\in G$. The  $G$-invariant elements of ${\mathcal A}$ form a
 subalgebra of ${\mathcal A}$ denoted ${\mathcal A}^G$.
 Note that ${\mathcal A}^\g$ is stable under the action of $G$;
  generally speaking $ {\mathcal A}^\g\neq {\mathcal A}^G$.

  For example, the tensor
algebra $\oplus_{n\geq 0} \g^{\otimes n}$ is a $(G,\g)$-algebra
where $\g$ acts by \eqref{action} and each $g\in G$ acts  by
$$g(w_1\otimes w_2\otimes \cdots \otimes w_n)= {}^g\!w_1\otimes {}^g\!w_2\otimes \cdots \otimes {}^g\!w_n  $$ for any   $w_1,\ldots, w_n\in
\g$.

  An algebra ${\mathcal A}$ over a Lie pair $(G,\g)$  is \emph{quasi-Poisson} if ${\mathcal A}$,
considered as a $\g$-algebra, has a quasi-Poisson bracket $\{ -,
-\}$  associated with a $G$-invariant trivector $\phi\in
\wedge^3_\g$ and for all $g\in G$,  $a,b\in {\mathcal A}$,
\begin{equation}\label{ssdM++} g
\bracket{a} {b}  = \bracket{ga }{gb}.
\end{equation}
 For $G=\{1\}$, we recover the notion of a quasi-Poisson algebra over  $\g$.
For  $\g=\{0\}$,  a quasi-Poisson algebra over $(G,\g)$ is nothing
but a Poisson algebra endowed with an action of $G$   by Poisson
algebra automorphisms.

 \subsection{The Lie pair $(G_N,\g_N)$}\label{QPA5}
 We will focus   in the sequel  on the  quasi-Poisson algebras over the
 Lie pair $(G_N,\g_N)$ where $N\geq 1$ is an integer,
 $G_N=\GL_N(\kk)$ is the $N$-th general linear group, and $\g_N=\mathfrak{gl}_N(\kk)$
 is the Lie algebra  of    $N\times N$-matrices with entries in
 $\kk$.
The Lie bracket in $\g_N$ is given by $[u,v]=uv-vu$ and    $G_N$
acts on $\g_N$  by ${}^g\! v=gvg^{-1}$. For $i,j\in \{1, \ldots,
N\}$ let $f_{ij}\in \g_N$ be the elementary matrix whose $(i,j)$-th
entry is $1$ and all other entries are equal to zero.  We consider
the tensor
\begin{equation}\label{phiN}
\phi_N=-f_{ij}\otimes f_{jk}\otimes f_{ki} + f_{jk}\otimes f_{ij}\otimes f_{ki} \ \in \g_N^{\otimes 3}.
\end{equation}
 Here  and below we sum up over all repeating indices.  The
tensor $\phi_N$ is skew-symmetric,  $G_N$-invariant, and $\g_N$-invariant.
  Indeed, $\phi_N$ is the Cartan trivector \eqref{Cartan}
determined by the trace pairing $v\cdot w = \tr(vw)$ in   $\g_N$.
 Unless explicitly stated to the contrary, by a  \emph{quasi-Poisson bracket} in a $(G_N,\g_N)$-algebra
 we will mean a quasi-Poisson bracket associated with $\phi_N$.

\section{The algebra  $A_N$ and the main theorem}\label{sectionalgebrasmaintheorem}

In this section, we derive from an arbitrary algebra $A $ a sequence
of commutative algebras $(A_N)_{N\geq 1}$.  Then we apply this
construction to  the group algebras  of the fundamental groups   of
  surfaces and state our main theorem.

\subsection{The algebra  $A_N$}\label{AMT0} Let $A$ be  an algebra.
Following  \cite{Cb}, \cite{VdB},  we define   a sequence of
commutative algebras $A_1,A_2,\ldots$  For  $N\geq 1$, the
 algebra $A_N$ is  generated by
the symbols $a_{ij}$ where $a$ runs over $A$ and $i,j $ run over the
set $ \{1,2, \ldots, N\}$. These generators commute with each other
and satisfy the following relations: $1_{ij}=\delta_{ij}$ where
$\delta_{ij}$ is the Kronecker delta; for all $a,b\in A$, $k\in
\kk$, and $i,j\in \{1,2, \ldots, N\}$,
$$(ka)_{ij}=k a_{ij}, \quad  (a + b)_{ij}=a_{ij}+ b_{ij}, \quad {\rm {and}} \quad (ab)_{ij}=  a_{il} b_{lj}. $$
(In accordance with our conventions, in the third relation,  we sum up over the repeating index $l$.)
The construction of $A_N$ is functorial: any  algebra homomorphism
$f:A\to A'$ induces an algebra homomorphism
$f_N:A_N\to A'_N$ by $f_N(a_{ij})= (f(a))_{ij}$ for all $a\in A$ and $i,j \in  \{1,  \ldots, N\}$.

 The definition of $A_N$ is designed so  that the functor
 ${\Alg}\to {\CAlg}, A\mapsto A_N$ is left adjoint to the functor ${\CAlg} \to {\Alg}, B\mapsto \Mat_N(B)$.
 Here $ {\Alg}$ and   ${\CAlg}$   are  the categories of algebras and commutative algebras, respectively,
 and $\Mat_N(B)$ is the algebra of $N\times N$-matrices over $B$.
 In other words,   for any commutative algebra~$B$, there is  a canonical bijection
 \begin{equation}\label{adjunction}
 \Hom_{\CAlg} (A_N, B)  \simeq  \Hom_{\Alg} (A, \Mat_N(B))
 \end{equation}
 which is natural in $A$ and $B$.
  The bijection \eqref{adjunction}  carries any $r:A_N\to B$
to the homomorphism $ A \to \Mat_N(B) $ sending any $a \in A$ to
 the $N\times N$-matrix $(r(a_{ij}))_{i,j}$.
  The inverse bijection carries any $s:A\to \Mat_N(B)$ to the
 homomorphism $ A_N\to B$ sending a generator $a_{ij}$ to
 the $(i,j)$-th term of the matrix $s(a)$ for all $a\in A$.
The existence of a natural isomorphism $\Hom_{\CAlg} (A_N, -) \simeq
\Hom_{\Alg} (A, \Mat_N(-))$ can be rephrased in the language of algebraic geometry (see
\cite[pp.\ 17--18]{Do} or \cite[\S I.1.3]{Ja}):
 there is an affine scheme (over $\kk$) whose set of $B$-points  is $\Hom_{\Alg} (A, \Mat_N(B))$
 for any commutative algebra $B$, and $A_N$ is the coordinate algebra of that scheme.

 The linear map $\tr=\tr_N:A\to A_N$ carrying any $a\in A$ to $
\tr(a)=\sum_{i=1}^N a_{ii}\in A_N$ is   called the \emph{trace}. It
is easy to check that $\tr([A,A])=0$ where
  $[A,A]$ is the  submodule of $A$ spanned by the set
$\{ab-ba\}_{ a,b\in A}$. Hence the trace induces a linear map $
A/[A,A] \to A_N$ also denoted by $\tr$.  For any  algebra
homomorphism $f:A\to A'$, the   homomorphism $f_N:A_N \to A'_N$
satisfies $ f_N \tr = \tr f:A\to A'_N$.

The algebra $A_1$ corresponding to $N=1$ is the maximal commutative
algebra obtained as the quotient of $A$. More precisely,  the
homomorphism $\tr:A\to A_1$ is surjective and its kernel is the
two-sided ideal of $A$ generated by $[A,A]$.

\subsection{ Actions on   $A_N$}\label{AMT0++}

For all $N\geq 1$, we   turn the algebra $A_N$ of Section \ref{AMT0}
into an algebra over the Lie pair $(G_N, \g_N) =(\GL_N(\kk),
\mathfrak{gl}_N(\kk))$ introduced in Section~\ref{QPA5}.   The
action of $G_N$ on   $A_N$ by algebra automorphisms is defined
by
\begin{equation}\label{GN_on_ANold}
g a_{ij}=
(g^{-1})_{i,k}\, g_{l,j}\,  a_{kl}
\end{equation}
for all $g=(g_{k,l})_{k,l=1}^N\in G_N$, $a\in A$ and $i,j\in \{1,2,
\ldots, N\}$. Though one usually writes numerical coefficients to
the left of the variables, as   on the right-hand side of
\eqref{GN_on_ANold}, it is easier to remember this formula in the
following equivalent form:
\begin{equation}\label{GN_on_AN}
g a_{ij}= (g^{-1})_{i,k}\,
 a_{kl} \, g_{l,j}.
\end{equation}
In the sequel we rather write our formulas in the latter form.
The Lie algebra $\g_N$ acts  by derivations  on the algebra $A_N$ by
\begin{equation}\label{gN_on_AN}
wa_{ij}= a_{is} w_{s,j}-
w_{i,s} a_{sj}
\end{equation}
for all   $w=(w_{k,l})_{k,l=1}^N\in \g_N$, $a\in A$ and $i,j\in \{1,2, \ldots, N\}$.
  In terms of the elementary matrices $f_{kl}\in \g_N$ we have
$f_{kl}a_{ij}=\delta_{lj} a_{ik}-\delta_{ik} a_{lj}$ for all
$i,j,k,l$. Direct computations show that formulas \eqref{GN_on_AN}
and \eqref{gN_on_AN} are compatible with the defining relations of
$A_N$ and turn  $A_N$ into a $(G_N,\g_N)$-algebra.
It is clear that for   any algebra
homomorphism $ A\to A'$, the induced algebra homomorphism $ A_N \to
A'_N$ is $(G_N,\g_N)$-equivariant.

The action of $G_N$ on $A_N$ has the following origin. Given a
commutative algebra $B$, the group $G_N$ acts on $\Mat_N(B)$   by
$M\mapsto gMg^{-1}$ for $M\in \Mat_N(B)$ and $g\in G_N$. This
induces a (left) action of $G_N$ on $ \Hom_{\Alg} (A, \Mat_N(B))$
and, via   \eqref{adjunction},   a (left) action of $G_N$ on
$\Hom_{\CAlg} (A_N, B)$. The latter action is natural in $B$ and
therefore, by the  Yoneda  lemma,  it induces a (left) action of $G_N$ on $A_N$ so that
\begin{equation}\label{unique_action}
r(gx)=( g^{-1} r)(x)  \quad
\hbox{for   any $r\in \Hom_{\CAlg} (A_N, B)$,  $g\in G_N$, $x\in A_N$}.
\end{equation}
For $B=A_N$, $r=\id: {A_N}\to B$, and $x=a_{ij}$ with $a\in A$ and
$i,j\in \{1,\dots,N\}$, we obtain $g a_{ij}   = (g^{-1} \id)(a_{ij})
=  (g^{-1})_{i,k} a_{kl} g_{l,j}$ as   in our  definition
\eqref{GN_on_AN} of the action of $G_N$ on $A_N$.  
The actions of $G_N$ and $\g_N$   on  $A_N$   are further studied in Appendix~\ref{appendix_schemes}.

The algebra $A_N$ has three useful subalgebras: $A_N^{\g_N}$,
$A_N^{G_N}$, and the  algebra $A^t_N$ generated   by $\tr(A) \subset
A_N$.  It is easy to check that $A^t_N\subset A_N^{\g_N}\cap
A_N^{G_N}$. When $\kk $ is a field of characteristic zero   and the
algebra $A$ is finitely generated, we have  $A^t_N=A_N^{G_N}$ 
and therefore $A_N^{G_N}\subset A_N^{\g_N}$; see  \cite{LbP}, \cite{Cb}.

\subsection{The  case of a group algebra}\label{AMT--1} Given a group $\pi$, we can
 apply the constructions   above   to the group
algebra   $A=\kk \pi$.   This gives   for each $N\geq 1$  a
commutative $(G_N,\g_N)$-algebra $A_N=(\kk \pi)_N$   generated by
the commuting symbols $a_{ij}$ where $a\in \pi$ and $i,j \in \{1,
\ldots, N\}$ subject to the relations $1_{ij} = \delta_{ij}$ and
$(ab)_{ij}=  a_{il} b_{lj}$ for all $a,b\in \pi$ and $i,j \in \{1,
\ldots, N\}$. Note that $A/[A,A]$ is the free  
$\kk$-module with basis $\check \pi$ where $\check \pi$ is the set
of the conjugacy classes of elements of~$\pi$.  For all $N\geq 1$,
we have the linear  map $\tr:\kk\check \pi \to A_N$.
The discussion at the end of Section~\ref{AMT0} shows that  $A_1$ is
the group algebra of $H_1(\pi)=\pi/[\pi, \pi]$,  and the map
$\tr:\kk\check \pi \to A_1$ is   the linear extension of the obvious
projection $\check \pi\to H_1(\pi)$.

For a  commutative algebra $B$,   consider the group $\GL_N(B)$ of
invertible   $N\times N$-matrices   over $B$ and set $\mathcal
H_B=\Hom  (\pi, \GL_N(B))$. Restricting algebra homomorphisms $A\to
\Mat_N(B)$ to $\pi\subset A$, we obtain a  bijection
$$
\Hom_{\Alg} (A, \Mat_N(B)) \simeq \Hom  (\pi, \GL_N(B)) =\mathcal H_B
$$
  and, composing with \eqref{adjunction},   we obtain a   bijection
\begin{equation}\label{secondident}
\mathcal H_B  \simeq     \Hom_{\CAlg} (A_N, B).
\end{equation}
 Thus, there is an   affine scheme (over $\kk$)
whose set of $B$-points  is $\mathcal H_B$ for any commutative
algebra $B$, and  $A_N$ is the coordinate algebra of that scheme.

Note that every   $x\in A_N$  determines a function $  \mathcal
H_B\to B $ which corresponds under \eqref{secondident} to evaluation
at   $x$.   This defines the  {\it evaluation} map
\begin{equation}\label{evaluation}
\ev_B: A_N \longrightarrow \Map(\mathcal H_B,B),
\end{equation}
which is an algebra homomorphism from $A_N$ to the algebra  $\Map(\mathcal H_B,B)$ of $B$-valued functions on $\mathcal
H_B$ with point-wise addition and multiplication. For any $a\in
\pi$, $i,j\in \{1,\dots, N\}$ and $s \in \mathcal H_B$, we have
\begin{equation}\label{evaluation_formula}
\ev_B(a_{ij})(s) = \hbox{$(i,j)$-th term of the matrix $s(a)$}.
\end{equation}
The action of $G_N$ on $\GL_N(B)$ by conjugations induces an action
of $G_N$   on $\mathcal H_B$; the latter induces an action of $G_N$
on $\Map(\mathcal H_B,B)$ so that $(gf)(h)=f(g^{-1}h)$ for any $g
\in G_N$, $f\in \Map(\mathcal H_B,B)$ and $h\in \mathcal H_B$.  The
map $\ev_B$ is $G_N$-equivariant:  $\ev_B(gx)= g\ev_B(x)$ for
any $x\in A_N$ and $g\in G_N$.   Indeed, for  any $s\in \mathcal
H_B$ corresponding through \eqref{secondident} to $r\in \Hom_{\CAlg}
(A_N, B)$, we have
$$
\ev_B(gx)(s) = r(gx) \stackrel{\eqref{unique_action} }{=} (g^{-1}r)(x)
= \ev_B(x)(g^{-1}s)=\big(g \ev_B(x)\big) (s).
$$

 Consider now the case where $\pi$ is a   free group
with basis   $\{x_u\}_{u\in U}$ indexed by a (possibly infinite) set $U$.  
It follows from the definitions that the algebra
$A_N=(\kk \pi)_N$ is generated by the commuting symbols
$x^u_{ij}=(x_u)_{ij}$ and ${\overline x}^u_{ij}=(x_u^{-1})_{ij}$
with $u\in U$, $i,j \in \{1, \ldots, N\}$ subject only to the
relation   $  x^u_{il} \, {\overline x}^u_{ lj}=\delta_{ij}$ for all
$u, i,j  $.  These relations may be expressed by saying that the
$N\times N$-matrices   $x^u= (x^u_{ij})_{i,j}$   and  ${\overline
x}^u= ({\overline x}^u_{ij})_{i,j}$  are mutually inverse for all
$u$. The same algebra is generated by the commuting symbols  $y^u$
and $x^u_{ij}$ with $u\in U$ and $i,j\in\{1,\dots, N\}$, subject
to the relation   $ y^u \det (x^u )  =1$ for all $u $. If $U$ is
finite and $n=\card(U)$, then $A_N$ is the tensor product of $n$
copies of the   algebra   generated by the commuting symbols
 $y$ and $x_{ij}$ with $i,j\in\{1,\dots, N\}$  subject   to the single relation   
 $y\, \det\left((x_{ij} )_{i,j}\right)  =1$. 
This is compatible with the standard computation of the coordinate
algebra of the  affine scheme obtained  as the direct
product of $n$ copies of $\GL_N$ (see, for example, \cite[Example
3.5.6]{Do} or \cite[\S I.2.2]{Ja}).

\subsection{Main theorem}\label{AMT1}

We   state   our  main result.
 The   definition of Goldman's Lie bracket used in this
statement will be recalled in Section~\ref{endofmain}.

 \begin{theor}\label{CQPS}     Let  $\Sigma$ be an oriented surface
 with   base point $\ast\in \partial \Sigma$. Let $\pi=\pi_1(\Sigma,\ast)$
 and $A=\kk\pi$.
 Then, for all $N\geq 1$, we have the following:

 - the $(G_N,\g_N)$-algebra $A_N$
 admits a canonical quasi-Poisson bracket $\bracket{-}{-}$ associated with the trivector~$\phi_N$;

 -  the trace map   $\tr: \kk\check \pi  \to  A_N^{\g_N}$
 is a homomorphism of Lie algebras where $\kk\check
 \pi$ is endowed with $2 \times { {(the \,\, Goldman \,\, Lie\,\,  bracket)}}$  and $A_N^{\g_N}$ is
 endowed with the restriction of the bracket $\bracket{-}{-}$.
\end{theor}

The quasi-Poisson bracket in $A_N$ produced by this theorem is
defined by an explicit formula and does not depend on any
provisional choices (see Section~\ref{PMT}). This bracket is natural
in the sense that it is invariant under homeomorphisms of surfaces
preserving orientation and the base point.

The restriction of our quasi-Poisson bracket in $A_N$ to
$A_N^{\g_N}$ is a Poisson bracket turning $A_N^{\g_N}$ in a Poisson
algebra. The second claim of Theorem~\ref{CQPS} implies that the
subalgebra $A^t_N$ of $A_N^{\g_N}$ generated   by $\tr(\check \pi)
$ is a Poisson subalgebra. We may summarize the situation by saying
that Goldman's bracket multiplied by 2 induces a    Poisson bracket
in $A_N^{t}$, and the latter extends canonically to a natural
quasi-Poisson bracket in $A_N$. If $2$ is invertible in $\kk$, then
dividing the latter bracket by $2$ we obtain   a quasi-Poisson
bracket in $A_N$ which is associated with $\phi_N/4$ and extends the
Poisson bracket in $A_N^{t}$ induced by Goldman's bracket. If $ \kk
$ is a field of characteristic zero and $\Sigma$ is compact,
then $A^t_N=A^{G_N}_N$ and we obtain a Poisson bracket in $A^{G_N}_N$.

For $N=1$ our quasi-Poisson bracket  in $A_1$, i.e.\ in the group
algebra of $H_1(\pi)$,   may be computed explicitly by
$\bracket{a}{b}= 2 (a \mediumdot b)\, ab$ for all  $a,b\in
H_1(\pi)=H_1(\Sigma)$, where $\mediumdot$ denotes  the homological
intersection form  of $\Sigma$ and we use multiplicative notation
for the group operation in $H_1(\Sigma)$. This bracket is in fact a
Poisson bracket, which is compatible with the equality  $
A_1^{\g_1}= A_1$.

  The proof of Theorem~\ref{CQPS}   starts with a bilinear form
$\eta$ in $A=\kk\pi$ defined in \cite{Tu}. This form is a Fox
pairing in the sense of \cite{MT}. A simple normalization turns
$\eta$ into a skew-symmetric Fox pairing $\eta^s$ in $A$. We show
that every skew-symmetric Fox pairing in $A$ induces a double
bracket in $A$ in the sense of Van den Bergh \cite{VdB}. The double
bracket induced by $\eta^s$ turns out to be quasi-Poisson. Then a
construction of Van den Bergh yields a quasi-Poisson bracket in
$A_N$.  We introduce all these tools in
Sections~\ref{VdB}--\ref{FFPTDB}  and   finish the proof in
Section~\ref{PMT}.

\section{Double and triple brackets} \label{VdB}

We  outline the theory of ``multiple brackets''   due to Van den
Bergh~\cite{VdB}. Throughout this section, $A$ is an  arbitrary
algebra.

\subsection{Preliminaries}\label{pre}

 For $n\geq 2$, we    will often write any element $x$ of $A^{\otimes n}$ as
$x^{(1)}\otimes  \cdots \otimes x^{(n)}$ and drop the summation sign.
  Given a permutation $(i_1,\dots,i_n)$ of $(1,\dots,n)$, 
we denote by $P_{i_1\cdots i_n}$ the  linear  map $A^{\otimes n}\to A^{\otimes n}$  
carrying any $x\in A^{\otimes n}$ to $ x^{(i_1)} \otimes x^{(i_2)} \otimes \cdots \otimes x^{(i_n)}$.  
Unless explicitly stated otherwise,   we endow   $A^{\otimes n}$ with the ``outer''
$A$-bimodule structure defined by
\begin{equation}\label{outer}
axb= ax^{(1)}\otimes  x^{(2)}\otimes \cdots \otimes  x^{(n-1)} \otimes x^{(n)} b \quad
\hbox{for   $a,b \in A$ and $x\in A^{\otimes n}$.}
\end{equation}
A linear map $D:A\to A^{\otimes n}$ is a \emph{derivation} if
$D(ab)=a D(b)+ D(a) b$ for all $a,b\in A$.

\subsection{Double   brackets}\label{db}
A \emph{double bracket} on the algebra $A$ is a  linear map $\double{-}{-}:
A^{\otimes 2} \to A^{\otimes 2}$ which is a derivation in the second
variable and  is skew-symmetric. This means that for all $a,b,c\in A$
$$
\double{a}{bc} = b \double{a}{c} + \double{a}{b}c
\quad \hbox{and} \quad \double{b}{a}=-P_{21}\double{a}{b}
$$
These properties imply that $\double{ab}{c}=
a*\double{b}{c}+\double{a}{c}*b$ for all $a,b,c\in A$  where $\ast$
is the ``inner'' $A$-bimodule structure on $A^{\otimes 2}$   defined
by
\begin{equation}\label{inner}
l*(a_1 \otimes a_2)*r= a_1 r \otimes l a_2 \quad
\hbox{for   $l,a_1,a_2,r \in A$.}
\end{equation}

We now relate the double brackets to  the algebra $A_N$ defined in Section~\ref{AMT0}.

\begin{lemma}\label{fdbtopb}  \cite[Proposition 7.5.1]{VdB}
Given a double   bracket $\double{-}{-}$ in   $A$ and an integer
$N\geq 1$, there is a unique   bilinear form $\{ -, -\}:A_N\times
A_N \to A_N$ which is a derivation in each variable and satisfies
\begin{equation}\label{db_to_qpb}
 \bracket{a_{ij}}{b_{uv}}=  \double{a}{b }_{uj}^{(1)}  \double{a }{b }_{iv}^{(2)}
\end{equation}
for all  $a,b\in A$ and $i,j,u,v\in \{1,\ldots, N\}$. This form $\{
-, -\}$ is skew-symmetric and satisfies the identities
 \eqref{ssdM--}, \eqref{ssdM++}.
\end{lemma}

\begin{proof}  We extend \eqref{db_to_qpb} to a bilinear form
$\{ -, -\}:A_N\times A_N \to A_N$ which is a derivation in each
variable. To see that this form is well-defined, we need   to verify
the compatibility with the defining relations of $A_N$. That the
right-hand side of \eqref{db_to_qpb} is  linear  in both $a$ and $b$
follows from the bilinearity of $\double{-}{- }$.
The compatibility with the relation $1_{ij} =\delta_{ij}$ is a consequence
of the fact that $\double{1}{b}= 0 =\double{a}{1}$ for any $a,b\in A$.
To verify the
compatibility with the relation of type $(bc)_{uv}=  b_{ul} c_{lv}$, let us
expand $ \double{ a}{b}=x^{(1)}\otimes x^{(2)}$ and $ \double{
a}{c}=y^{(1)}\otimes y^{(2)}$. Then $$\double {a} {bc}=b\double {a}
{c} + \double {a} {b} c = b y^{(1)}\otimes y^{(2)} + x^{(1)}\otimes
x^{(2)}c.$$ Therefore
\begin{eqnarray*}
 \bracket{a_{ij}}{(bc)_{uv}} &= &{\double{a} {bc}}^{(1)}_{uj} {\double{a} {bc}}^{(2)}_{iv}\\
 &=&(b y^{(1)})_{uj} y^{(2)}_{iv}+ x^{(1)}_{uj} (x^{(2)}c)_{iv} \\
&=& b_{ul} y^{(1)}_{lj} y^{(2)}_{iv}+ x^{(1)}_{uj} x^{(2)}_{il} c_{lv}\\
&=& b_{ul} \bracket{a_{ij}}{ c_{lv}}+\bracket{a_{ij}}{b_{ul}} c_{lv}
\ = \ \bracket{a_{ij}}{b_{ul}c_{lv}}.
\end{eqnarray*}
A similar computation gives
$\bracket{(ab)_{ij}}{c_{uv}}= \bracket{a_{il} b_{lj}}{ c_{uv}}$.

The skew-symmetry of $\bracket{-}{-}$ follows from the skew-symmetry
of $\double{-}{-}$:
$$\bracket{b_{uv}}{a_{ij}}= \double{b}{a}^{(1)}_{iv}
\double{b}{a}^{(2)}_{uj}=- \double{a}{b}^{(2)}_{iv}
\double{a}{b}^{(1)}_{uj}=- \bracket {a_{ij}} {b_{uv}}.$$

Pick $w=(w_{k,l})_{k,l}\in \g_N $. It is easy to see that if
 \eqref{ssdM--}  holds for       the generators of $A_N$, then it holds for all elements of $A_N$.
We check \eqref{ssdM--} for   the generators:
\begin{eqnarray*}
 w \bracket{   a_{ij}}{  b_{uv} }
&=& w (\double{a}{b }_{uj}^{(1)}  \double{a }{b }_{iv}^{(2)})\\
&=& w (\double{a}{b }_{uj}^{(1)} ) \double{a }{b }_{iv}^{(2)}
+ \double{a}{b }_{uj}^{(1)} w( \double{a }{b }_{iv}^{(2)}) \\
&=& \double{a}{b }_{uk}^{(1)}  w_{k,j}  \double{a }{b }_{iv}^{(2)}
- w_{u,k}  \double{a}{b }_{kj}^{(1)}  \double{a }{b }_{iv}^{(2)} \\
&& +  \double{a}{b }_{uj}^{(1)}  \double{a }{b }_{ik}^{(2)}  w_{k,v}
-  \double{a}{b }_{uj}^{(1)}  w_{i,k}  \double{a }{b }_{kv}^{(2)} \\
&=&  \bracket{  a_{ik} w_{k,j} - w_{i,k} a_{kj}}{  b_{uv} } +
\bracket{   a_{ij}}{ b_{uk} w_{k,v} -w_{u,k} b_{kv} }\\
&=&  \bracket{  w a_{ij}}{  b_{uv} } +  \bracket{   a_{ij}}{ w b_{uv} }.
 \end{eqnarray*}

Similarly, it is enough to check \eqref{ssdM++}  for any
$g=(g_{k,l})\in G_N$ and  the generators of $A_N$. We have
\begin{eqnarray*}
\bracket{ g a_{ij}}{g b_{uv} }
&=& \bracket{(g^{-1})_{i,k}  a_{kl} g_{l,j}}{(g^{-1})_{u,s} b_{st}  g_{t,v}}\\
&= & (g^{-1})_{i,k}  g_{l,j} (g^{-1})_{u,s} g_{t,v} \bracket{ a_{kl} }{ b_{st} }\\
&=& (g^{-1})_{i,k}  g_{l,j} (g^{-1})_{u,s} g_{t,v}  \double{a}{b}^{(1)}_{sl} \double{a}{b}^{(2)}_{kt}\\
&=& (g^{-1})_{u,s}  \double{a}{b}^{(1)}_{sl} g_{l,j}  \,   (g^{-1})_{i,k} \double{a}{b}^{(2)}_{kt}  g_{t,v} \\
&=& (g \double{a}{b}^{(1)}_{uj}) (g \double{a}{b}^{(2)}_{iv})\\
& = & g (\double{a}{b}^{(1)}_{uj}\double{a}{b}^{(2)}_{iv}) \ = \ g \bracket{   a_{ij}}{  b_{uv} }.
\end{eqnarray*}

\vspace{-0.5cm}
\end{proof}

\subsection{A bracket $\langle- ,- \rangle$}\label{stdb} A  double bracket $\double{- }{-}$ on $A$
induces further pairings  which we now discuss.  For $a, b\in A$,
set
\begin{equation}\label{otherbracket}\langle a,b \rangle= {\double{a }{b}}^{(1)}  {\double{a }{b}}^{(2)}\in A.\end{equation} We
state several properties of  $\langle -, - \rangle:A\times A \to A$
following \cite[Section 2.4]{VdB}:

(i)  the composition of $\langle -, - \rangle$ with the projection
to  $ A/[A,A]$ is skew-symmetric;

(ii)  $\langle [A,A], A \rangle=0$;

(iii)  the pairing $\langle -, - \rangle $ is a derivation in the
second variable.

Set $\check{A}=A/[A,A]$. The properties (i) and (ii) imply that
$\langle -, - \rangle$ induces    bilinear pairings $\check
A\times A\to A$    and   $\check A\times \check A\to \check
A$. The latter pairing is  skew-symmetric. Both pairings are denoted
by $\langle - , - \rangle$.

\begin{lemma}\label{comparis} For all $N\geq 1$, the trace map $\tr:\check A\to A_N$
carries the bracket $\langle- , - \rangle$ on $\check A$ to the
bracket $\bracket{-}{-}$ on $A_N$ defined by Lemma \ref{fdbtopb}.
\end{lemma}

\begin{proof} Let $a,b\in A$ and let $\check a, \check b$ be their projections to $\check A$.  We have
\begin{eqnarray*}
\bracket{\tr(\check a)}{\tr(\check b)} &=& \sum_{i,j} \bracket{a_{ii}}{b_{jj}}\\
& \stackrel{\eqref{db_to_qpb}}{=}&  \sum_{i,j} \double{a}{b}^{(1)}_{ji} \double{a}{b}^{(2)}_{ij}\\
&=&  \sum_{j} \left(\double{a}{b}^{(1)} \double{a}{b}^{(2)}\right)_{jj}\\
&=&  \tr \left(\double{a}{b}^{(1)} \double{a}{b}^{(2)}\right)
\ \stackrel{\eqref{otherbracket}}{=} \ \tr\big(\langle   a,  b \rangle\big) \ = \ \tr\big(\langle \check a, \check b \rangle\big).
\end{eqnarray*}

\vspace{-0.5cm}
\end{proof}

\subsection{Triple brackets}\label{trb}

The brackets  $\bracket{-}{-}$ on $A_N$ and    $\langle - , - \rangle$
on $\check A=A/[A,A]$ are skew-symmetric but do not
necessarily satisfy the Jacobi identity. Their deviation from the
Jacobi identity is formulated in terms of triple brackets on $A$. A
\emph{triple bracket} on $A$ is a  linear map $\triple{-}{-}{-}:
A^{\otimes 3} \to A^{\otimes 3}$ which is a derivation in the third
variable and is cyclically symmetric. This means that for   all
$a,b,c,d\in A$,  $$\triple{a}{b}{cd} = c \triple{a}{b}{d} +
\triple{a}{b}{c}d  \quad {\rm {and}} \quad \triple{c}{a}{b}=
 P_{312} \triple{a}{b}{c}.$$
By \cite[Proposition 2.3.1]{VdB}, a  double bracket $\double{-}{-}$
on $ A$ determines a triple bracket on $A$ by
\begin{equation}\label{triple}
 \triple{-}{-}{-} =
\sum_{i=0}^2  P_{312}^i ( \double{-}{-} \otimes \id_A) (\id_A \otimes \double{-}{-}) P_{312}^{-i}.
\end{equation}

\begin{lemma}\label{qP--}  \cite[Proposition 7.5.2 \& Corollary 2.4.4]{VdB} Consider a
double bracket $\double {-} {-}$ on $A$, the associated triple
  bracket $\triple {-} {-} {-}$ on $A$,   the induced brackets $\bracket {-} {-}$ on $A_N$ for $N\geq 1$
  and $\langle - , - \rangle$ on $A$. For any $a,b,c\in A$
   and any $p,q,r,s,u,v\in \{1,\dots, N\}$, we have
\begin{eqnarray}
\label{Jacobi_tb}
&&\bracket {a_{pq}} {\bracket {b_{rs}} {c_{uv}}}+ \bracket {b_{rs}}  {\bracket  {c_{uv}} {a_{pq}}}+
\bracket  {c_{uv}}  {\bracket  {a_{pq}}  {b_{rs}}}\\
\nonumber &=&{\triple {a}  {b}{c}}_{uq}^{(1)}    {\triple {a}  {b}{c}}_{ps}^{(2)}
  {\triple {a}  {b}{c}}_{rv}^{(3)} - {\triple {a}  {c}{b}}_{rq}^{(1)} {\triple {a}  {c}{b}}_{pv}^{(2)} {\triple {a}  {c}{b}}_{us}^{(3)}
\end{eqnarray}
and
\begin{equation}\label{derivd}
\langle \langle a, b\rangle, c\rangle
- \langle a , \langle b , c \rangle \rangle + \langle b , \langle a , c \rangle \rangle =  m_3(\triple {b} {a} {c} - \triple {a} {b} {c})
\end{equation}
where $m_3:A^{\otimes 3}\to A$     carries any $x\in A^{\otimes 3}$ to $x^{(1)} x^{(2)}   x^{(3)}\in A$.
\end{lemma}

A double bracket on $A$ is \emph{strong} if the associated triple
bracket on $A$ satisfies
$$
m_3(\triple {a} {b} {c})= m_3( \triple {b} {a} {c})
$$
for all $a,b,c \in A$. Formula \eqref{derivd} and properties (i),
(ii),  (iii) of Section~\ref{stdb} imply that if a double bracket on
$A$ is strong, then the associated bracket $\langle - , - \rangle$
on $\check A$ is a Lie bracket and the pairing $\langle - , -
\rangle:\check A\times A\to   A$ turns $A$ into a $\check
A$-algebra.   The Lie bracket $\langle - , - \rangle$ in $\check{A}$
determined by a strong double bracket in $A$ has the following
property: for each $\check a\in \check A$, the map $\langle
\check{a}, - \rangle : \check A \to \check A$ is  induced by a
derivation of   $A$. Such a Lie bracket  in $\check{A}$ is called by
Crawley-Boevey \cite{Cb} an \emph{$H_0$-Poisson structure} on $A$.
By \cite[Theorem 4.5]{Cb}, for any $H_0$-Poisson structure
$\langle-,-\rangle: \check A \times \check A \to \check A$ on $A$
there is a unique Poisson bracket $\{-,-\}$ in   the subalgebra
$A_N^t$   of $A_N$ generated  by $\tr(A)$ such that
$\bracket{\tr(\check a)}{\tr(\check b)}=\tr\left(\langle \check a ,
\check b\rangle\right)$ for all $a,b\in A$. When the $H_0$-Poisson
structure  is induced by a strong double bracket on $A$, this claim
follows from Lemma~\ref{comparis}.

A double bracket on $A$ is  \emph{Poisson} if the associated triple bracket is   zero.
A Poisson double bracket is strong and moreover,  according to \eqref{Jacobi_tb},
the  associated   bracket  $\bracket {-} {-}$ on $A_N$ is a Poisson bracket.
We will be interested in a
somewhat different class of strong double brackets discussed in the next subsection.

\subsection{Quasi-Poisson double brackets}\label{qpdb}
Each derivation $D:A\to A\otimes A$   determines a triple
bracket on $A$ by
$$ {\triple{a}{b}{c}_D} = D(c)^{(1)} D(a)^{(2)}\otimes D(a)^{(1)}
D(b)^{(2)} \otimes D(b)^{(1)} D(c)^{(2)}
$$
for any $a,b,c\in A$. All requirements on a triple bracket are straightforward.
(See \cite[Proposition 4.1.1]{VdB} for a more general construction.)
The triple bracket determined in this way  by  the
derivation $E:A\to A\otimes A$ carrying  any $a\in A$ to $a\otimes 1
-1\otimes a$ can be explicitly computed as follows:
\begin{eqnarray}\label{compE}
{\triple{a}{b}{c}_E} &=&
a \otimes 1 \otimes bc + 1 \otimes ab \otimes c  +  ca\otimes b \otimes 1 +  c \otimes a \otimes b\\
\nonumber && -  1  \otimes a \otimes bc - a\otimes b \otimes c -  ca \otimes 1 \otimes b -   c \otimes ab \otimes 1.
\end{eqnarray}

 A double bracket $\double{-}{-}$ on $A$ is {\it quasi-Poisson}  if the
associated triple bracket  $\triple{-}{-}{-}$ satisfies
\begin{equation}\label{quasi-Poisson-}
 \triple{a}{b}{c}=  \triple{a}{b}{c}_E
\end{equation}
for all $a,b,c \in A$. It follows from \eqref{compE} that a
quasi-Poisson double bracket is strong. Hence, it induces  a Lie
bracket $\langle - , - \rangle$ on $\check A$  as well as a $\check A$-algebra structure on $A$.

\begin{lemma}\label{qP}  \cite[Theorem 7.12.2 \& Remark 7.12.3]{VdB}
For any quasi-Poisson double   bracket $\double{-}{-}$ on   $A$ and
any $N\geq 1$, the  algebra $A_N$ over $(G_N,\g_N)$ with the
 bracket $\{ -, -\}$   provided by Lemma~\ref{fdbtopb} is a
quasi-Poisson algebra.
\end{lemma}

\begin{proof} By Lemma~\ref{fdbtopb}, we need only to verify that
for all $x,y,z\in A_N$,
\begin{equation}\label{veriftriple}
 \bracket{x} {\bracket{y}{z}}  + \bracket{y} {\bracket{z}{x}} + \bracket {z} {\bracket{x}{y}} = \phi_N(x, y, z). 
\end{equation}
It is straightforward to see that both sides of this formula are
skew-symmetric $A_N$-valued trilinear forms on $A_N$ which are
derivations in each variable. Therefore it suffices to verify
\eqref{ssdM} for the generators of $A_N$. Let $\mathcal L$ and
$\mathcal R$ be respectively the left-hand and right-hand sides of
\eqref{veriftriple} for $x=a_{pq}, y=b_{rs}, z=c_{uv}$ with
$p,q,r,s,u,v\in \{1,\dots, N\}$ and $a,b,c\in A$. Computing $\mathcal L$ by
Lemma~\ref{qP--} and then using \eqref{quasi-Poisson-} and
\eqref{compE}, we obtain that
\begin{eqnarray*}
\mathcal L & = &
a_{uq} \delta_{ps} (bc)_{rv}+ \delta_{uq} (ab)_{ps} c_{rv}+(ca)_{uq} b_{ps} \delta_{rv}+ c_{uq} a_{ps} b_{rv}\\
&& -  \delta_{uq} a_{ps} (bc)_{rv}-a_{uq} b_{ps} c_{rv}-(ca)_{uq} \delta_{ps} b_{rv}-c_{uq} (ab)_{ps} \delta_{rv}\\
&& - a_{rq} \delta_{pv} (cb)_{us}- \delta_{rq} (ac)_{pv} b_{us}-(ba)_{rq} c_{pv} \delta_{us}- b_{rq} a_{pv} c_{us}\\
&& +\delta_{rq} a_{pv} (cb)_{us}+a_{rq} c_{pv} b_{us}+(ba)_{rq} \delta_{pv} c_{us}+b_{rq} (ac)_{pv} \delta_{us}.
\end{eqnarray*}
To compute $\mathcal R$, we observe that
\begin{eqnarray*}
&& (f_{ij}\otimes f_{jk}\otimes f_{ki}) (a_{pq} \otimes b_{rs} \otimes c_{uv})\\
&=& (\delta_{jq} a_{pi} - \delta_{pi} a_{jq}) (\delta_{ks} b_{rj} - \delta_{rj} b_{ks})(\delta_{iv} c_{uk} - \delta_{uk} c_{iv})\\
&=& \delta_{jq} a_{pi} \delta_{ks} b_{rj} \delta_{iv} c_{uk}+\delta_{jq} a_{pi} \delta_{rj} b_{ks} \delta_{uk} c_{iv}\\
&& + \delta_{pi} a_{jq} \delta_{ks} b_{rj} \delta_{uk} c_{iv}+ \delta_{pi} a_{jq} \delta_{rj} b_{ks} \delta_{iv} c_{uk}\\
&&  - \delta_{pi} a_{jq}  \delta_{ks} b_{rj} \delta_{iv} c_{uk} - \delta_{pi} a_{jq} \delta_{rj} b_{ks} \delta_{uk} c_{iv} \\
&&- \delta_{jq} a_{pi} \delta_{ks} b_{rj} \delta_{uk} c_{iv} - \delta_{jq} a_{pi} \delta_{rj} b_{ks} \delta_{iv} c_{uk}\\
&=&   a_{pv}   b_{rq}   c_{us}+ a_{pi} \delta_{rq} b_{us}  c_{iv} +   a_{jq}
  b_{rj} \delta_{us} c_{pv}+   a_{rq}   b_{ks} \delta_{pv} c_{uk}\\
&& - a_{jq}    b_{rj} \delta_{pv} c_{us}   -   a_{rq}   b_{us}   c_{pv}
-   a_{pi} b_{rq} \delta_{us} c_{iv} -   a_{pv} \delta_{rq} b_{ks}   c_{uk} \\
&=&   a_{pv}   b_{rq}   c_{us}  + (ac)_{pv} \delta_{rq} b_{us}
  +  (ba)_{rq} \delta_{us} c_{pv}     +   a_{rq}   (cb)_{us} \delta_{pv}\\
&&- (ba)_{rq} \delta_{pv} c_{us} -   a_{rq}   b_{us}   c_{pv}
 -   (ac)_{pv}  b_{rq} \delta_{us}   -   a_{pv} \delta_{rq} (cb)_{us}  .
\end{eqnarray*}
Similarly, we have
\begin{eqnarray*}
&& (f_{jk}\otimes f_{ij}\otimes f_{ki}) (a_{pq} \otimes b_{rs} \otimes c_{uv})\\
&=& (\delta_{kq} a_{pj} - \delta_{pj} a_{kq}) (\delta_{js} b_{ri} - \delta_{ri} b_{js})(\delta_{iv} c_{uk} - \delta_{uk} c_{iv})\\
&=& \delta_{kq} a_{pj} \delta_{js} b_{ri} \delta_{iv} c_{uk}+\delta_{kq} a_{pj} \delta_{ri} b_{js} \delta_{uk} c_{iv}\\
&&+ \delta_{pj} a_{kq} \delta_{js} b_{ri} \delta_{uk} c_{iv}+ \delta_{pj} a_{kq} \delta_{ri} b_{js} \delta_{iv} c_{uk}\\
&& - \delta_{pj} a_{kq}  \delta_{js} b_{ri} \delta_{iv} c_{uk} - \delta_{pj} a_{kq} \delta_{ri} b_{js} \delta_{uk} c_{iv}\\
&&- \delta_{kq} a_{pj} \delta_{js} b_{ri} \delta_{uk} c_{iv} - \delta_{kq} a_{pj} \delta_{ri} b_{js} \delta_{iv} c_{uk} \\
&=&  a_{ps}   b_{rv}   c_{uq}+  a_{pj}  b_{js} \delta_{uq} c_{rv}+   a_{uq}
\delta_{ps} b_{ri}   c_{iv}+   a_{kq}  b_{ps} \delta_{rv} c_{uk}\\
&& -   a_{kq}  \delta_{ps} b_{rv}   c_{uk} -   a_{uq}  b_{ps}   c_{rv}
-   a_{ps}  b_{ri} \delta_{uq} c_{iv} -   a_{pj} b_{js} \delta_{rv} c_{uq} \\
& = & a_{ps}   b_{rv}   c_{uq}+  (ab)_{ps} \delta_{uq} c_{rv}+ a_{uq} \delta_{ps} (bc)_{rv}+ b_{ps} \delta_{rv} (ca)_{uq}\\
&& -   (ca)_{uq}  \delta_{ps} b_{rv}    -   a_{uq} b_{ps}   c_{rv}
-   a_{ps}(bc)_{rv} \delta_{uq}   -   (ab)_{ps} \delta_{rv} c_{uq}  .\\
\end{eqnarray*}
 Substituting the resulting expressions in the formula
 $$\mathcal R=  - (f_{ij}\otimes f_{jk}\otimes f_{ki}) (a_{pq} \otimes b_{rs} \otimes
c_{uv})+(f_{jk}\otimes f_{ij}\otimes f_{ki}) (a_{pq} \otimes b_{rs} \otimes
c_{uv})$$ we obtain that
 $\mathcal R=\mathcal L$.
\end{proof}

 Under conditions of Lemma \ref{qP}, the $\g_N$-invariant subalgebra $A_N^{\g_N}$ of $ A_N$  is a Poisson algebra,
 and   the trace map $\tr:\check A\to A_N^{\g_N} $ is a Lie algebra homomorphism.

\subsection{Remark}   The original definitions of Van den Bergh apply when $\kk$ is a field of characteristic
zero; the  generalization   above  to   arbitrary commutative rings
is straightforward. We slightly modified the definition of a
quasi-Poisson double bracket in order to get rid of fractional
coefficients  appearing in \cite{VdB}:  a quasi-Poisson double
bracket in our sense is $2$ times a quasi-Poisson double bracket in the sense of \cite{VdB}.

 \section{Fox derivatives and Fox pairings}\label{sectionFOX}

We introduce  and study     Fox pairings   in   augmented
algebras. 

   \subsection{Fox derivatives}\label{kd1}   Fox derivatives were first introduced by Fox in his study of the
group rings of free groups. Fox' definitions   extend to  arbitrary
algebras  endowed with   augmentation homomorphisms. More precisely,
let $A$ be an algebra (over $\kk$) endowed with
   an    algebra homomorphism
$\varepsilon :A\to
 \kk$. A   linear  map $\partial :{A} \to
{A}$ is   a  {\it  left}  (respectively,   a {\it right) Fox
derivative} if for all $a,b\in {A}$, we have $\partial (ab)=\partial
( a) \varepsilon (b) + a \partial (b)$  (respectively, $\partial
(ab)=\partial ( a) b + \varepsilon (a)  \partial (b)$).

   For example, for   $e\in A$, the map ${A} \to {A}$ carrying any
$a\in A$ to $(a-\varepsilon(a)1)e$ is   a left   Fox derivative and
the map ${A} \to {A}$ carrying any $a\in A$ to $e(a-\varepsilon(a)1)$
is a right Fox derivative. Such  derivatives   are said to be {\it
inner}.

\subsection{Fox pairings}\label{kd2}
Let $A=(A,   \varepsilon )$ be an augmented algebra.
   A   {\it Fox pairing} in $A$  or,
shorter, an   {\it F-pairing} in $A$   is a  bilinear  map $\rho:{A}
\times {A} \to {A}$  which is a left Fox derivative with respect to
the first variable and a right Fox derivative with respect to the second variable. 
In the case of group algebras, an equivalent notion   
was   introduced independently in   \cite{Pa} under the name of ``biderivation''
and in \cite{Tu} under the name of ``$\Delta$-form.''  Here we follow the terminology of \cite{MT}.

Note for the record the  product formulas
\begin{equation}\label{equ1}\rho (a_1 a_2, b) =\rho (a_1, b)\,  \varepsilon (a_2) +  a_1   \rho (a_2, b) \quad {\rm for \,\,\, any} \,\,\, a_1,a_2,b
\in A,  \end{equation}
\begin{equation}\label{equ2}\rho (a, b_1 b_2)= \rho (a, b_1) b_2 + \varepsilon  (b_1) \rho (a, b_2)
\quad {\rm for \,\,\, any} \,\,\, a , b_1,b_2  \in {A}.\end{equation}

For example, any     $e\in A$ gives rise to  an F-pairing  $\rho_e$
defined by
$$
\rho_e (a,b)=(a -\varepsilon   (a)1 ) \, e\,
(b-\varepsilon  (b) 1)
$$
 for all $a,b\in A$.  We call such an F-pairing {\it inner}.

Fox pairings in $A$ form a $\kk$-module under the usual
addition/multiplication by scalars   of bilinear forms. The
inner  F-pairings form a submodule of this module. Two
F-pairings $\rho$ and $\rho'$ in $A$ are \emph{equivalent} if they
differ by  an inner  F-pairing, i.e$.$ if there exists
$e\in A$ such that $\rho-\rho'=\rho_e$.

\subsection{Fox pairings in Hopf algebras}\label{kd3}
Suppose   that  $A=(A,\Delta,\varepsilon,S)$ is   a Hopf algebra
with comultiplication $\Delta$,   counit $\varepsilon$ and antipode
$S$. We shall always use the counit $\varepsilon$ to define
F-pairings in $A$. The antipode $S$ induces a transposition on the
F-pairings in $A$ as follows.  The \emph{transpose} of a
 bilinear  form $\rho: A \times A \to A$ is the  bilinear
form $\overline{\rho}: A \times A \to A$ defined by
$$
\overline{\rho}(a,b) =   S\rho\!\left(S(b),S(a)\right)
$$
for   $a,b\in A$.  For example, the transpose of the  inner
F-pairing $\rho_e$   is $\rho_{S(e)}$ for all $e\in A$. Recall that
a Hopf algebra is  \emph{involutive} if its antipode is an
involutive map.

\begin{lemma}\label{transpose}
If $A$ is involutive, then the  transpose  $\overline \rho$  of any F-pairing $\rho: A
\times A \to A$ is an F-pairing  and $\overline{\overline
\rho}=\rho$.
\end{lemma}

\begin{proof}
  For   $a_1, a_2, b  \in A$, we have
\begin{eqnarray*}
\overline{\rho}(a_1a_2,b)&=&   S\rho\!\left(S (b),S(a_1 a_2) \right)\\
&=&  S\rho\!\left(S (b),S(a_2) S( a_1) \right)  \\
&=& S \big(\rho\!\left(S (b),S(a_2)  \right)  S( a_1) +\varepsilon (S(a_2)) \, \rho\!\left(S (b),S(a_1)  \right) \big) \\
&=& a_1  S  \rho\!\left(S (b),S(a_2)  \right)+ \varepsilon (a_2)\, S  \rho\!\left(S (b),S(a_1)  \right) \\
&=& a_1  \overline{\rho}(a_2,b)+  \overline{\rho}(a_1,b) \, \varepsilon (a_2).
\end{eqnarray*}
A similar computation shows that  $\overline{\rho}(a, b_1b_2)=
\overline{\rho}(a ,b_1) b_2+  \varepsilon (b_1) \overline{\rho}(a
,b_2)$ for all $a, b_1, b_2\in A$. Thus, $\overline \rho$ is an
F-pairing.  For all $a,b\in A$,
$$
\overline{\overline \rho}(a,b)  = S \overline \rho\!\left(S (b),S(a)  \right) = S^2 \rho (S^2(a), S^2(b))=\rho(a,b).
$$

\vspace{-0.5cm}
\end{proof}

An  F-pairing $\rho $ in an involutive Hopf algebra $A$ is
\emph{skew-symmetric} if $\overline{\rho}=-\rho$.  Any F-pairing
$\rho$ in such an $A$ induces a skew-symmetric F-pairing
$\rho^s=\rho - \overline{\rho}$ in $A$.

 We will need an alternate expression for $\overline \rho$,
which involves the comultiplication $\Delta$ of~$A$.  In this
expression and in the sequel we   use Sweedler's notation for
comultiplication: for $a\in A$ we write $ \Delta(a)=\sum_{(a)}
a'\otimes a''$ meaning that there is a finite family   $(a'_i ,
a''_i )_i$ in $A \times A$ such that $\Delta(a)=\sum_i a'_{i}
\otimes a''_{i}$. Similarly,  we write
$$(\Delta \otimes \id) \Delta(a)=(\id \otimes \Delta) \Delta
(a)=\sum_{(a)} a'\otimes a'' \otimes a'''.$$

\begin{lemma}\label{transpose+}
If $A$ is involutive, then for any F-pairing $\rho$ in $A$ and any
$a,b\in A$,  we have $\overline{\rho}(a,b) =  \sum_{(a),(b)} a'
S\rho\!\left(b'',a''\right) b'$.
\end{lemma}

\begin{proof}
The equality $\sum_{(b)} S(b')b''=\varepsilon(b)  1$  implies
that
\begin{eqnarray*}
 0\ = \ {\rho}(\varepsilon(b) 1,a)&=&   \sum_{(b)} \rho(S(b') b'',a)\\
&=&  \sum_{(b)} \big( \rho(S(b') ,a)  \, \varepsilon (b'') +S(b') \rho(  b'',a) \big) \\
&=&  \sum_{(b)} \big( \rho(S(b' \varepsilon (b'')) ,a)    +S(b') \rho(  b'',a) \big)\\
&=& \rho(S(b),a)+  \sum_{(b)}  S(b') \rho(  b'',a)  .
\end{eqnarray*}
Therefore \begin{equation}\label{firstrho+} \rho(S(b),a)=-
\sum_{(b)} S(b') \rho( b'',a) .\end{equation}
 A similar computation starting from $\sum_{(a)}
S(a')a''=\varepsilon(a)  1$ shows that $$ \rho(b ,a)=-
\sum_{(a)} \rho( b, S(a'))a'' .$$ Replacing here $a$ by $S(a)$ and
using the involutivity of $A$, we obtain that
\begin{equation}\label{firstrho++} \rho(b, S(a))=-  \sum_{(a)}   \rho(
b, a'' ) S(a') .\end{equation}  
Now, we can conclude: using \eqref{firstrho+} and then \eqref{firstrho++}, we obtain that
\begin{eqnarray*}
\overline{\rho}(a,b) \ = \  S\rho\!\left(S(b),S(a)\right) &=&
-S\Bigg( \sum_{(b)} S(b')  \rho(b'', S(a) )  \Bigg) \\
&= & S\Bigg( \sum_{(a), (b)} S(b')  \rho(b'', a'' )  S(a')\Bigg)\\
&=&\sum_{(a),(b)} a' S\rho\!\left(b'',a''\right) b'.
\end{eqnarray*}

\vspace{-0.5cm}
\end{proof}

\section{From Fox pairings to double brackets}\label{FFPTDB}

In this section, $A=(A,\Delta,\varepsilon,S)$ is   a Hopf algebra with comultiplication
$\Delta$,   counit $\varepsilon$ and antipode $S$. We show, under
certain assumptions, how  to derive a double bracket in  $A$ from a Fox pairing  in $A$.

\subsection{A double bracket from a Fox pairing}
Any bilinear form $\rho: A\times A\to A$ determines a  linear map
$\double{-}{-}^\rho:A \otimes A \to A \otimes A$ by
$$
\double{a}{b}^{\rho} =
\sum_{(a), (b), (\rho(a'',b''))} b'
S\!\left(\big(\rho( a'', b'')\big)'\right)  a' \otimes \big(\rho( a'', b'')\big)''
$$
for any $a,b\in A$.

\begin{lemma}\label{derivation_sv}
If $\rho $ is an F-pairing in $A$, then $\double{-}{-}^\rho $ is a
derivation in the second variable.
\end{lemma}

\begin{proof}
For any    right Fox derivative $\partial$ in $A$ and any $v\in
{A}$, we define a   linear  map $\delta= \delta_{v, \partial} :{A} \to A\otimes A$  by
\begin{equation}\label{delta}
\delta(b) = \sum_{(b),(\partial(b''))} b'
S\left(\big(\partial(b'')\big)'\right)  v \otimes
\big(\partial(b'')\big)''
\end{equation}
for  $b\in A$.
We claim that this map is a  derivation. This will imply the lemma because  for any $a\in A$, we have
$\double{a}{-}^\rho= \sum_{(a)} \delta_{a',\rho(a'',-)}$.

We now  prove the above claim.  For any $a,b\in A$, we    have
\begin{eqnarray*}
 \delta(ab) &=&
\sum_{(ab),(\partial((ab)''))} (ab)' S\left(\big(\partial((ab)'')\big)'\right) v \otimes \big(\partial((ab)'')\big)''\\
&=&
\sum_{(a),(b),(\partial(a''b''))} a'b ' S\left(\big(\partial(a''b'')\big)'\right) v \otimes \big(\partial(a''b'')\big)''.
\end{eqnarray*}
Using the fact that $\partial(a''b'') = \partial(a'') b'' + \varepsilon(a'')\partial(b'')$,  this sum splits into two parts:
\begin{eqnarray*}
&&
\sum_{(a),(b),(\partial(a'')b'')} a'b ' S\left(\big(\partial(a'')b''\big)'\right) v \otimes \big(\partial(a'')b''\big)''\\
& = &
\sum_{(a),(b),(\partial(a''))} a'b ' S\left(\big(\partial(a'')\big)'b''\right) v \otimes \big(\partial(a'')\big)''b'''\\
& = &
\sum_{(a),(b),(\partial(a''))} a'b ' S(b'')S\left(\big(\partial(a'')\big)'\right)  v \otimes \big(\partial(a'')\big)''b'''\\
& = &
\sum_{(a),(b),(\partial(a''))} a' \varepsilon(b') S\left(\big(\partial(a'')\big)'\right)  v \otimes \big(\partial(a'')\big)''b''\\
& = &
\sum_{(a),(\partial(a''))} a'  S\left(\big(\partial(a'')\big)'\right)  v \otimes \big(\partial(a'')\big)''b
\quad \quad \quad \quad \quad \quad  = \quad \delta(a) b
\end{eqnarray*}
and
\begin{eqnarray*}
&&
\sum_{(a),(b),(\varepsilon(a'')\partial(b''))} a'b '
S\left(\big(\varepsilon(a'')\partial(b'')\big)'\right) v \otimes \big(\varepsilon(a'')\partial(b'')\big)''\\
 & = &
\sum_{(a),(b),(\partial(b''))} a'\varepsilon(a'') b ' S\left(\big(\partial(b'')\big)'\right) v \otimes \big(\partial(b'')\big)''\\
 & = &
\sum_{(b),(\partial(b''))} ab ' S\left(\big(\partial(b'')\big)'\right) v \otimes \big(\partial(b'')\big)''
\quad \quad \quad \quad \quad \quad \quad  \quad \quad = \quad a\delta(b).
\end{eqnarray*}

\vspace{-0.5cm}
\end{proof}

\begin{lemma}\label{double_bracket}
If  $A$ is involutive and  $\rho $ is a skew-symmetric F-pairing in
$A$, then $\double{-}{-}^\rho$ is a double bracket.
\end{lemma}

\begin{proof}
  We start by proving that for an arbitrary F-pairing $\rho$ in $A$,
\begin{equation}\label{transpose_transpose}
\double{-}{-}^{\overline \rho} = P_{21} \double{-}{-}^\rho P_{21}.
\end{equation}
By definition, we have
\begin{eqnarray*}
\double{a}{b}^{\overline \rho} &=&
\sum_{(a), (b), (\overline\rho(a'',b''))} b'
S\!\left(\big(\overline\rho( a'', b'')\big)'\right)  a' \otimes \big(\overline\rho( a'', b'')\big)''
\end{eqnarray*}
where, by Lemma \ref{transpose+},
$$
\overline\rho(a'',b'')=
\sum_{(a''),(b'')} (a'')'\, S\rho\big((b'')'',(a'')''\big)\, (b'')'.
$$
Therefore
\begin{eqnarray*}
&&\double{a}{b}^{\overline \rho} \\
&=& \sum_{(a), (b),  ( a'' S\rho(b''',a''') b'' )} b'
S\left( \big( a'' S\rho\big(b''',a'''\big) b'' \big)' \right)
\, a' \otimes \left( a'' S\rho\big(b''',a'''\big) b'' \right)''\\
&=& \sum_{\substack{(a), (b),  (a''),(b'') \\(S\rho(b''',a''')) }} b'
S\left((a'')'   \big(S\rho(b''',a''')\big)'  (b'')'  \right)
\, a' \otimes (a'')''  \big(S\rho(b''',a''')\big)''   (b'')''\\
&=& \sum_{(a), (b), (S\rho(b'''',a'''')) } b'
S\left(  a'' \big(S\rho(b'''',a'''')\big)'   b'' \right)
\, a' \otimes a'''  \big(S\rho(b'''',a'''')\big)''b'''  \\
&=& \sum_{(a), (b), (S\rho(b'''',a'''')) } b'
S(b'')  S\left( \big(S\rho(b'''',a'''')\big)'  \right) S (a'')
\, a' \otimes  a''' \big(S\rho(b'''',a'''')\big)''  b''' \\
&=& \sum_{(a), (b), (S\rho(b''',a''')) } \varepsilon(b')  S\left( \big(S\rho(b''',a''')\big)'  \right) \varepsilon (a')
 \otimes  a'' \big(S\rho(b''',a''')\big)''  b'' \\
&=& \sum_{(a), (b), (S\rho(b''',a''')) }   S\left( \big(S\rho(b''',a''')\big)'  \right)
 \otimes  \varepsilon (a') a'' \big(S\rho(b''',a''')\big)''  \varepsilon(b') b'' \\
&=& \sum_{(a), (b), (S\rho(b'',a'')) }   S\left( \big(S\rho(b'',a'')\big)'  \right)
 \otimes  a' \big(S\rho(b'',a'')\big)''  b' \\
&=& \sum_{(a), (b), (\rho(b'',a'')) } S^2\left( \big(\rho(b'',a'')\big)''  \right)
 \otimes  a' S\left(\big(\rho(b'',a'')\big)'\right)  b'
\quad  = \quad P_{21}(\double{b}{a}^\rho).
\end{eqnarray*}
If $\rho$ is skew-symmetric, then
$$
\double{-}{-}^{\rho} = - \double{-}{-}^{\overline \rho}
 \stackrel{\hbox{\scriptsize (\ref{transpose_transpose})}}{=} -  P_{21} \double{-}{-}^\rho P_{21}.
$$
Now, Lemma \ref{derivation_sv} implies that $\rho$ is a double
bracket.
\end{proof}

\subsection{Computation for   inner  Fox pairings}
The following lemma computes $ \double{-}{-}^{\rho}: A \otimes A \to
A \otimes A$ for the  inner F-pairings $\rho$ in $A$.

\begin{lemma}\label{db_tp}
If $A$ is involutive, then for any $a,b,e\in A$,
$$
\double{a}{b}^{\rho_e} =
\sum_{(e)} \Big(  S(e') \otimes a e''b   + b S\!\left(e'\right)  a \otimes e''
-b S(e') \otimes a e'' -S(e')  a \otimes e'' b\Big).
$$
\end{lemma}

\begin{proof} By definition,
\begin{eqnarray*}
\double{a}{b}^{\rho_e} &=&
\sum_{(a), (b), (\rho_e(a'',b''))} b'
S\!\left(\big(\rho_e( a'', b'')\big)'\right)  a' \otimes \big(\rho_e( a'', b'')\big)''.
\end{eqnarray*}
We have
\begin{eqnarray*}
 && \Delta(\rho_e( a'', b''))=\Delta\big((a'' -\varepsilon(a'')1\big) \Delta(e) \Delta\big(b'' -\varepsilon(b'')1\big)\\
&=& \sum_{(a''),(e),(b'')} (a'')'e'(b'')' \otimes (a'')''e''(b'')''  + \sum_{(e)} \varepsilon(a'')\varepsilon(b'') e'\otimes e''\\
&& -  \sum_{(a''),(e)} \varepsilon(b'') (a'')' e'  \otimes (a'')''e'' -  \sum_{(e),(b'')} \varepsilon(a'') e'(b'')' \otimes e''(b'')'' .
\end{eqnarray*}
Thus,
$$
\double{a}{b}^{\rho_e} = X_{ab}+X-X_a-X_b
$$ where
\begin{eqnarray*}
X_{ab} & = &
\sum_{(a), (b), (a''),(e),(b'')} b'
S\!\left((a'')'e'(b'')'\right)  a' \otimes (a'')''e''(b'')''  \\
& = &
\sum_{(a), (b), (e)} b'
S\!\left(a'' e' b'' \right)  a' \otimes a'''e''b'''  \\
& = &
\sum_{(a), (b), (e)} b' S(b'') S(e') S(a'') a' \otimes a'''e''b'''  \\
& = &
\sum_{(a), (b), (e)} \varepsilon(b') S(e') \varepsilon(a') \otimes a''e''b''
\quad = \quad
\sum_{(e)} S(e')  \otimes a e''b,
\end{eqnarray*}
\begin{eqnarray*}
X_{a} & = &
\sum_{(a), (b), (a''),(e)} \varepsilon(b'') b'
S\!\left((a'')'e' \right)  a' \otimes (a'')''e''  \\
& = & \sum_{(a) ,(e)} b
S\!\left(a'' e' \right)  a' \otimes a''' e''  \\
& = & \sum_{(a) ,(e)} b S(e') S(a'') a' \otimes a''' e''  \\
& = & \sum_{(a) ,(e)} b S(e') \varepsilon(a') \otimes a'' e''
\quad  =  \quad \sum_{(e)} b S(e') \otimes a e'' ,
\end{eqnarray*}
\begin{eqnarray*}
X_{b} & = &
\sum_{(a), (b), (e),(b'')} b'
S\!\left( e'(b'')'\right)  \varepsilon(a'')a' \otimes e''(b'')''  \\
& = &
\sum_{(b), (e)} b'
S\!\left( e' b''\right)  a \otimes e'' b'''  \\
& = &
\sum_{ (b), (e)} b' S(b'') S(e')  a \otimes e'' b'''  \\
& = &
\sum_{ (b), (e)}  \varepsilon(b') S(e')  a \otimes e''b''
\quad  = \quad
\sum_{ (e)} S(e')  a \otimes e'' b
\end{eqnarray*}
and
$$
X =  \sum_{(a), (b), (e)} \varepsilon(b'')b'
S\!\left(e'\right)  \varepsilon(a'')a' \otimes e''
= \sum_{(e)} b  S\!\left(e'\right)  a \otimes e''.
$$

\vspace{-0.5cm}
\end{proof}

\subsection{Remarks}
1. If $\double{-}{-}$ is a double bracket in an augmented algebra $A=(A,\varepsilon)$, then the
map $ A \times A \to A$ carrying any pair $(a,b)\in A\times A$ to  $
(\varepsilon \otimes \id_A) \double{a}{b}$ is an F-pairing.
This defines a map, $\Phi$, from the set of double brackets in $A$
to the set of F-pairings in $A$. If $A$ is an involutive Hopf
algebra, then Lemma~\ref{double_bracket} defines a map, $\Psi$, from
the set of skew-symmetric F-pairings in $A$ to set of double
brackets in $A$. Clearly,   $\Phi  \Psi=\id$. Generally speaking,
the composition $\Psi \Phi  $ is not  defined.

2.  
As indicated to the authors by P. Schauenburg,
the fact that a right Fox derivative in a Hopf algebra $  A $ and an element of $  A$
determine a derivation $ A \to A \otimes A$ by  \eqref{delta} 
can be alternatively explained in terms of $A$-bimodules. The explanation is based on  the following two facts:  
1)   derivations $A \to M$ with values in an $A$-bimodule $M$
bijectively correspond to $A$-bimodule maps from the $A$-bimodule of differentials on $A$ to $M$, 
and 2) the Hopf algebra structure on $A$ allows to form the tensor product of $A$-bimodules.

\section{The homotopy intersection form of a surface}\label{PMT}

In this section, $\Sigma$ is an oriented surface with non-empty
boundary and a base point $* \in \partial  \Sigma$. We set $\pi =
\pi_1(\Sigma,*)$ and    consider the group algebra $A= \kk\pi$.

\subsection{The pairing $\eta$}\label{pairingeta}
We provide $ \partial \Sigma$ with the orientation induced by that
of $\Sigma$.  By paths and loops we shall mean piecewise-smooth
paths and loops in $\Sigma$. The product $\alpha \beta$ of two paths
$\alpha$ and $\beta$ is obtained   by running first along $\alpha$
and then along $\beta$. Given    two distinct simple (that is,
multiplicity 1) points $p,q$ on a path $\alpha$, we denote by
$\alpha_{pq}$ the path from $p$ to $q$   running along $\alpha$ in
the positive direction.

We shall  use a second base point $\bigdot \in \partial \Sigma
\setminus \{*\}$ lying ``slightly before'' $*$ on
 $\partial \Sigma$. We   fix an  embedded path $\bord_{\bigdot \ast}$
running from  $\bigdot$ to $\ast$ along $\partial \Sigma$ in the
positive direction, and we denote the inverse path by $\drob_{\ast
\bigdot}$. The element of $\pi$ represented by a loop $\alpha$ based
at $*$ is denoted $[\alpha]$. We say that  a loop $\alpha$  based at
$\bigdot$ \emph{represents}   $ [\drob_{\ast \bigdot} \alpha
\bord_{\bigdot \ast}] \in \pi$.

The \emph{homotopy intersection form} of $\Sigma$ is the
 bilinear  map $\eta: A \times A\to A$ defined, for any $a,b\in
\pi$, by
\begin{equation}\label{eta}
\eta(a,b) = \sum_{p\in \alpha \cap \beta}
\varepsilon_p(\alpha,\beta)  \,
\left[\drob_{*\bigdot }\alpha_{\bigdot p} \beta_{p *}\right].
\end{equation}
Here, $\alpha$ is a loop   based at $\bigdot$  and  representing
$a$,
 $\beta$ is a loop   based at $*$ and representing $b$ so that $\alpha$ and $\beta$ meet transversely in a
finite set $ \alpha \cap \beta$ of simple points of $\alpha, \beta$.
Each crossing  $p\in \alpha \cap \beta$ has a sign
$\varepsilon_p(\alpha,\beta)=\pm 1$ which is equal to $+1$ if  the
frame (the positive tangent vector of $\alpha$ at $p$, the positive
tangent vector of $\beta$ at $p$) is positively oriented and is
equal to $-1$ otherwise. It is easy to verify  that   $\eta$ is a
well-defined F-pairing: this is essentially the intersection pairing
introduced in \cite{Tu}, see also \cite{Pe} and \cite{MT}.
The pairing $\eta$   is  connected to Reidemeister's equivariant intersection pairings
and, in this   form,   it appears implicitly in \cite{Pa}.  

\subsection{The  pairing $\eta^s$} The group algebra $A=\kk\pi$ has a canonical structure of an involutive Hopf algebra.
The comultiplication $\Delta:A\to A\otimes A$, the counit $\varepsilon:A\to \kk$ and the antipode $S:A\to A$
 are defined by
$$\Delta(a)=a\otimes a, \quad \varepsilon (a)=1 , \quad S(a)=a^{-1}
\quad {\rm {for \, \, \, any}} \quad  a\in \pi\subset A.$$   By
Lemma \ref{transpose+}, the transpose $\overline\eta:A\times
A\to A$ of $\eta$ is given by $ \overline\eta (a,b)=a S(\eta (b,a))
b $ for   $a,b\in \pi$. It is easy to check (see \cite{Tu},
\cite{MT}) that $\eta + \overline\eta = -\rho_{1}$ where $\rho_{1}$
is the inner  F-pairing in $A$ associated with $1\in A$.
The F-pairing $$\eta^s = \eta -
 \overline\eta  = 2\eta + \rho_{1}:A\times A\to A $$ is skew-symmetric. By
definition, for any $a,b\in A$,
\begin{equation}\label{eta^s}
\eta^s(a,b) = 2 \eta(a,b) + \big(a-\varepsilon(a)1\big)\big(b-\varepsilon(b)1\big)
\, .
\end{equation}
Let $\double{-}{-}^s=\double{-}{-}^{\eta^s}$ be the double bracket
on $A$  determined by   $\eta^s$.  Lemma~\ref{db_tp} and formula
(\ref{eta^s}) give, for all $a,b\in A$,
\begin{equation}\label{db_hif}
\double{a}{b}^s = 2 \double{a}{b}^\eta +
  1 \otimes a b   + b   a \otimes 1    -   a \otimes  b -b  \otimes a.
\end{equation}
For $a,b\in \pi$, the term $\double{a}{b}^\eta$ can be explicitly
computed as follows.

\begin{lemma}\label{db_eta}   For any $a,b\in \pi$,
\begin{equation}\label{db_hif+}
\double{a}{b}^\eta =
 \sum_{p\in \alpha \cap \beta} \varepsilon_p(\alpha,\beta)\
\left[\beta_{*p}   \alpha_{p \bigdot }   \bord_{\bigdot \ast}\right]
\otimes \left[\drob_{*\bigdot }\alpha_{\bigdot p} \beta_{p *}\right]
\end{equation}
where $\alpha$ is a loop  based at $\bigdot$   and representing $a$,
$\beta$ is a loop   based at $*$ and representing $b$ such that
$\alpha$ and $\beta$ meet transversely in a finite set of simple
points.
\end{lemma}

\begin{proof}
Since $a$ and $b$ are group-like elements of $A$, we have
\begin{eqnarray}\label{compdouble}
\double{a}{b}^\eta &=&  \sum_{(\eta(a,b))} b
S\!\left(\eta( a, b)'\right)  a \otimes \eta( a, b)''.
\end{eqnarray}
Applying (\ref{eta}) we obtain that
\begin{eqnarray*}
\double{a}{b}^\eta &=&  \sum_{p\in \alpha \cap \beta} \varepsilon_p(\alpha,\beta)\
bS\!\left(\left[\drob_{*\bigdot }\alpha_{\bigdot p} \beta_{p *}\right]\right)  a
\otimes \left[\drob_{*\bigdot }\alpha_{\bigdot p} \beta_{p *}\right]\\
&=&  \sum_{p\in \alpha \cap \beta} \varepsilon_p(\alpha,\beta)\
[\beta] \left[(\beta^{-1})_{*p} ( \alpha^{-1})_{p \bigdot } \bord_{\bigdot *} \right]
 [\drob_{\ast \bigdot} \alpha \bord_{\bigdot \ast}]
\otimes \left[\drob_{*\bigdot }\alpha_{\bigdot p} \beta_{p *}\right]\\
&=&  \sum_{p\in \alpha \cap \beta} \varepsilon_p(\alpha,\beta)\
\left[\beta_{*p}   (\alpha^{-1})_{p \bigdot }  \alpha \bord_{\bigdot \ast}\right]
\otimes \left[\drob_{*\bigdot }\alpha_{\bigdot p} \beta_{p *}\right]\\
&=&  \sum_{p\in \alpha \cap \beta} \varepsilon_p(\alpha,\beta)\
\left[\beta_{*p}   \alpha_{p \bigdot }   \bord_{\bigdot \ast}\right]
\otimes \left[\drob_{*\bigdot }\alpha_{\bigdot p} \beta_{p *}\right].
\end{eqnarray*}

\vspace{-0.5cm}
\end{proof}

\begin{lemma}\label{homotopy triple bracket} The double bracket $\double{-}{-}^s$   is quasi-Poisson.
\end{lemma}

\begin{proof} In the proof we will use the following notation suggested in \cite{VdB}.
 Given a    map $\double{-}{-}: A^{\otimes 2}   \to A^{\otimes 2}$
and
  $u \in A$, $v\in A^{\otimes 2}$, we set
$$
\double{u}{v}_L = \double{u}{v^{(1)}} \otimes v^{(2)}, \quad
\double{u}{v}_R = v^{(1)} \otimes \double{u}{v^{(2)}},
$$
$$
\double{v }{u}_L = \double{v^{(1)}}{u} \otimes v^{(2)}, \quad
\double{v }{u}_R = v^{(1)} \otimes \double{v^{(2)}}{u}.
$$

Let  $\triple{-}{-}{-}^s$ be the   triple bracket in $A$ determined
by   $\double{-}{-}^s$. We must prove that for all $a,b,c \in A$,
\begin{eqnarray}
\label{tb_hif} \quad\quad \triple{a}{b}{c}^s &=&
a \otimes 1
\otimes bc + 1 \otimes ab \otimes c
   +   ca\otimes b \otimes 1
+   c \otimes a \otimes b \\
\notag && -  1  \otimes a \otimes bc            -
a\otimes b \otimes c -  ca \otimes 1 \otimes b -   c \otimes ab
\otimes 1.
\end{eqnarray}
It is enough to consider   $a,b,c\in \pi$. Using the
  skew-symmetry of $\double{-}{-}^s$, we
obtain
\begin{eqnarray*}
\triple{a}{b}{c}^s &=& \double{a}{\double{b}{c}^s}^s_L
+ P_{312}^{-1}\double{c}{\double{a}{b}^s}^s_L + {P_{312}}\double{b}{\double{c}{a}^s}^s_L\\
&=&  \double{a}{\double{b}{c}^s}^s_L  - P_{312}^{-1} P_{213}\double{\double{a}{b}^s}{c}^s_L
-  {P_{312}}\double{b}{{P_{21}}\double{a}{c}^s}^s_L\\
&=&  \double{a}{\double{b}{c}^s}^s_L  -  P_{132}\double{\double{a}{b}^s}{c}^s_L
- \double{b}{\double{a}{c}^s}^s_R.
\end{eqnarray*}
We now compute   the three resulting  terms. Firstly, we have
$$
\double{a}{\double{b}{c}^s}_L^s
= 2\double{a}{{\double{b}{c}^\eta}}^s_L
+ \double{a}{   1 \otimes  bc   + c   b\otimes 1  -c  \otimes b  -   b \otimes  c }^s_L,
$$
where
\begin{eqnarray}
\label{1a} \quad \quad \quad  2\double{a}{\double{b}{c}^\eta}_L^s &=&
4 \double{a}{\double{b}{c}^\eta}_L^\eta + 2 \otimes a \double{b}{c}^\eta -   2 a \otimes \double{b}{c}^\eta \\
  \notag  &&  +  2 P_{132}(\double{b}{c}^\eta * a \otimes 1) -  2 P_{213}(a \otimes \double{b}{c}^\eta)
\end{eqnarray}
and
\begin{eqnarray}
\label{1b} &&\double{a}{   1 \otimes  bc   + c   b\otimes 1  -c  \otimes b  -   b \otimes  c }^s_L \\
\notag&=&   c \double{a}{b}^s\otimes 1 +   \double{a}{c}^sb\otimes 1
 -    \double{a}{c}^s\otimes b -    \double{a}{b}^s\otimes c\\
\notag&=&  2 c \double{a}{b}^\eta\otimes 1 +  2 \double{a}{c}^\eta b\otimes 1
 -   2 \double{a}{c}^\eta \otimes b -  2  \double{a}{b}^\eta \otimes c\\
\notag&& +   c \otimes ab \otimes 1 +   c ba \otimes 1 \otimes 1
-  ca\otimes b \otimes 1 -   cb \otimes a \otimes 1\\
\notag&&  +  1 \otimes acb \otimes 1
+    ca \otimes b \otimes 1 -  a \otimes cb \otimes 1  -  c \otimes ab \otimes 1\\
\notag&&  -1\otimes ac \otimes b -  ca \otimes 1 \otimes b
+   a \otimes c \otimes b +   c \otimes a \otimes b \\
\notag&&  - 1 \otimes ab \otimes c -  ba \otimes 1 \otimes c
+   a \otimes b \otimes c +   b \otimes a \otimes c.
\end{eqnarray}

Secondly, we have
\begin{eqnarray*}
&&-  P_{132}\double{\double{a}{b}^s}{c}^s_L\\
&=&  - 2 P_{132}\double{\double{a}{b}^\eta }{c}^s_L
 -  P_{132}\double{  1 \otimes a b   + b   a \otimes 1  -b  \otimes a  -   a \otimes  b }{c}^s_L
\end{eqnarray*}
where
\begin{eqnarray}
\label{2a} && - 2 P_{132}\double{\double{a}{b}^\eta }{c}^s_L\\
\notag&=&  - 4 P_{132}\double{\double{a}{b}^\eta }{c}^\eta_L
- {2} P_{132}  P_{132}  (c \double{a}{b}^\eta \otimes 1)
- {2} P_{132} (1\otimes \double{a}{b}^\eta *c)\\
\notag&&+ {2} P_{132} P_{132}( \double{a}{b}^\eta \otimes c)
+ {2} P_{132} (c \otimes \double{a}{b}^\eta)\\
\notag&=&  - 4 P_{132}\double{\double{a}{b}^\eta }{c}^\eta_L
- {2}  c \double{a}{b}^\eta \otimes 1
- {2} P_{132} (1\otimes \double{a}{b}^\eta *c)\\
\notag&&+ {2}  \double{a}{b}^\eta \otimes c
+ {2} P_{132} (c \otimes \double{a}{b}^\eta)
\end{eqnarray}
and
\begin{eqnarray}
\label{2b}&&  -  P_{132}\double{  1 \otimes a b   + b   a \otimes 1  -b  \otimes a  -   a \otimes  b}{c}^s_L\\
\notag &=& -   P_{132}(b*\double{a}{c}^s \otimes 1) -   P_{132}(\double{b}{c}^s*a \otimes 1) \\
\notag&& +   P_{132}(\double{b}{c}^s \otimes a) +   P_{132}(\double{a}{c}^s \otimes b)\\
\notag&=& -  {2} P_{132}(b*\double{a}{c}^\eta \otimes 1)
-   {2} P_{132}(\double{b}{c}^\eta*a \otimes 1) \\
\notag&& +  {2} P_{132}(\double{b}{c}^\eta \otimes a) +  {2} P_{132}(\double{a}{c}^\eta \otimes b)\\
\notag&& - 1 \otimes 1 \otimes bac -    ca \otimes 1 \otimes b
+   a \otimes 1 \otimes bc  +   c \otimes 1 \otimes ba \\
\notag&&  -    cba \otimes 1 \otimes 1 -    a \otimes 1 \otimes bc
+    ba \otimes 1 \otimes c  +     ca \otimes 1 \otimes b\\
\notag&&      +   cb\otimes a \otimes 1   + 1  \otimes a \otimes bc
-  b\otimes a \otimes c   -   c \otimes a \otimes b \\
\notag&&    +   ca\otimes b \otimes 1    + 1  \otimes b \otimes ac
 -  a\otimes b \otimes c   -   c \otimes b \otimes a.
\end{eqnarray}

Thirdly, we have
$$
- \double{b}{\double{a}{c}^s}^s_R = - 2\double{b}{\double{a}{c}^\eta}^s_R
-\double{b}{1\otimes ac+ca \otimes 1 - c\otimes a - a \otimes c}^s_R
$$
where   \begin{eqnarray}
 \label{3a}  \quad \quad \quad - 2 \double{b}{\double{a}{c}^\eta}^s_R\! \! &= &\! \! - 4\double{b}{\double{a}{c}^\eta}^\eta_R
- {2} \double{a}{c}^\eta b \otimes 1  +  {2} \double{a}{c}^\eta \otimes b \\
 \notag &&- {2}P_{132}( b*\double{a}{c}^\eta\otimes 1) + {2} P_{132}(\double{a}{c}^\eta \otimes b)
\end{eqnarray}
  and
\begin{eqnarray}
\label{3b}&&-\double{b}{1\otimes ac+ca \otimes 1 - c\otimes a - a \otimes c}^s_R\\
\notag&=& - 1  \otimes a\double{b}{c}^s -    1 \otimes \double{b}{a}^sc
+    c \otimes \double{b}{a}^s +   a \otimes \double{b}{c}^s \\
\notag&=&  -  {2} \otimes a\double{b}{c}^\eta
 - {2}  \otimes \double{b}{a}^\eta c
+  {2}  c \otimes \double{b}{a}^\eta +  {2}  a \otimes \double{b}{c}^\eta\\
\notag&&- 1  \otimes a \otimes bc
- 1  \otimes acb \otimes 1+  1 \otimes ab \otimes c +  1 \otimes ac \otimes b \\
\notag&&  - 1 \otimes ab \otimes c
- 1 \otimes 1 \otimes bac +  1 \otimes b \otimes ac +  1 \otimes a \otimes bc\\
\notag&& +   c \otimes ab \otimes 1 +
   c \otimes 1 \otimes ba -  c \otimes b \otimes a -  c \otimes a \otimes b\\
\notag&& +   a \otimes cb \otimes 1
+   a \otimes 1 \otimes bc -    a \otimes b \otimes c -   a \otimes c \otimes b.
\end{eqnarray} Summing up the equalities (\ref{1a}) -- (\ref{3b}) and canceling
identical terms with opposite signs,   we obtain that
$\triple{a}{b}{c}^s$ is equal to the following sum:
\begin{eqnarray*}
\label{abc} &&  4 \double{a}{\double{b}{c}^\eta}_L^\eta
 -  4 P_{132}\double{\double{a}{b}^\eta }{c}^\eta_L - 4 \double{b}{\double{a}{c}^\eta}^\eta_R \\
 \notag&&
 -  4 P_{132}(b*\double{a}{c}^\eta \otimes 1)  +  4 P_{132}(\double{a}{c}^\eta \otimes b)
- {2} P_{132} (1\otimes \double{a}{b}^\eta *c) \\
 \notag&&  - {2}  \otimes \double{b}{a}^\eta c
+ {2} P_{132} (c \otimes \double{a}{b}^\eta) +  {2}  c \otimes \double{b}{a}^\eta  \\
\notag&&  - {2} \otimes1 \otimes bac +  {2} c \otimes 1 \otimes ba
+  {2}  \otimes b \otimes ac   -  {2} c \otimes b \otimes a\\
\notag&&  -  ca \otimes 1 \otimes b   +   ca\otimes b \otimes 1  - 1 \otimes ab \otimes c
 +   1 \otimes a \otimes bc   \\
\notag&& +    a \otimes 1 \otimes bc   +   c \otimes ab \otimes 1
 -   c \otimes a \otimes b    -   a\otimes b \otimes c.
\end{eqnarray*}
Next, we deduce from  the equality  $\eta + \overline\eta =
-\rho_{1}$ and   \eqref{transpose_transpose} that
$$
-{P_{21}}\double{a}{b}^\eta = \double{b}{a}^\eta +  ab\otimes  1 + 1\otimes ba  -a \otimes b -b \otimes a.
$$
Hence
\begin{eqnarray*}
&& - {2} P_{132} (1\otimes \double{a}{b}^\eta *c)  \ = \ -  {2} \otimes {P_{21}}(\double{a}{b}^\eta) c \\
&=&  {2} \otimes \double{b}{a}^\eta c +   {2}\otimes  ab\otimes  c
+  {2}\otimes 1  \otimes ba c
- {2} \otimes a \otimes bc - {2} \otimes b \otimes ac
\end{eqnarray*}
and
\begin{eqnarray*}
&&  {2} P_{132} (c \otimes \double{a}{b}^\eta)   \ = \ 2 c \otimes P_{21}( \double{a}{b}^\eta)   \\
&=& - {2} c\otimes \double{b}{a}^\eta - {2}   c\otimes ab\otimes  1
-  {2} c\otimes 1\otimes ba  +  {2}c\otimes a \otimes b +  {2} c\otimes b \otimes a.
\end{eqnarray*}
Thus,  the expression for $\triple{a}{b}{c}^s$ above
simplifies to
\begin{eqnarray*}
 \quad\quad\quad  \triple{a}{b}{c}^s &=&   4\double{a}{\double{b}{c}^\eta}_L^\eta
 -  4P_{132}\double{\double{a}{b}^\eta }{c}^\eta_L - 4\double{b}{\double{a}{c}^\eta}^\eta_R\\
\notag&& - 4 P_{132}(b*\double{a}{c}^\eta \otimes 1)  +  4 P_{132}(\double{a}{c}^\eta \otimes b)\\
\notag&& +a \otimes 1
\otimes bc + 1 \otimes ab \otimes c
   +   ca\otimes b \otimes 1
+   c \otimes a \otimes b \\
\notag && -  1  \otimes a \otimes bc            -
a\otimes b \otimes c -  ca \otimes 1 \otimes b -   c \otimes ab
\otimes 1 .
\end{eqnarray*}
To finish the proof of \eqref{tb_hif}, it suffices to show that
\begin{eqnarray}
\label{abc_bisb} && \double{a}{\double{b}{c}^\eta}_L^\eta
 -   P_{132}\double{\double{a}{b}^\eta }{c}^\eta_L -  \double{b}{\double{a}{c}^\eta}^\eta_R\\
\notag &=&     P_{132}(b*\double{a}{c}^\eta \otimes 1)  - P_{132}(\double{a}{c}^\eta \otimes b).
\end{eqnarray}

To proceed,   we fix a third base point $\circ\in \partial \Sigma$
which we assume to be ``slightly before'' $\bigdot$ on $\partial
\Sigma$. Consider  a loop $\alpha$ based at $\circ$ representing $a$
(i.e$.$ $[\drob_{\ast \circ} \alpha \bord_{ \circ\ast}] =a \in
\pi$), a loop $\beta$ based at $\bigdot$ representing $b$, and a
loop $\gamma$ based at $\ast$ representing $c$; we   assume that
$\alpha$, $\beta$, $\gamma$ meet transversely in a finite set of
simple points and that $\alpha \cap \beta \cap \gamma =
\varnothing$. The first term in \eqref{abc_bisb} is
\begin{eqnarray*}
\label{1'} && \double{a}{\double{b}{c}^\eta}_L^\eta   \ =  \ \sum_{p\in \beta \cap \gamma} \varepsilon_p(\beta,\gamma)\ \double{a}{
\left[\gamma_{*p}   \beta_{p \bigdot }   \bord_{\bigdot \ast}\right]
\otimes \left[\drob_{*\bigdot }\beta_{\bigdot p} \gamma_{p *}\right]}_L^\eta\\
\notag  &=&\sum_{p\in \beta \cap \gamma} \varepsilon_p(\beta,\gamma)\ \double{a}{
\left[\gamma_{*p}   \beta_{p \bigdot }   \bord_{\bigdot \ast}\right] }^\eta\otimes \left[\drob_{*\bigdot }\beta_{\bigdot p} \gamma_{p *}\right]\\
\notag &=&\sum_{\substack{p\in \beta \cap \gamma\\q\in \alpha \cap \gamma_{*p}}}
\varepsilon_p(\beta,\gamma) \varepsilon_q(\alpha,\gamma)
[\gamma_{*q}\alpha_{q\circ} \bord_{\circ \ast}]
\otimes [\drob_{\ast \circ} \alpha_{\circ q} \gamma_{qp} \beta_{p\bigdot} \bord_{\bigdot \ast}]
\otimes \left[\drob_{*\bigdot }\beta_{\bigdot p} \gamma_{p *}\right]\\
\notag && + \sum_{\substack{p\in \beta \cap \gamma\\ q\in \alpha \cap \beta_{p\bigdot}}}
\varepsilon_p(\beta,\gamma) \varepsilon_q(\alpha,\beta)
[\gamma_{*p}\beta_{pq} \alpha_{q\circ} \bord_{\circ \ast}]
\otimes [\drob_{\ast \circ} \alpha_{\circ q} \beta_{q\bigdot} \bord_{\bigdot \ast}]
\otimes \left[\drob_{*\bigdot }\beta_{\bigdot p} \gamma_{p *}\right].
\end{eqnarray*}
The second term in \eqref{abc_bisb} is
\begin{eqnarray*}
\label{2'}  && \quad -P_{132} \double{\double{a}{b}^\eta }{c}^\eta_L\\
&=& -P_{132}  \sum_{p\in \alpha \cap \beta} \varepsilon_p(\alpha,\beta)
\double{
\left[\drob_{\ast \bigdot} \beta_{\bigdot p}   \alpha_{p \circ}   \bord_{\circ \ast}\right]
\otimes \left[\drob_{*\circ }\alpha_{\circ p} \beta_{p \bigdot} \bord_{\bigdot \ast}\right]}{c}^\eta_L\\
\notag &=&  -P_{132} \sum_{p\in \alpha \cap \beta} \varepsilon_p(\alpha,\beta)
\double{\left[\drob_{\ast \bigdot} \beta_{\bigdot p}    \alpha_{p \circ}   \bord_{\circ \ast}\right] }{c}^\eta
\otimes \left[\drob_{*\circ }\alpha_{\circ p} \beta_{p \bigdot} \bord_{\bigdot \ast}\right]\\
\notag &=&  -P_{132} \sum_{\substack{p\in \alpha \cap \beta\\q\in \beta_{\bigdot p} \cap\gamma}}
\varepsilon_p(\alpha,\beta)\varepsilon_q(\beta,\gamma)
[\gamma_{\ast q} \beta_{qp} \alpha_{p\circ} \bord_{\circ \ast}]
\otimes [\drob_{\ast \bigdot} \beta_{\bigdot q} \gamma_{q\ast}]
\otimes \left[\drob_{*\circ }\alpha_{\circ p} \beta_{p \bigdot} \bord_{\bigdot \ast}\right]\\
\notag &&  -P_{132} \sum_{\substack{p\in \alpha \cap \beta\\q\in \alpha_{p \circ } \cap\gamma}}
\varepsilon_p(\alpha,\beta)\varepsilon_q(\alpha,\gamma)
[\gamma_{\ast q}  \alpha_{q\circ} \bord_{\circ \ast}]
\otimes [\drob_{\ast \bigdot} \beta_{\bigdot p}  \alpha_{pq} \gamma_{q\ast}]
\otimes \left[\drob_{*\circ }\alpha_{\circ p}  \beta_{p \bigdot} \bord_{\bigdot \ast}\right]\\
\notag &=& - \sum_{\substack{p\in \alpha \cap \beta\\q\in \beta_{\bigdot p} \cap\gamma}}
\varepsilon_p(\alpha,\beta)\varepsilon_q(\beta,\gamma)
[\gamma_{\ast q} \beta_{qp} \alpha_{p\circ} \bord_{\circ \ast}]
\otimes \left[\drob_{*\circ }\alpha_{\circ p}  \beta_{p \bigdot} \bord_{\bigdot \ast}\right]
\otimes [\drob_{\ast \bigdot} \beta_{\bigdot q} \gamma_{q\ast}] \\
\notag &&  - \sum_{\substack{p\in \alpha \cap \beta\\q\in \alpha_{p \circ } \cap\gamma}}
\varepsilon_p(\alpha,\beta)\varepsilon_q(\alpha,\gamma)
[\gamma_{\ast q}  \alpha_{q\circ} \bord_{\circ \ast}]
\otimes  \left[\drob_{*\circ }\alpha_{\circ p}  \beta_{p \bigdot} \bord_{\bigdot \ast}\right]
\otimes [\drob_{\ast \bigdot} \beta_{\bigdot p}  \alpha_{pq} \gamma_{q\ast}].
\end{eqnarray*}
The third term in \eqref{abc_bisb} is
\begin{eqnarray*}
 -\double{b}{\double{a}{c}^\eta}_R^\eta
&=& -\sum_{p\in \alpha \cap \gamma} \varepsilon_p(\alpha,\gamma)
\double{b}{\left[\gamma_{\ast p}   \alpha_{p \circ}   \bord_{\circ \ast}\right]
\otimes \left[\drob_{*\circ }\alpha_{\circ p} \gamma_{p \ast} \right]}^\eta_R\\
&=& - \sum_{p\in \alpha \cap \gamma} \varepsilon_p(\alpha,\gamma)
\left[\gamma_{\ast p}   \alpha_{p \circ}   \bord_{\circ \ast}\right]
\otimes \double{b}{\left[\drob_{*\circ }\alpha_{\circ p} \gamma_{p \ast} \right]}^\eta.
\end{eqnarray*}
To finish the computation of the third term, we denote by $\alpha'$
a loop based at $\bigdot$ and $\beta'$ a loop based at $\circ$
obtained from $\alpha$ and $\beta$ respectively by ``switching''
their base points along $\partial \Sigma$: thus, $\alpha' \cap
\beta'$ is  $\alpha \cap \beta$ with four extra points. We obtain
\begin{eqnarray*}
\label{3'} && \quad -\double{b}{\double{a}{c}^\eta}_R^\eta \ = \ - \sum_{p\in \alpha' \cap \gamma} \varepsilon_p(\alpha',\gamma)
\left[\gamma_{\ast p}   \alpha_{p \bigdot}'   \bord_{\bigdot \ast}\right]
\otimes \double{b}{\left[\drob_{*\bigdot }\alpha'_{\bigdot p} \gamma_{p \ast} \right]}^\eta\\
\notag &=& - \sum_{\substack{p\in \alpha' \cap \gamma\\q\in \beta' \cap \gamma_{p\ast}}}
\varepsilon_p(\alpha',\gamma) \varepsilon_q(\beta',\gamma)
\left[\gamma_{\ast p}   \alpha_{p \bigdot}'   \bord_{\bigdot \ast}\right]
\otimes [\drob_{\ast \bigdot} \alpha'_{\bigdot p} \gamma_{p q}  \beta'_{q\circ} \bord_{\circ \ast}]
\otimes [\drob_{\ast \circ}\beta'_{\circ q} \gamma_{q\ast}]\\
\notag && - \sum_{\substack{p\in \alpha' \cap \gamma\\q\in \beta' \cap \alpha'_{\bullet p}}}
\varepsilon_p(\alpha',\gamma) \varepsilon_q(\beta',\alpha')
\left[\gamma_{\ast p}   \alpha_{p \bigdot}'   \bord_{\bigdot \ast}\right]
\otimes [\drob_{\ast \bigdot} \alpha'_{\bigdot q}\beta'_{q\circ} \bord_{\circ \ast}]
\otimes [\drob_{\ast \circ}\beta'_{\circ q} \alpha'_{qp} \gamma_{p\ast}]\\
\notag &=& - \sum_{\substack{p\in \alpha \cap \gamma\\q\in \beta \cap \gamma_{p\ast}}}
\varepsilon_p(\alpha,\gamma) \varepsilon_q(\beta,\gamma)
\left[\gamma_{\ast p}   \alpha_{p \circ}   \bord_{\circ\ast}\right]
\otimes [\drob_{\ast \circ} \alpha_{\circ p} \gamma_{p q}  \beta_{q\bigdot} \bord_{\bigdot \ast}]
\otimes [\drob_{\ast \bigdot}\beta_{\bigdot q} \gamma_{q\ast}]\\
\notag && - \sum_{\substack{p\in \alpha \cap \gamma\\q\in \beta \cap \alpha_{\circ p}}}
\varepsilon_p(\alpha,\gamma) \varepsilon_q(\beta,\alpha)
\left[\gamma_{\ast p}   \alpha_{p \circ}   \bord_{\circ \ast}\right]
\otimes [\drob_{\ast \circ} \alpha_{\circ q}\beta_{q\bigdot} \bord_{\bigdot \ast}]
\otimes [\drob_{\ast \bigdot}\beta_{\bigdot q} \alpha_{qp} \gamma_{p\ast}]\\
\notag && - \sum_{p\in \alpha \cap \gamma} \varepsilon_p(\alpha,\gamma)
\left[\gamma_{\ast p}   \alpha_{p \circ}   \bord_{\circ \ast}\right]  \otimes b
\otimes [\drob_{\ast \circ}\alpha_{\circ p} \gamma_{p\ast}]\\
\notag && + \sum_{p\in \alpha \cap \gamma} \varepsilon_p(\alpha,\gamma)
\left[\gamma_{\ast p}   \alpha_{p \circ}   \bord_{\circ \ast}\right]  \otimes 1
\otimes b [\drob_{\ast \circ}\alpha_{\circ p} \gamma_{p \ast}]\\
\notag &=& - \sum_{\substack{p\in \alpha \cap \gamma\\q\in \beta \cap \gamma_{p\ast}}}
\varepsilon_p(\alpha,\gamma) \varepsilon_q(\beta,\gamma)
\left[\gamma_{\ast p}   \alpha_{p \circ}   \bord_{\circ\ast}\right]
\otimes [\drob_{\ast \circ} \alpha_{\circ p} \gamma_{p q}  \beta_{q\bigdot} \bord_{\bigdot \ast}]
\otimes [\drob_{\ast \bigdot}\beta_{\bigdot q} \gamma_{q\ast}]\\
\notag && - \sum_{\substack{p\in \alpha \cap \gamma\\q\in \beta \cap \alpha_{\circ p}}}
\varepsilon_p(\alpha,\gamma) \varepsilon_q(\beta,\alpha)
\left[\gamma_{\ast p}   \alpha_{p \circ}   \bord_{\circ \ast}\right]
\otimes [\drob_{\ast \circ} \alpha_{\circ q}\beta_{q\bigdot} \bord_{\bigdot \ast}]
\otimes [\drob_{\ast \bigdot}\beta_{\bigdot q} \alpha_{qp} \gamma_{p\ast}]\\
\notag && -P_{132}(\double{a}{c}^\eta \otimes b) + P_{132} (b*\double{a}{c}^\eta \otimes 1).
\end{eqnarray*}
Summing up these expressions for the first three terms of
\eqref{abc_bisb}, and canceling pairwise the six sums indexed by
$(p,q)$, we obtain \eqref{abc_bisb} and \eqref{tb_hif}.

\vspace{-0.5cm}
\end{proof}

\subsection{Proof of Theorem \ref{CQPS}}\label{endofmain} By Lemma~\ref{double_bracket},
the skew-symmetrized homotopy intersection pairing $\eta^s$
induces a double bracket
  $\double{-}{-}^s$ on   $A=\kk\pi$. By   Lemma~\ref{homotopy triple bracket}, this double
bracket  is quasi-Poisson. For each $N\geq 1$, we endow the
$(G_N,\g_N)$-algebra $A_N$   with the
 bracket $\{ -, -\}$  derived from  $\double{-}{-}^s$  via Lemma~\ref{fdbtopb}.
 In the notation of  Lemma~\ref{db_eta},   this bracket is computed by
\begin{eqnarray}
\label{addedf}  \{a_{ij}, b_{uv}\}& =  &
2 \sum_{p\in \alpha \cap \beta}   \varepsilon_p(\alpha,\beta)\
\left[\beta_{*p}   \alpha_{p \bigdot }   \bord_{\bigdot \ast}\right]_{uj}
  \left[\drob_{*\bigdot }\alpha_{\bigdot p} \beta_{p *}\right]_{iv}\\
\notag&& + \, \delta_{uj} (ab)_{iv} +       \delta_{iv} (ba)_{uj} - a_{uj} b_{iv} - b_{uj} a_{iv}
\end{eqnarray}
  for all $a,b\in \pi$   and $i,j,u,v\in \{1,  \ldots, N\}$.
According to Lemma~\ref{qP}, this turns $A_N$ into a quasi-Poisson
algebra and proves the first claim of Theorem \ref{CQPS}.

 As explained in  Sections \ref{stdb}--\ref{qpdb}, the quasi-Poisson double bracket
  $\double{-}{-}^s$ in    $A$ induces a Lie bracket $\langle -,- \rangle$  in   $\check A=A/[A,A]$.
  Recall from Section~\ref{AMT--1} that $\check A= \kk\check \pi$ is the free $\kk$-module whose basis
$\check \pi$   can be identified with the set of free homotopy
classes of loops in $\Sigma$. Formulas
  \eqref{otherbracket}, \eqref{db_hif} and \eqref{db_hif+} imply
  that for any generic loops $\alpha, \beta$ in $\Sigma$,
  $$\langle \alpha, \beta\rangle= 2 \sum_{p\in \alpha \cap \beta} \varepsilon_p(\alpha,\beta)\  \alpha_p \beta_p$$
  where $\alpha_p\beta_p$ is the product of the loops $\alpha$ and
  $\beta$ based at $p$. The right-hand side of this formula is
    2 times the Goldman Lie bracket in $\check A$.
  Therefore the second claim of Theorem \ref{CQPS} follows directly
  from Lemma~\ref{comparis}.

\subsection{Remarks.}\label{Chas}

1. The bracket $\langle -, - \rangle: A \times A \to A$ associated
to the double bracket $\double{-}{-}^s$ by formula
\eqref{otherbracket} is twice the form $\sigma$ introduced by
Kawazumi and Kuno in \cite{KK}. An operation similar to
$\double{-}{-}^\eta:A \otimes A \to A \otimes A$ has been
independently introduced by Kawazumi and Kuno
\cite{KK_intersections} for other purposes.

2. When   $2=0 $ in $\kk$,   the quasi-Poisson double bracket in
$A=\kk \pi$ simplifies to
$$
\double{a}{b}^s = 1 \otimes ab + ba \otimes 1 + a \otimes b + b \otimes a
$$
  and  the quasi-Poisson bracket $\bracket{-}{-}$
on $A_N$ simplifies to
$$
 \bracket{a_{ij}}{b_{uv}} = \delta_{uj} (ab)_{iv} +     \delta_{iv}  (ba)_{uj}  + a_{uj} b_{iv} + b_{uj} a_{iv}.
$$
The latter  formula does not involve  topology of the surface
and   yields a    quasi-Poisson bracket on $A_N$
for any algebra $A$ (provided $2=0$ in $\kk$).

 3. Consider the case where   $\Sigma$ is a compact connected
oriented  surface of genus $g\geq 1$ with   $\partial \Sigma \cong
S^1$. For $\kk=\QQ$,  the F-pairing $\eta$ admits a tensorial
description  in terms of  a symplectic expansion of $\pi$, see
\cite{Mas}, \cite{MT}. In  a sequel to this paper, the authors will
derive a tensorial description of   the double bracket
$\double{-}{-}^s$ and use it to relate this bracket to the Poisson
double bracket associated by Van den Bergh \cite{VdB} to the quiver
having one vertex and $g$ edges. This will allow us to relate the
cobracket  of loops in $\Sigma$ introduced by the second-named
author to the cobracket of loops in   quivers introduced by T.\ Schedler.

4. A part of the results of this paper  has an analogue 
for oriented  smooth  manifolds with boundary of dimension $\geq 3 $. 
Given such a  manifold  $M$ with base point $\ast\in \partial M$,   
consider the space $\Omega=\Omega(M,\ast)$ of loops in $M$ based at $\ast$. 
Suppose that the ground ring $\kk$ is a field.
  Since $\Omega$  is an $H$-group,   its homology $A=H_\ast(\Omega;\kk)$ has a natural structure of a graded Hopf algebra. 
Imitating the formula \eqref{db_hif+} in the context of string topology of Chas and Sullivan  \cite{CS}, 
we define  a double bracket in $A$. This double bracket induces 
a  graded  Poisson    bracket 
in the associated commutative graded algebras $\{A_N\}_{N\geq 1}$.  
One can   view $A_N$ as the ``coordinate algebra" 
of the  ``scheme''   of graded algebra homomorphisms $A\to \Mat_N(B)$
 where $B$ runs over all commutative graded algebras. 
If $M$ is simply connected and   $\kk$ is a field of characteristic zero, 
then by the Milnor--Moore theorem,  $A$ is the universal enveloping algebra 
of the graded Lie algebra $\oplus_{i\geq 2} \, \pi_i(M,\ast)  \otimes \kk  $   
with Whitehead product in the role of the graded  Lie bracket.  
Thus,   $A_N$ is the  ``coordinate algebra''  of the   ``scheme''  of Lie algebra homomorphisms 
from $\oplus_{i\geq 2} \,  \pi_i(M,\ast) \otimes \kk $ to $\Mat_N(B)$ 
where $B$ runs over all commutative graded algebras.  
These   results will be discussed in detail in a sequel to this paper.

5. The constructions of this paper can be generalized to the case of surfaces with several
marked points on the boundary. This generalization involves representations of the fundamental groupoid
of the surface based at those points. We plan to study this general case  in another place.

  \section{Quasi-Poisson structures on the representation
manifolds}\label{QPMRS}

In this section we construct a quasi-Poisson structure on the
representation manifold associated with a   compact oriented surface
with boundary. We begin by  recalling the notion of a quasi-Poisson
manifold following \cite{AKsM}. Throughout this section  we assume
that   $\kk=\RR$.

\subsection{Quasi-Poisson manifolds} \label{QPA4}
 We first fix  a few conventions. For a vector space  $V$ over $\RR$, we
denote by $\Lambda V$ the exterior algebra of $V$, i.e$.$, the
quotient of the tensor algebra $\oplus_{n\geq 0}\, V^{\otimes n} $
by the two-sided ideal generated by the set $\{v\otimes v\}_{v\in
V}$. As usual, multiplication in $\Lambda V=\oplus_{n\geq 0}\,
\Lambda^n V $ is denoted by the symbol $\wedge$. We define a linear
map $\iota: \Lambda V \to \oplus_{n\geq 0}\, V^{\otimes n}$  by
$$\iota( v_1\wedge \cdots \wedge v_n ) = \sum_{\sigma  } \varepsilon(\sigma)\, v_{\sigma(1)} \otimes \cdots
\otimes v_{\sigma(n)}$$ for any $v_1,..., v_n\in V$ with $n\geq 1$.
Here $\sigma$ runs over all  permutations of the set $\{1, ...,n\}$
and $ \varepsilon(\sigma)$ is the sign of  $\sigma$. The map $\iota$
is an isomorphism of $\Lambda V $ onto the subspace of $
\oplus_{n\geq 0}\, V^{\otimes n}$ consisting of all skew-symmetric
tensors.

 Consider a smooth manifold $M$ and let $T=T(M)$ be
its tangent bundle. Applying the functor $V\mapsto \Lambda^n V$
fiber-wise we obtain  a smooth vector bundle $\Lambda^n T $ over $M$
for all $n\geq 0$. Let $C^\infty(\Lambda^n T )$ be the vector space
of smooth sections of $\Lambda^n T $. The elements of
$C^\infty(\Lambda^n T )$ are called {\it $n$-vector fields} on~$M$.
The direct sum $C^\infty(\Lambda T )=\oplus_{n\geq 0}\,
C^\infty(\Lambda^n T )$ is a graded algebra with respect to the
pointwise exterior multiplication $\wedge$. Here $
C^\infty(\Lambda^0 T )= C^\infty(M)$ is the algebra of smooth
$\RR$-valued functions on $M$ and $ C^\infty(\Lambda^1 T )=
C^\infty(T)$ is the vector space of smooth tangent vector fields on $M$.

  The \emph{Schouten--Nijenhuis bracket} on $M$ is the
$\RR$-bilinear map $$[-,-]: C^\infty(\Lambda T) \times
C^\infty(\Lambda T) \to C^\infty(\Lambda T)$$ uniquely determined by
the following conditions (see, for example,  \cite{Mar}):
\begin{enumerate}
\item for any $f,g \in C^\infty(M )$,   we have   $[f,g]=0$;
\item for any $X\in  C^\infty(T)$ and  $U \in C^\infty(\Lambda T)$, the bracket    $[X,U]$ is the Lie derivative of   $U$
along the vector field $X$;
\item\label{iki} for any $U \in C^\infty(\Lambda^k T)$,  $V \in C^\infty(\Lambda^l T)$   with $k,l\geq 0$   and     $W \in C^\infty(\Lambda
T)$,
  $$[U,V]= -(-1)^{(k-1)(l-1)}[V,U],$$ $$[U,V\wedge W]=
[U,V]\wedge W + (-1)^{(k-1)l}V\wedge [U,W].$$
\end{enumerate}

Set   $\mathcal A= C^\infty(M)$ and let $T^*=T^*(M)$ be the
cotangent bundle of $M$.   Consider the   algebra of differential
forms on $M$
$$C^\infty(\Lambda  T^*)= \oplus_{n\geq 0}C^\infty(\Lambda^n T^*)$$
with the exterior multiplication $\wedge$. There is a natural
$\mathcal A$-bilinear pairing $\langle -,-\rangle$ from
$C^\infty(\Lambda  T^*)\times C^\infty(\Lambda T) $ to $\mathcal A$
which is uniquely determined by the condition that $\langle\,
C^\infty(\Lambda^n T^*) , C^\infty(\Lambda^k T  )\, \rangle=0$ for
  $n\neq k$ and $$ \langle \alpha_1 \wedge \cdots \wedge \alpha_k,
X_1 \wedge \cdots \wedge X_k  \rangle = \det\, \left(\langle
\alpha_i, X_j\rangle \right)_{i,j=1}^k
$$
for all    differential 1-forms $\alpha_1, \ldots , \alpha_k$ and  vector
fields $X_1, \ldots , X_k$ on $M$ with $k\geq 0$. We   use the pairing $\langle -,-\rangle$
to  associate with  each $k$-vector field  $U\in C^\infty(\Lambda^k T )$
 a skew-symmetric multiderivation $\{-,\cdots, -\}_U:\mathcal A^k \to \mathcal A$   by
$$
\{f_1,\dots,f_k\}_U  =\langle df_1 \wedge \cdots \wedge df_k,
U\rangle
$$
for any $f_1,..., f_k\in \mathcal A$.

Suppose now that $M$ is endowed with a smooth (left) action of a Lie
group $G$. This
 induces an action of $G$ on the algebra ${\mathcal A}=C^\infty (M)$ by
\begin{equation}\label{G_on_A}
 (gf)(m)=f(g^{-1}m)
\end{equation}
 for  $g\in G$,  $f\in {\mathcal A}$, $m\in M$.
The   action of $G$ on $M$
 also induces an action of the Lie algebra, $\g$, of $G$ on ${\mathcal A}$ by
\begin{equation}\label{g_on_A}
(vf)(m)= \Big. \frac{d}{dt} \Big\vert_{ t=0}f(e^{-tv}m)
\end{equation}
for $ v\in \g$, $f\in {\mathcal A}$, $m\in M$. It is easy to check
that ${\mathcal A}$ is a $(G,\g)$-algebra  in the sense of Section \ref{QPA3}.  

For any $v\in \g$,  the derivation \eqref{g_on_A} of ${\mathcal
A}$  corresponds to a smooth tangent vector field on $M$  denoted by
$v_M$  and said to be \emph{generated} by $v$.  Note
that $([v,w])_M= [v_M,w_M]$ for any $v,w\in \g$. The linear map $\g
\to C^{\infty}(T), v \mapsto v_M$ extends uniquely to an algebra
homomorphism $\Lambda \g \to C^{\infty}(\Lambda T)$. In this way,
any $k$-vector $\Phi\in \Lambda^k\g$ with $k\geq 0$ \emph{generates}
a $k$-vector field on $M$ denoted $\Phi_M$. The action
\eqref{G_on_A}  extends to an action of $G$ on
$C^{\infty}(\Lambda T)$   by $g U = \Lambda (d \gamma_g) \circ
U \circ \gamma_{g^{-1}}$, for all $U\in C^{\infty}(\Lambda T)$ and
any  $g\in G$ whose action on $M$ is denoted  here 
by $\gamma_g:M \to M$. It is easy to check that the map
$\Lambda \g \to C^{\infty}(\Lambda T), \Phi \mapsto \Phi_M$ is
$G$-equivariant.  In particular,  a multivector field on $M$
generated by a $G$-invariant  vector in $\Lambda \g$ is also $G$-invariant. 

Suppose finally that the Lie algebra  $\g$ is  endowed with a
$G$-invariant non-degenerate symmetric bilinear form~$ \cdot : \g
\times \g \to \RR$ and consider the associated  Cartan trivector $\phi \in \wedge^3_{\g}$,
defined at \eqref{Cartan}. 
(Recall that $\wedge^3_{\g}$ is the vector space of 
$\g$-invariant skew-symmetric elements of $\g^{\otimes 3}$.)
By a \emph{quasi-Poisson structure} on $M$ we shall mean a $G$-invariant
bivector field $P\in C^\infty (\Lambda^2 T )$ such that the
Schouten--Nijenhuis bracket $[P,P] \in C^\infty (\Lambda^3 T )$ is
equal to the  $G$-invariant trivector field $(\iota^{-1}(2\phi ))_M=
2 (\iota^{-1}( \phi ))_M$ on $M$   generated   by 
$\iota^{-1}(2\phi )=2 \iota^{-1}( \phi )\in \Lambda^3 \g$. We call
such a $P$ a {\it quasi-Poisson bivector field} on $M$.   Note that,
given dual bases $(e_i)_i$ and $(e^\sharp_i)_i$ of $\g$ (so that
$e_i \cdot e^\sharp_j =\delta_{ij}$ for all $i,j)$, we have
$$
\iota^{-1}(\phi ) = \iota^{-1}\Big(    \sum_{i,j,k} 
(e_i \cdot [e_j,e_k])\,  e^\sharp_i \otimes e^\sharp_j \otimes e^\sharp_k \Big)
= \frac{1}{6} \sum_{i,j,k} (e_i \cdot [e_j,e_k]) \, e^\sharp_i \wedge e^\sharp_j \wedge e^\sharp_k.
$$

The following lemma   \cite{AKsM} shows that a quasi-Poisson
structure on~$M$ induces   on $C^\infty (M)$ a structure of a quasi-Poisson algebra  
associated with $\phi$, as defined in Section \ref{QPA3}. 

\begin{lemma}
Let $M$ be a smooth manifold endowed with a smooth action of a Lie
group $G$, whose Lie algebra $\g$ is equipped with a $G$-invariant
non-degenerate symmetric bilinear form. Let $\phi  \in
\wedge^3_{\g}$ be the associated Cartan trivector and ${\mathcal
A}=C^\infty (M)$. Then, the pairing $\bracket{-}{-}_P:{\mathcal
A}\times {\mathcal A}\to {\mathcal A}$   determined by a
quasi-Poisson bivector field $P$ on $M$ as above is a quasi-Poisson
bracket   in ${\mathcal A}$ associated with $\phi$.
\end{lemma}

\begin{proof}
Recall that any  $k$-vector field $U\in C^\infty(\Lambda^k T )$
 determines an   $\mathcal A$-linear map \begin{equation*}\label{mapiU} i_U: C^\infty(\Lambda T^*)\to
C^\infty( \Lambda T^*)\end{equation*} of degree $-k$ called the {\it
right interior product by $U$}, see, for instance, \cite{Mar}. The
map $i_U$ is uniquely determined by the condition that $\langle i_U
\alpha, X \rangle= \langle \alpha, U \wedge X \rangle$ for all
$\alpha \in C^\infty( \Lambda^n T^*)$ with $n\geq k$ and all $X \in
C^\infty( \Lambda^{n-k} T )$. By \cite[Proposition 4.1]{Mar},
\begin{equation}\label{Koszul}
i_{[U,V]} = - \left[\left[i_V,d\right],i_U\right]
\end{equation}
for any   $U \in C^\infty( \Lambda^k T )$ and $V \in C^\infty(
\Lambda^l T )$ with $k,l\geq 0$. On the right-hand side of
\eqref{Koszul}, the bracket $[\varphi,\psi]$ of two graded
endomorphisms $\varphi$ and $\psi$ of $C^\infty(\Lambda T^*)$ stands
for the graded commutator $\varphi \psi-(-1)^{\deg(\varphi)\deg(\psi)}\psi\varphi$.  
By \eqref{Koszul}, we have for any
$f,g,h \in {\mathcal A}$,
\begin{eqnarray*}
\{f,g,h\}_{[P,P]}&=&  \langle df \wedge dg \wedge dh, [P,P] \rangle\\
&=&   i_{[P,P]}   (df \wedge dg \wedge dh)\\
&=& -[i_P d - d i_P,i_P](df \wedge dg \wedge dh)\\
&=& -(i_P di_P - d  i_P i_P -  i_P i_P d + i_Pd i_P  )    (df \wedge dg \wedge dh)\\
&=& -2 i_P di_P (df \wedge dg \wedge dh).
\end{eqnarray*}
Here we use that $d (df \wedge dg \wedge dh)=0$ and   that $i_P i_P$
decreases the degree by $4$ and therefore annihilates $df \wedge dg
\wedge dh$. Expanding  the determinant of a $(3\times 3)$-matrix
with respect to the last row, we obtain that for any vector field
$X$ on $M$,
$$
\langle df \wedge dg \wedge dh, P \wedge X \rangle
= \langle df\wedge dg , P\rangle \langle dh, X\rangle -  \langle df\wedge dh, P \rangle \langle dg, X\rangle
+ \langle dg \wedge dh, P\rangle \langle df, X\rangle
.$$
  Therefore  $$i_P(df \wedge dg \wedge dh)= \bracket{f}{g}_P dh +
\bracket{h}{f}_P dg  + \bracket{g}{h}_P df.$$ We deduce that
\begin{eqnarray}
\label{Schouten-Jacobi} \quad \quad \quad \{f,g,h\}_{[P,P]}
&=& -2 i_P\left (d \bracket{f}{g}_P \wedge dh + d\bracket{h}{f}_P\wedge dg  + d \bracket{g}{h}_P \wedge df \right)\\
\notag &=& 2 \left( \bracket{h}{\bracket{f}{g}_P}_P + \bracket{g}{\bracket{h}{f}_P}_P + \bracket{f}{\bracket{g}{h}_P}_P   \right).
\end{eqnarray}
The assumption $[P,P]=2(\iota^{-1}(\phi ))_M$ implies that
$\bracket{-}{-}_P$ is a quasi-Poisson bracket in ${\mathcal A}$
associated with $\phi $.
\end{proof}

  \subsection{A quasi-Poisson structure on $\mathcal H$}\label{Specs} Let   $\Sigma$
  be a compact  connected   oriented surface with non-empty boundary and base point $\ast\in \partial
  \Sigma$. Then the group $\pi=\pi_1(\Sigma, \ast)$ is free of   a
finite rank  $n\geq 0$. We fix an integer $N\geq 1$ and consider the
smooth manifold
$$
\mathcal H   =\Hom (\pi, G_N) \cong  \left(G_N\right)^n \quad {\rm {where}} \quad G_N=\GL_N(\RR).
$$
  The group $G_N$ acts on $\mathcal H$ by conjugations in
the obvious way.   The aim of this subsection is to construct a
natural  quasi-Poisson structure on $\mathcal H$.

Set $A=\RR\pi$ and recall the algebra $A_N$ from
Section~\ref{sectionalgebrasmaintheorem}. By Section \ref{AMT--1},
we have an evaluation map $\ev=\ev_{\RR} $ from $A_N$ to the algebra
$\Map(\mathcal H,\RR)$ of $\RR$-valued functions on $\mathcal H$. To
describe $\ev$ in concrete terms, pick a basis
  $\{x_u\}_{u=1}^n$   
of~$\pi$ and use it to  identify $\mathcal H = (G_N)^n$.
Consider the presentation of $A_N$ by generators and relations
discussed at the end of Section~\ref{AMT--1}. The map $\ev$ carries
the generator $x^u_{ij}$ (respectively, $y^u$) of $A_N$ to the
function $\mathcal H = (G_N)^n\to \RR$ sending a sequence of $n$
invertible matrices to the $(i,j)$-th entry of the $u$-th matrix
 (respectively, to the inverse of the determinant of the $u$-th matrix).
We can view $\mathcal H$ as the affine algebraic set defined by the equations
$\tilde y^u\det(\tilde x_{ij}^u)_{i,j}=1$ in the affine space $(\Mat_N(\RR) \times \RR)^n = \RR^{(N^2+1)n}$
with coordinates  $((\tilde x_{ij}^u)_{i,j},\tilde y^u)_{u}$.
The  coordinate algebra $\RR[\mathcal H]$ of $\mathcal H$ is the quotient of $\RR\big[((\tilde x_{ij}^u)_{i,j},\tilde y^u)_{u}\big]$
by the ideal generated by the polynomials $\tilde y^u\det(\tilde x_{ij}^u)_{i,j}-1$ for $u\in\{1,\dots,n\}$. 
Restricting polynomial functions $\RR^{(N^2+1)n} \to \RR$ to  $\mathcal H\subset \RR^{(N^2+1)n}$, 
we  identify $\RR[\mathcal H]$ with a subalgebra of $C^\infty (\mathcal H )\subset \Map(\mathcal H,\RR)$.
Then the   description of $\ev$ above shows that $\ev$  is an isomorphism from $A_N$ onto $\RR[\mathcal{H}]$.
Transporting along $\ev$ the quasi-Poisson bracket in $A_N$ provided by
Theorem~\ref{CQPS}   we obtain a skew-symmetric bracket $\{-,-\}$ in
$\RR[\mathcal H]$ which   is a derivation in each variable.

\begin{theor}\label{CQPShahaha}
There is a unique quasi-Poisson structure  on $\mathcal H$ such that
the associated bracket in   $C^\infty (\mathcal H )$ extends the
bracket $\{-,-\}$ in   $\RR[\mathcal H]$. This quasi-Poisson
structure is invariant under the action of $\Homeo(\Sigma, \ast)$
on~$ \mathcal H $.
\end{theor}

\begin{proof}  We claim that there is a  unique   bivector field  $P\in C^\infty
( \Lambda^2 T{\mathcal H})$  such that $ \{f,g\}= \{f,g\}_P=
\langle df\wedge dg,P\rangle$ for  all $f,g\in \RR[\mathcal H]$.
Indeed, pick a point $m\in \mathcal H$ and let $T_m \mathcal H$ be
the tangent space of the smooth manifold $\mathcal H$ at $m$.
Consider the maximal ideal $I_m$ of $\RR[\mathcal H] $ consisting of
all $f \in \RR[\mathcal H]$ such that $f(m)=0$. Recall that at
any smooth point of an affine algebraic set,  the Zariski tangent
space may be identified with the tangent space of differential
topology. Specifically,   there is a non-degenerate bilinear form $
(I_m/I_m^2)\times T_m\mathcal H \to \RR$ defined by
$(f,v)\mapsto d_mf(v)$ where $f\in I_m$ and $v\in T_m \mathcal H$.
On the other hand, the map $(f,g)\mapsto \{f,g\} (m)$, where $f,g\in
I_m$, induces a skew-symmetric bilinear form $(I_m/I_m^2)\times
(I_m/I_m^2)  \to \RR$. Therefore there is a unique bivector $P_m\in
\Lambda^2 T_m\mathcal H$ such that $\{f,g\} (m)= \langle d_mf \wedge
d_mg, P_m\rangle$ for all $f,g\in I_m$. In fact, the latter formula
holds for all $f,g\in \RR[\mathcal H]$ because $\{1,-\}=\{-,1\}=0$.
The map $m\mapsto P_m$ defines a unique section $P$ of the bundle
$\Lambda^2 T{\mathcal H}$ over $\mathcal H$ such that $ \{f,g\} =
\langle df\wedge dg,P \rangle$ for all $f,g\in \RR[\mathcal H]$. It
remains to justify that  $P\in C^\infty (  \Lambda^2 T{\mathcal
H})$,   i.e., that $P$ is a smooth section. For this, we
consider the  functions $\tilde x^u_{ij}= \ev(x^u_{ij})\in
\RR[\mathcal H]$  as above. These functions form a local system of
coordinates in a neighborhood of any point of $  \mathcal H$,
and we can expand
$$
P = \frac{1}{2}  \sum_{u,v,i,j,k,l} \langle d\tilde x^u_{ij} \wedge d \tilde x^v_{kl}, P\rangle\
\frac{\partial \ }{\partial \tilde x^u_{ij}} \wedge \frac{\partial  \ }{\partial \tilde x^v_{kl}}
=  \frac{1}{2} \sum_{u,v,i,j,k,l}   \{{\tilde x}^u_{ij} , {\tilde x}^v_{kl}\} \
\frac{\partial \ }{\partial \tilde x^u_{ij}} \wedge \frac{\partial  \ }{\partial \tilde x^v_{kl}}\, .
$$
 Since the functions $  \{{\tilde x}^u_{ij} , {\tilde
x}^v_{kl}\} \in \RR[\mathcal H]  $ are smooth,       $P$ is a smooth
 bivector field   on $\mathcal H$ such that the  bracket
$\bracket{-}{-}_P$ in $C^\infty (\mathcal H )  $ extends the bracket
$\bracket{-}{-}$ in $\RR[\mathcal H] $.

We verify now that $P$ is $G_N$-invariant. Observe   that the action
\eqref{G_on_A} of $G_N$ on $C^\infty (\mathcal H )$ preserves the
subalgebra $\RR[\mathcal H] $. By Section \ref{AMT--1}, the
evaluation map $\ev$ from the $(G_N, \g_N)$-algebra $ A_N$ to
$C^\infty (\mathcal H )$ is $G_N$-equivariant. This   and the
$G_N$-invariance of the quasi-Poisson bracket in $A_N$ imply the
$G_N$-invariance of the bracket $\{-,-\}$ in $\RR[\mathcal H]$. The
latter implies the $G_N$-invariance of $P$.

To see that $P$ defines a quasi-Poisson structure on $\mathcal H$,
we need only to check the modified Jacobi identity.  Observe first
that the action \eqref{g_on_A} of $\g_N$ on $C^\infty (\mathcal H )$
preserves $\RR[\mathcal H] $ and the evaluation map $\ev: A_N \to
C^\infty (\mathcal H )$ is $\g_N$-equivariant.  It suffices to
check   the $\g_N$-equivariance of $\ev$  on each generator of $A_N$
of the form $x^u_{ij}$ with $u\in\{1,\dots,n\}$ and
$i,j\in\{1,\dots,N\}$. For any  $w\in \g_N$ and  any   point
$m=(m_v)_{v=1}^n \in \mathcal H = (G_N)^n$, we have
\begin{eqnarray*}
\left(w \ev(x_{ij}^u)\right)(m) &= & \left(w  \tilde x_{ij}^u\right)(m) \\
& \stackrel{\eqref{g_on_A}}{=} &
 \frac{d}{dt} \Big\vert_{ t=0}  \tilde x_{ij}^u (e^{-tw}m)\\
 &=&  \frac{d}{dt} \Big\vert_{ t=0}  \tilde x_{ij}^u (e^{-tw}m_1 e^{tw},\dots, e^{-tw}m_n e^{tw})\\
 &=& \frac{d}{dt} \Big\vert_{ t=0}   (e^{-tw}m_u e^{tw})_{i,j}\\
 & =& ( m_uw-wm_u)_{i,j} \\
 &=&  (m_u)_{i,s} w_{s,j} -w_{i,s} (m_u)_{s,j} \\
 &=& (\tilde x^u_{is} w_{s,j} -w_{i,s} \tilde x^u_{sj} )(m)
 \quad \stackrel{\eqref{gN_on_AN}}{=} \quad \ev (w x^u_{ij})(m).
\end{eqnarray*}
Since our   bracket in $A_N$  is a quasi-Poisson bracket  associated
with $\phi_N$, so is the bracket $\bracket{-}{-}$ in  $\RR[\mathcal
H]$. Thus, for any $f,g,h \in \RR[\mathcal H] $,
\begin{eqnarray*}
\langle df \wedge dg \wedge dh, [P,P] \rangle  &=& \{f,g,h\}_{[P,P]}\\
\notag  &\stackrel{\eqref{Schouten-Jacobi}}{=} &
 2\left( \bracket{f} {\bracket{g }{h}_P}_P  + \bracket{g}{\bracket{h }{f}_P}_P +  \bracket{h}{\bracket{f}{g}_P}_P \right)\\
\notag &=&  2\left( \bracket{f} {\bracket{g }{h}}  + \bracket{g}{\bracket{h }{f}} +  \bracket{h}{\bracket{f}{g}} \right)\\
\notag & =&  2 \phi_N(f,g,h) \\
\notag & =&   2 \left \langle df
\wedge dg \wedge dh, \left(\iota^{-1}(\phi_N )\right)_{\mathcal H} \right\rangle  
\end{eqnarray*}
  where
  $\left(\iota^{-1}(\phi_N )\right)_{\mathcal H}$   is  the  trivector field on
${\mathcal H} $ induced by $\iota^{-1}(\phi_N)\in \Lambda^3\g_N$.
Since the differentials of regular functions fill in the  cotangent
space,   $[P,P]=  2\left(\iota^{-1}(\phi_N)\right)_{\mathcal H}$. 

The invariance of $P$ under the action of $\Homeo(\Sigma, \ast)$
follows from the corresponding property of the quasi-Poisson bracket
in $A_N$ and the fact that the evaluation map  preserves this
action.
\end{proof}

\subsection{Computations} \label{compact}

We   give   explicit formulas for the quasi-Poisson brackets associated  
 with   a compact connected oriented surface  $\Sigma$ of genus $g\geq 0$ with $m+1\geq 1$ boundary components.
 Fix  a  basis $(p_1,q_1,\dots,p_g,q_g,z_1,\dots,z_m)$  of $\pi=\pi_1(\Sigma, \ast)$  as  shown on Figure \ref{surface}.   

\begin{figure}
\labellist \small \hair 2pt
\pinlabel {$p_1$} [bl] at 224 81
\pinlabel {$p_g$} [bl] at 477 81
\pinlabel {$q_1$} [bl] at 345 55
\pinlabel {$q_g$} [bl] at 554 45
\pinlabel {\Large $\circlearrowleft$} at 631 156
\pinlabel {$\cdots$} at 355 90
\pinlabel {\large $\ast$}  at 567 5
\pinlabel {$\cdots$} at 720 90
\pinlabel {$z_1$} [b] at 650 105
\pinlabel {$z_m$} [b] at 790 105
\endlabellist
\centering
\includegraphics[scale=0.4]{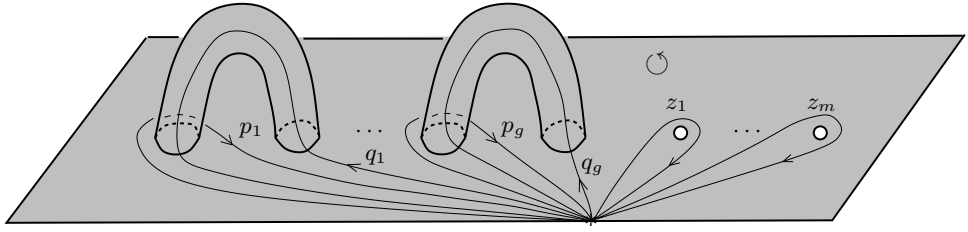}
\caption{A   surface $\Sigma$ and a basis of $\pi$.}
\label{surface}
\end{figure}

We first compute the double bracket $\double{-}{-}^s$ in $ \kk\pi$. 
For any $u,v\in \{1,\dots,m\}$ such that
$u<v$,  we have $\double{z_u}{z_v}^\eta=0$ so that
$$
  \double{z_u}{z_v}^s
= 1\otimes z_u z_v + z_v z_u \otimes 1 - z_u \otimes z_v - z_v \otimes z_u
$$
and, for any $u\in \{1,\dots,m\}$, we have $ \double{z_u}{z_u}^\eta =  z_u \otimes z_u -1 \otimes z_u^2$ so that
$$
\double{z_u}{z_u}^s =   z_u^2 \otimes 1 -1 \otimes z_u^2 .
$$
For any $u,v\in \{1,\dots,g\}$ with $u<v$ and   any    $a,b \in \{p,q\}$, we
have $\double{a_u}{b_v}^\eta=0$ so that
$$
\double{a_u}{b_v}^s
= 1\otimes a_u b_v + b_v a_u \otimes 1 - a_u \otimes b_v - b_v \otimes a_u.
$$
For any $u\in \{1,\dots,g\}$,  we have
$$\double{p_u}{p_u}^\eta= p_u \otimes p_u -1 \otimes p_u^2, \quad  \double{q_u}{q_u}^\eta= q_u
\otimes q_u -q_u^2 \otimes 1, \quad \double{p_u}{q_u}^\eta = q_u \otimes p_u$$   so that  
$$
\double{p_u}{p_u}^s =   p_u^2 \otimes 1 -1 \otimes p_u^2, \quad
\double{q_u}{q_u}^s =  1 \otimes q_u^2 - q_u^2 \otimes 1,
$$
and
$$
\double{p_u}{q_u}^s = 1\otimes p_u q_u + q_u p_u \otimes 1 - p_u \otimes q_u + q_u \otimes p_u.
$$ 
For any $u\in \{1,\dots,g\}$      and    $v\in\{1,\dots,m\}$,
we have $\double{p_u}{z_v}^\eta= 0$ so that
$$
\double{p_u}{z_v}^s
= 1\otimes p_u z_v + z_v p_u \otimes 1 - p_u \otimes z_v - z_v \otimes p_u.
$$
For any $u\in \{1,\dots,g\}$     and    $v\in\{1,\dots,m\}$,
we have $\double{q_u}{z_v}^\eta= 0$ so that
$$
\double{q_u}{z_v}^s
= 1\otimes q_u z_v + z_v q_u \otimes 1 - q_u \otimes z_v - z_v \otimes q_u.
$$

We now fix an $N\geq 1$.
By Section \ref{AMT--1}, given a   basis   $(x_1,x_2,\dots,x_{2g+m})$ of~$\pi$,  the quasi-Poisson bracket $\bracket{-}{-}$ in $(\kk\pi)_N$
produced by Theorem \ref{CQPS}  is determined by its values on the
generators $x^u_{ij}=(x_u)_{ij}$ with $u\in \{1,\dots,2g+m\}$ and $i,j\in \{1,\dots,N\}$. Applying this fact  to the chosen basis of $\pi$,
we deduce from the computations above  the following formulas which hold for arbitrary $i,j,k,l \in \{1,\dots,N\}$. For any $u,v\in\{1,\dots,m\}$ with $u<v$,
\begin{equation}\label{values_1}
\bracket{z_{ij}^u}{z_{kl}^v} =
\delta_{kj}  z^u_{ir} z^v_{rl} + z^v_{ks} z^u_{sj} \delta_{il} - z^u_{kj} z^v_{il} - z^v_{kj}  z^u_{il}.
\end{equation}
For any $u\in\{1,\dots,m\}$,
\begin{equation}\label{values_2}
\bracket{z_{ij}^u}{z_{kl}^u} = z^u_{kr} z^u_{rj} \delta_{il} - \delta_{kj}  z^u_{is} z^u_{sl}.
\end{equation}
For any $u,v\in \{1,\dots,g\}$ with
$u<v$ and  any $a,b\in\{p,q\}$,
\begin{equation}\label{values_one}
\bracket{a_{ij}^u}{b_{kl}^v} =
\delta_{kj}  a^u_{ir} b^v_{rl} + b^v_{ks} a^u_{sj} \delta_{il} - a^u_{kj} b^v_{il} - b^v_{kj}  a^u_{il}.
\end{equation}
For  any $u\in \{1,\dots,g\}$,
\begin{eqnarray}
\label{values_two}
\bracket{p_{ij}^u}{p_{kl}^u} &=& p^u_{kr} p^u_{rj} \delta_{il} - \delta_{kj}  p^u_{is} p^u_{sl};\\
\label{values_three}
\bracket{q_{ij}^u}{q_{kl}^u} &=&  \delta_{kj}  q^u_{is} q^u_{sl} - q^u_{kr} q^u_{rj} \delta_{il} ;\\
\label{values_four}
\bracket{p_{ij}^u}{q_{kl}^u} &=&
\delta_{kj}  p^u_{ir} q^u_{rl} + q^u_{ks} p^u_{sj} \delta_{il} - p^u_{kj} q^u_{il} + q^u_{kj}  p^u_{il}.
\end{eqnarray}
For any $u\in \{1,\dots,g\}$  and $v\in\{1,\dots,m\}$,
\begin{eqnarray}\label{values_mixed}
\bracket{p^u_{ij}}{z^v_{kl}} &=&
\delta_{kj}  p^u_{ir} z^v_{rl} + z^v_{ks} p^u_{sj} \delta_{il} - p^u_{kj} z^v_{il} - z^v_{kj}  p^u_{il};\\
\label{values_mixed_bis}
\bracket{q^u_{ij}}{z^v_{kl}} &=& \delta_{kj}  q^u_{ir} z^v_{rl} + z^v_{ks} q^u_{sj} \delta_{il} - q^u_{kj} z^v_{il} - z^v_{kj}  q^u_{il}.
\end{eqnarray}

For $\kk=\RR$, we can use a basis $(x_1,x_2,\dots,x_{2g+m})$ of $\pi$ to identify the
  space $\mathcal H= \Hom(\pi, \GL_N(\RR))$ with the
open subset $(\GL_N(\RR))^n$ of $(\Mat_N(\RR))^n=\RR^{nN^2}$. The  functions $( \tilde x^u_{ij})_{u,i,j}$  form a system of smooth
coordinates on $\mathcal H$ 
where  $u\in \{1,\dots,2g+m\}$,  $i,j\in \{1,\dots,N\}$, and $\tilde x^u_{ij} :\mathcal H \to \RR$ carries a
 tuple of $2g+n$ matrices to the $(i,j)$-th entry of the $u$-th matrix.
  In the coordinates derived in this way from the chosen basis of $\pi$,     the quasi-Poisson structure
on~$\mathcal{H}$ provided by Theorem \ref{CQPShahaha} is determined 
by the bivector field 
\begin{eqnarray*}
P &=&  \sum_{\substack{u<v\\ i,j,k,l}} \bracket{\tilde  z^u_{ij}}{\tilde  z^v_{kl}}
\frac{\partial \ }{\partial \tilde z^u_{ij}}  \wedge \frac{\partial \ }{\partial \tilde z^v_{kl}}
+ \frac{1}{2} \sum_{\substack{u,i,j,k,l}} \bracket{\tilde  z^u_{ij}}{\tilde  z^u_{kl}}
\frac{\partial \ }{\partial \tilde z^u_{ij}}  \wedge \frac{\partial \ }{\partial \tilde z^u_{kl}}\\
&&+ \sum_{\substack{u<v\\ i,j,k,l}}
\left( \bracket{\tilde  p^u_{ij}}{\tilde  p^v_{kl}} 
\frac{\partial \ }{\partial \tilde p^u_{ij}}  \wedge \frac{\partial \ }{\partial \tilde p^v_{kl}}+
\bracket{\tilde  q^u_{ij}}{\tilde  q^v_{kl}} \frac{\partial \ }{\partial \tilde q^u_{ij}}  \wedge \frac{\partial \ }{\partial \tilde q^v_{kl}}\right)\\
&&+ \sum_{\substack{u<v\\ i,j,k,l}}\left( \bracket{\tilde  p^u_{ij}}{\tilde  q^v_{kl}}
 \frac{\partial \ }{\partial \tilde p^u_{ij}}  \wedge \frac{\partial \ }{\partial \tilde q^v_{kl}}+
\bracket{\tilde  q^u_{ij}}{\tilde  p^v_{kl}} \frac{\partial \ }{\partial \tilde q^u_{ij}}  \wedge \frac{\partial \ }{\partial \tilde p^v_{kl}}\right)\\
&& +\frac{1}{2}  \sum_{\substack{u,i,j,k,l}}\left(    \bracket{\tilde  p^u_{ij}}{\tilde  p^u_{kl}}
\frac{\partial \ }{\partial \tilde p^u_{ij}}  \wedge \frac{\partial \ }{\partial \tilde p^u_{kl}}
+   \bracket{\tilde  q^u_{ij}}{\tilde  q^u_{kl}}
\frac{\partial \ }{\partial \tilde q^u_{ij}}  \wedge \frac{\partial \ }{\partial \tilde q^u_{kl}} \right)\\
&& +  \sum_{\substack{u,i,j,k,l}} \bracket{\tilde  p^u_{ij}}{\tilde  q^u_{kl}}
\frac{\partial \ }{\partial \tilde p^u_{ij}} \wedge \frac{\partial \ }{\partial \tilde q^u_{kl}}\\
&& + \sum_{\substack{u,v\\i,j,k,l}}\left( \bracket{\tilde  p^u_{ij}}{\tilde  z^v_{kl}}
\frac{\partial \ }{\partial \tilde p^u_{ij}}  \wedge \frac{\partial \ }{\partial \tilde z^v_{kl}} 
+ \bracket{\tilde  q^u_{ij}}{\tilde  z^v_{kl}}
\frac{\partial \ }{\partial \tilde q^u_{ij}}  \wedge \frac{\partial \ }{\partial \tilde z^v_{kl}} \right).
\end{eqnarray*} 
Here  the   brackets $\bracket{-}{-}$ are computed from
\eqref{values_1} -- \eqref{values_mixed_bis} by replacing all $p$'s, $q$'s and $z$'s 
with $\tilde p$'s, $\tilde q$'s and $\tilde z$'s respectively.

 \subsection{Remarks}\label{qP-manifolds}
 1. The restriction of the quasi-Poisson bracket $\{-,-\}$ in $C^\infty (\mathcal H)$ to
 the subalgebra  of
$G_N$-invariant elements $C^\infty (\mathcal H
 )^{G_N}$ is a Poisson bracket, and the map  $\ev\circ \tr:
\RR\check \pi \to \RR[\mathcal H]^{G_N} \subset C^\infty (\mathcal H )^{G_N}$   
carries the Goldman bracket in~$\RR\check
\pi$ into $(1/2) \{-,-\}$.    The restriction of $\{-,-\}$ to
$\RR[\mathcal H]^{G_N}$ can be uniquely recovered from the Goldman
bracket because the set  $\ev \left(\tr(  \check \pi)\right)$   generates the algebra
$\RR[\mathcal H]^{G_N}$.

2. Our definition of a quasi-Poisson structure on a manifold differs
from that in \cite{AKsM} by a factor of 2: a quasi-Poisson bivector
field   in our sense is $2$ times a quasi-Poisson bivector field  in
the sense of \cite{AKsM}.

3.  Under the assumptions of Section \ref{compact},  the isomorphism class of $(\mathcal H,P)$ (in the category of quasi-Poisson manifolds) does not depend on the choice of the base point $\ast \in \partial \Sigma$. This follows from the naturality of our bracket under surface homeomorphisms  and the fact that for any points $\ast, \ast' \in \partial \Sigma$, there is an orientation-preserving self-homeomorphism of $\Sigma$ carrying $\ast$ to $\ast'$.

\section{Moment maps  and surfaces without boundary}\label{TBE+}

We   discuss moment maps and, as an application,   associate certain
Poisson algebras with surfaces without boundary.

 \subsection{Moment maps}
Motivated by the notion of a multiplicative moment map for a
quasi-Poisson manifold \cite{AKsM}, Van den Bergh \cite{VdB} defined
a    similar notion   for a  quasi-Poisson  double bracket $\double{-}{-}$
 in   any algebra $A$. In our notation,  his definition may be
reformulated as follows: a \emph{moment map} for $\double{-}{-}$ is
an   invertible   element $\mu\in A$ such that for all $a\in A$,
\begin{equation}\label{mommom}
 \double{\mu}{a}= a\otimes \mu + a\mu \otimes 1 -\mu \otimes a - 1 \otimes \mu a .
\end{equation}

Let $N \geq 1$ be an integer.
Formula \eqref{mommom} allows us to compute the bracket
$ \bracket{(\mu^m)_{ij}}{a_{uv}} \in A_N$ for any $a\in A$,   $m\in \ZZ$, and
  $i,j,u,v\in \{1,\dots, N\}$.  
For   $m>0$,   
\begin{eqnarray*}
\double{\mu^m}{a} &=& \double{\mu^{m-1}\mu}{a} \\
&=& \mu^{m-1}* \double{\mu}{a} + \double{\mu^{m-1}}{a}*\mu\\
&=& a\otimes \mu^m + a\mu \otimes \mu^{m-1} - \mu \otimes \mu^{m-1} a - 1 \otimes \mu^m  a+ \double{\mu^{m-1}}{a}*\mu\\
&=& \sigma_{0,m}(\mu,a)+  \sigma_{1,m-1}(\mu,a) + \double{\mu^{m-1}}{a}*\mu
\end{eqnarray*}
where    $\sigma_{k,m-k}(\mu, a)= a \mu^k \otimes \mu^{m-k}- \mu^k \otimes \mu^{m-k}a$ for any $k\in\{0,\dots,m\}$. 
Since $\sigma_{k,m-1-k}(\mu, a)*\mu= \sigma_{k+1,m-(k+1)}(\mu, a)$, we deduce  recursively  that
\begin{equation}\label{mu^m}
\double{\mu^m}{a} =  \sigma_{0,m}(\mu, a) +  \sigma_{m,0}(\mu, a) + 2 \sum_{k=1}^{m-1} \sigma_{k,m-k}(\mu, a).
\end{equation}
Therefore 
\begin{eqnarray}
\label{mu^m_bis}\quad \quad \bracket{(\mu^m)_{ij}} {a_{uv}} &=&
 a_{uj} (\mu^m)_{iv} - \delta_{uj} (\mu^m a)_{iv}+    (a\mu^m)_{uj} \delta_{iv} - (\mu^m)_{uj} a_{iv}  \\
\notag && +2 \sum_{k=1}^{m-1} \left( (a\mu^{k})_{uj} (\mu^{m-k})_{iv} - (\mu^{k})_{uj} (\mu^{m-k} a)_{iv}  \right).
\end{eqnarray}
Similar formulas hold for the $(-m)$-th power of $\mu$ (where $m>0$). Set  $\overline \mu=\mu^{-1}\in A$. 
Using the equalities $\double{\overline \mu^m}{a}=- \overline{\mu}^m*\double{\mu^m}{a}*\overline{\mu}^{m}$
and $$\overline{\mu}^m*\sigma_{k,m-k}(\mu,a)*\overline{\mu}^m= \sigma_{m-k,k}(\overline{\mu},a),$$
we deduce from \eqref{mu^m} that
\begin{equation}\label{mu^-m}
\double{\overline \mu^m}{a} =  
-\sigma_{0,m}(\overline \mu, a) -  \sigma_{m,0}(\overline \mu, a) - 2 \sum_{k=1}^{m-1} \sigma_{k,m-k}(\overline \mu, a).
\end{equation}
 Hence 
\begin{eqnarray}
\label{mu^-m_bis} \quad \quad \bracket{(\overline \mu^m)_{ij}} {a_{uv}}&=&
- a_{uj} (\overline\mu^m)_{iv} + \delta_{uj} (\overline \mu^m a)_{iv} - (a\overline \mu^m)_{uj} \delta_{iv} 
+(\overline \mu^m)_{uj} a_{iv}  \\
\notag && +2 \sum_{k=1}^{m-1} \left( -(a\overline \mu^{k})_{uj} (\overline \mu^{m-k})_{iv} 
+ (\overline \mu^{k})_{uj} (\overline \mu^{m-k} a)_{iv}  \right).
\end{eqnarray}
These computations imply,  in particular,   that the subalgebra of $A_N$ generated by the  $(\mu^{ m} )_{ij}$'s
(with $m \in\ZZ$ and $i,j\in\{1,\dots,N\}$)   is closed under the bracket $\bracket{-}{-}$.

The next lemma gives another property of   moment  maps   used below.

\begin{lemma}\label{quotient}
Let $\double{-}{-}$ be a quasi-Poisson double bracket in an algebra
$A$ and let $\langle -,-\rangle: \check A \times \check A \to \check
A$ be the induced Lie bracket.
 Let $\mu \in A$ be a moment map,
let $A'=A/A(\mu -1)A$   be the quotient of $A$ by the two-sided
ideal generated by $ \mu-1 $ and let $\check A'= A'/[A',A']$.
 Let $p:A\to A'$ be the projection and
  $\check p:\check A\to \check A'$ be the induced map. Then, for any
$N\geq 1$,
  there is a unique Poisson bracket $\{-,-\}'$  in the subalgebra $(A')^t_N$   of
$A'_N$ generated by   $\tr(\check A' )\subset  A'_N$ such that   the
following diagram   commutes:
$$
\xymatrix@R=0.8cm @C=3cm {
\check A  \otimes \check A
\ar[d]^-{\langle   - ,  -\rangle }\ar[r]^-{ \check p \otimes  \check p}& \check A'\otimes \check A'  \ar[r]^-{ \tr \otimes  \tr}
& (A')^t_N \otimes (A')^t_N \ar[d]^-{\bracket{-}{-}' }   \\
\check A  \ar[r]^-{ \check p }& \check A'  \ar[r]^-{ \tr} & (A')^t_N \, .
}
$$
\end{lemma}

\begin{proof}
 Set $T=(A')^t_N$.   The uniqueness of   a  Poisson bracket  in
$T$ satisfying the conditions of the lemma  is obvious because by
the Leibniz rule, the values of such a bracket on generators of $T$
 determine the bracket. The existence can be deduced from
\cite[Proposition 5.1.5]{VdB} and \cite[Theorem 4.5]{Cb} using
$H_0$-Poisson structures,  but we rather give a direct proof.
Consider the algebra  map   $p_N:A_N\to A'_N$ induced by $p$
and the algebras $B'=(A'_N)^{\g_N} \subset A'_N$ and
$B=p_N^{-1}(B')\subset A_N$.
 The following two claims are verified
below:

(i)  $\{B,B\}\subset B$ where $\{-,-\}$ is the quasi-Poisson bracket
in $A_N$ induced by the  quasi-Poisson double  bracket
$\double{-}{-}$ in $A$, and

(ii) there is a unique bilinear pairing $\bracket{-}{-}':B' \times
B' \to B'$ such that the following diagram   commutes:
\begin{equation}\label{diagBB}
\xymatrix@R=0.8cm @C=3cm {
B  \otimes B
\ar[d]_-{\bracket{-}{-}}\ar[r]^-{p_N \otimes p_N} & B'\otimes B' \ar[d]^-{\bracket{-}{-}'} \\
B \ar[r]^-{p_N}  & B'.
}
\end{equation}

Then the properties of $\bracket{-}{-}$ and the triviality of the
action of $\g_N$ on $B' $ imply that $\{-,-\}'$ is a Poisson bracket
in $B'$. Clearly, $T \subset B'$. For any $a,b\in \check A$, we have
\begin{eqnarray*}
\bracket{\tr \check p(   a)}{\tr \check p(   b)}' &=&
\bracket{  p_N\tr (    a)}{ p_N \tr(   b)}' \ = \ p_N \bracket{    \tr (    a)}{   \tr(b)}\\
 &=& p_N \tr \left( \left\langle    a , b\right\rangle\right)
 \ = \ \tr \check p\left( \left\langle    a , b\right\rangle\right)
\end{eqnarray*}
where we use  the commutativity of the diagram \eqref{diagBB} and Lemma~\ref{comparis}.
Thus we have  $\{T,T\}'\subset T$, and the restriction of $\{-,-\}'$
to $T\subset B'$ is a Poisson bracket in $T$ satisfying the
conditions of the lemma.

 To prove the claims (i), (ii), we observe
that for all $y\in A_N$ and $i,j\in \{1,\dots,N\}$,
\begin{equation}\label{mu_action}
p_N \bracket{(\mu-1)_{ij}}{y} = 2 p_N(f_{ji} y)
\end{equation}
where $f_{ij}\in \g_N$ is the elementary matrix whose $(i,j)$-th
entry is 1 and the other entries are 0. Since both sides of
\eqref{mu_action} are derivations in $y $, it is enough to check
it for each generator $y=a_{uv}$ where $a\in A$ and $u,v\in
\{1,\dots,N\}$. We have
\begin{eqnarray*}
p_N \bracket{(\mu-1)_{ij}}{a_{uv}} &= &p_N\left(\double{\mu-1}{a}_{uj}^{(1)}  \double{\mu-1}{a}_{iv}^{(2)} \right)\\
&=& p_N\left( \double{\mu}{a}_{uj}^{(1)}  \double{\mu}{a}_{iv}^{(2)}\right)\\
&=& p_N\left(a_{uj} \mu_{iv} + (a\mu)_{uj} \delta_{iv} - \mu_{uj} a_{iv} - \delta_{uj} (\mu a)_{iv}\right)\\
&=& p_N(a_{uj}) \delta_{iv} + p_N(a_{uj}) \delta_{iv} - \delta_{uj} p_N(a_{iv}) - \delta_{uj} p_N(a_{iv})\\
&=& 2 p_N(f_{ji} a_{uv}).
\end{eqnarray*}

We now prove (i).  Observe that
\begin{equation}
\label{J_mu}
B
= p_N^{-1} ((A'_N)^{\g_N})=\left\{ x\in A_N:  wx  \in
\Ker p_N \,  \,\, {\rm {for \,\,\,  all}}  \,\,\,  w\in \g_N\right\}.
\end{equation}
Let us pick any $x,y\in B$ and show that $\bracket{x}{y}\in B$. We
need to show that $ w\bracket{x}{y} \in \Ker p_N$ for all $w\in
\g_N$. By the definition of~$p_N$, the
  ideal
 $\Ker p_N\subset A_N$ is generated by the elements $(\mu-1)_{ij}= \mu_{ij}
-\delta_{ij}$ with $i,j\in\{1,\dots,N\}$. Since $x,y\in B$, we have
$wx= \sum_a r_a (\mu-1)_{i_aj_a}$ and $wy= \sum_b s_b
(\mu-1)_{k_bl_b}$ where $a,b$ run over finite sets of indices,
  $r_a, s_b\in A_N$, and
$i_a,j_a, k_b,l_b  \in \{1,\dots,N\}$. Then
\begin{eqnarray*}
w\!\bracket{x}{y} = \bracket{wx}{y} + \bracket{x}{wy}
&=&  \sum_{a} \bracket{r_a}{y} (\mu-1)_{i_aj_a} + \sum_a r_a \bracket{(\mu-1)_{i_aj_a}}{y}\\
& & + \sum_b \bracket{x}{s_b}(\mu-1)_{k_bl_b} + \sum_b s_b \bracket{x}{(\mu-1)_{k_bl_b}}. 
\end{eqnarray*}
 The following is deduced from \eqref{mu_action} and the fact that
$x,y\in B$:
$$
p_N(w\bracket{x}{y}) = 2 \sum_a p_N(r_a) p_N( f_{j_ai_a}y) +  2 \sum_b p_N(s_b) p_N(f_{l_bk_b}x) = 0.
$$

We now  verify (ii). Since the map $p_N\vert_B:B\to B'$ is
surjective and the bracket $\bracket{-}{-}$ is skew-symmetric, it is
enough to verify that $p_N\bracket{x}{y}=0$ for any $x\in \Ker p_N$
and $y\in B$. It suffices to check the case where $x=r(\mu-1)_{ij}$
with $r\in A_N$ and $i,j\in \{1,\dots,N\}$. Then,  applying
\eqref{mu_action} again, we obtain that
$$
p_N\bracket{x}{y} = p_N \left(\bracket{r}{y}(\mu-1)_{ij}  + r \bracket{(\mu-1)_{ij}}{y}\right)
= 2 p_N(r) p_N( f_{ji} y) = 0.
$$

\vspace{-0.5cm}
\end{proof}

\subsection{Peripheral loops}  \label{peripheral}

   Let $\Sigma$ be an oriented surface with  base point $*\in
\partial \Sigma$, such
that the   component of $
\partial \Sigma$ containing $\ast$ is   a circle.
(This is always the  case if $\Sigma$ is compact.) Let 
$\nu \in  \pi= \pi_1(\Sigma,*)$ be represented by this circle component    viewed
    as a loop   based at $\ast$  with orientation  induced from that of $\Sigma$.   
    Then $\overline \nu =\nu^{-1} \in \pi$  is a moment map
for the quasi-Poisson double bracket $\double{-}{-}^s$   in $\kk\pi$   introduced in Section \ref{PMT}.
Indeed, by formula \eqref{db_hif+}, we have 
$\double{\overline\nu}{a}^\eta = a \otimes \overline \nu - 1\otimes \overline \nu a$  for all $a\in \pi$ so that
$$
\double{\overline \nu}{a}^s = a\otimes \overline \nu
 + a\overline \nu \otimes 1 -\overline \nu \otimes a - 1 \otimes \overline\nu a.
$$
Thus,  formulas \eqref{mu^m}--\eqref{mu^-m_bis} apply to $\mu=\overline{\nu}$.
That the  subalgebra of $(\kk\pi)_N$ generated by the 
$(\nu^k)_{ij}$'s (with $k\in \ZZ$ and  $i,j\in\{1,\dots,N\}$) is closed under the  bracket
can also be deduced from the naturality of $\double{-}{-}^s$ under  inclusions of surfaces:
if $\Sigma$ is not a 2-disk, then $\nu\in \pi$ is
an element of infinite order,  and  the   subalgebra in question   is
the $N$-th quasi-Poisson algebra associated with the annulus $S^1\times [0,1]$.

\subsection{Arbitrary oriented surfaces}\label{generalizswb}
 Let  $\Sigma$ be an oriented surface possibly without
boundary.   Pick a   point $\ast\in \Sigma$ and set
$\pi=\pi_1(\Sigma,*)$ and $A = \kk \pi$. If $\ast\notin \partial
\Sigma$, then we do not have a pairing in $A$ similar to the pairing
of Section \ref{pairingeta} and consecutively have neither a
corresponding double bracket in $A$ nor a quasi-Poisson bracket
in~$A_N$. However, for any choice of $\ast$,  the subalgebra $
A^t_N$ of $A_N$ generated by
  $\tr( \check A) \subset A_N $ has a natural
Poisson bracket which we now define. Recall first the Goldman Lie
bracket $\langle -,- \rangle_{\operatorname{G}}$ in $\check A =A/[A,
A]=\kk \check \pi$. By definition,  for any   $a,b\in \check \pi$,
$$
\langle a, b \rangle_{\operatorname{G}} =   \sum_{p\in \alpha \cap \beta} \varepsilon_p(\alpha,\beta)\  \alpha_p \beta_p
$$
where   $\alpha, \beta$ are  generic loops  in $\Sigma$ representing
$a,b $ and $\alpha_p\beta_p \in \check \pi$ is represented by the
product of the loops $\alpha$, $\beta$ based at $p$.

 \begin{theor}\label{th:quotient}
For any $N\geq 1$, the  algebra $A^t_N$    has a unique Poisson
bracket $\{-,-\}$  such that $\bracket{\tr   a}{\tr   b } = 2  \tr (
\langle     a ,    b \rangle_{\operatorname{G}})$ for all $  a,
b\in \check A$.
\end{theor}

\begin{proof}
The uniqueness of $\bracket{-}{-}$ is obvious and we need only
to prove the existence. When $\ast\in \partial \Sigma$, our
quasi-Poisson bracket in $A_N$ restricts to a Poisson bracket in
$A_N^{\g_N}$ which further restricts to a Poisson bracket   in
$A^t_N$ satisfying the conditions of the theorem. Suppose that
$\ast\in \Sigma \setminus \partial \Sigma$. Let $\Sigma_\circ$
be the oriented surface obtained from $\Sigma$ by removing the
interior of a closed embedded 2-disk  $D\subset \Sigma
\setminus \partial \Sigma$ such that $* \in
\partial D$.   Set $A_\circ=\kk\pi_\circ$ where $\pi_\circ=
\pi_1(\Sigma_\circ,*)$. The algebra homomorphism $ A_\circ \to A$
induced by the inclusion $\Sigma_\circ \subset \Sigma$ is
surjective, and its kernel is the two-sided ideal generated by
$\overline\nu-1$  where $\overline \nu\in \pi_\circ$ is the homotopy
class of $\partial D$ with the orientation induced from that of $D
\subset \Sigma$.   By Lemma \ref{quotient},  the quasi-Poisson
double bracket $\double{-}{-}^s$ in $A_\circ$  induces a Poisson
bracket
  in $A^t_N$. Since the Lie bracket in $\check A_\circ$
induced by $\double{-}{-}^s$ is twice the Goldman   bracket   of
$\Sigma_\circ$,    this Poisson bracket
  in $A^t_N$  satisfies the conditions of the theorem.
\end{proof}

If $\Sigma$ is   connected, then a different choice of a base point in $\Sigma$ results in a Poisson
algebra    canonically   isomorphic to $(A^t_N, \bracket{-}{-})$.
Indeed,    a path $\gamma$ in $\Sigma$ leading from  $\ast $ to $\ast' \in
\Sigma$ determines an algebra isomorphism $ \gamma_{\#}: A \to
  A'=\kk \pi_1 (\Sigma, \ast')   $. The induced isomorphism $ A_N\to
A'_N$ carries $A^t_N $ onto $ (A')^t_N$. Since the isomorphism
$\check \gamma_{\#}:\check A \to \check A' $ preserves the Goldman
bracket and does not depend on the choice of $\gamma$, the resulting
isomorphism $A^t_N \to (A')^t_N$ preserves the Poisson bracket and
does not depend on the choice of $\gamma$.

 When $\kk$ is a field of
characteristic zero and $\Sigma$ is compact,
$A^t_N=A_N^{G_N}$ (see the end of Section \ref{AMT0++}) and Theorem
\ref{th:quotient} yields  a Poisson bracket in $A_N^{G_N}$. When
$\kk=\RR$ or $\kk= \CC $ this is twice the bracket studied in
\cite{Go2}.

 \section{Generalization  to Fuchsian groups}\label{GFG}

In this last section, we generalize Theorems \ref{CQPS} and
\ref{th:quotient} to so-called weighted   surfaces and briefly
discuss connections with Fuchsian groups.

 \subsection{Weighted surfaces}\label{kd01}

  Let   $\Sigma$ be an oriented surface with $\partial
\Sigma\neq \emptyset$ endowed with base point $\ast$ (possibly not
lying on $\partial \Sigma$).  Note that each   circle component
 $X$ of $\partial \Sigma  $   determines a conjugacy class $[X]$ of elements of $\pi=\pi_1(  \Sigma, \ast) $.
 A typical representative of this class has the form $\alpha\beta \alpha^{-1}$ where $ \alpha$ is a path from
 $\ast$ to a point of   $X$   and $\beta$ is a loop going  once along $X$ in the direction determined by the orientation of
 $\Sigma$.

A \emph{weight}  on $\Sigma$ is a map from the set of
  circle components of $\partial \Sigma \setminus \{\ast\}$  to $\ZZ$. Given a weight $n$ on $\Sigma$,    consider the normal
subgroup $\langle n \rangle\subset \pi$ generated by all elements of
the form $x^{n(X)}$ where $X$ runs over the circle components of
$\partial \Sigma \setminus
  \{\ast\}$ and $x$ runs over   $[X]\subset \pi$. The quotient
  group
   $\pi'= \pi/\langle n \rangle $ is called the
   \emph{group of the  weighted surface} $(\Sigma,
   n)$. Set $A=\kk \pi$, $A'=\kk \pi'$, and let $p:A \to
 A'$   be the algebra homomorphism  induced by the projection $\pi \to \pi'$. Recall that
$$
 \check A=A/[A,A]= \kk\check \pi \quad {\rm {and}} \quad \check A'=A'/[A',A']= \kk\check \pi',
$$
  and denote by  $\check p:   \check A\to \check A'$  the linear map induced by  $p$.

We show   that   brackets   associated with $A$ induce corresponding brackets for~$A'$.
Our first observation is that the Goldman Lie bracket
$\langle - , - \rangle_{\operatorname{G}}$   in $\check A$ induces  a unique Lie
bracket   $\langle - , - \rangle'_{\operatorname{G}}$   in $\check A'$ such that the following diagram commutes:
\begin{equation}\label{eq-braiding1}
\begin{split}
\xymatrix@R=0.8cm @C=3cm {
\check A \otimes \check A \ar[r]^-{\langle - , - \rangle_{\operatorname{G}}}\ar[d]_{\check  p \otimes \check  p} & \check A \ar[d]^{\check  p  } \\
  \check A'\otimes \check A'  \ar[r]^-{\langle - , - \rangle'_{\operatorname{G}}} &  \check A'.
}
\end{split}
\end{equation}
This follows from the fact that, for any circle component $X$ of
$\partial \Sigma \setminus \{\ast\}$, all integral powers of
elements of  $ [X]$ are central for the Goldman  bracket of
$\Sigma$.

  \begin{theor}\label{gaws} Suppose that $\ast\in \partial
\Sigma$. Then the following claims hold.

  (i) Let $N\geq 1$ and $p_N:A_N \to
 A'_N$ be the algebra homomorphism induced by   $p$.  Let $\bracket{-}{-}$ be the quasi-Poisson bracket
in $A_N$ produced by Theorem~\ref{CQPS}.   There is a unique map
$\bracket{-}{-}': A'_N \times A'_N\to A'_N$ such that the following
diagram commutes:
 \begin{equation}\label{eq-braiding1+-+}
\begin{split}
\xymatrix@R=0.8cm @C=3cm {
A_N \otimes A_N \ar[r]^-{\bracket{-}{-}}\ar[d]_{p_N \otimes p_N} & A_N \ar[d]^{p_N  } \\
A'_N\otimes A'_N  \ar[r]^-{\bracket{-}{-}'} & A'_N\,.
}
\end{split}
\end{equation}
The map $\bracket{-}{-}'$ is a quasi-Poisson bracket in the
   $(G_N,\g_N)$-algebra    $A'_N$.

(ii)    For all $N \geq 1$,  the trace map $\tr: \check A' \to
 (A'_N)^{\g_N}$ is a homomorphism of Lie algebras where $ \check
 A'$ is endowed with the Lie bracket   $2 \langle - , - \rangle'_{\operatorname{G}}$    and $(A'_N)^{\g_N}$ is
 endowed with the restriction of the bracket $\bracket{-}{-}'$.
\end{theor}

\begin{proof} The uniqueness of  $\bracket{-}{-}'$ follows from the surjectivity of   $p_N$.
Since the bracket $\bracket{-}{-}$ is quasi-Poisson, the
commutativity of   \eqref{eq-braiding1+-+} implies that the bracket
$\bracket{-}{-}'$ is also quasi-Poisson. We need only to prove the
existence  of $\bracket{-}{-}'$.  We begin by proving that there is
a unique pairing $\eta': A'\times A'\to
 A'$ such that the following diagram commutes:
 \begin{equation}\label{eq-braiding2}
\begin{split}
\xymatrix@R=0.8cm @C=3cm {
A \otimes A \ar[r]^-{\eta}\ar[d]_{p \otimes p} & A \ar[d]^{p  } \\
A'\otimes A'   \ar[r]^-{\eta'} & A'.
}
\end{split}
\end{equation}
The uniqueness of  $ \eta'$ follows from the surjectivity of $p $,
and we need only to prove the existence of $\eta'$. Consider the
two-sided ideal $ \Ker
 p$ of $
 A$. We shall prove that $p\eta(\Ker
 p,A)=0$.   As a left ideal, $\Ker
 p$ is generated by the expressions of
 type $x^{n(X)}-1$ where $X$ runs over the circle components of
$\partial \Sigma \setminus
  \{\ast\}$ and $x$ runs over   $[X]\subset \pi$. Since $\eta$
 is a left Fox derivative in the first variable, it suffices to prove that $p\eta(x^{n(X)}, A)=0$ for all    $X$ and $x\in [X]$.
 Let $x=\alpha\beta \alpha^{-1}$ where $ \alpha$ is a path from
 $\ast$ to a point of   $X$   and $\beta$ is a loop in
 $X$. Clearly,  $x^{n(X)}=\alpha\beta^{n(X)} \alpha^{-1}$. Any loop $\gamma$ in $\Sigma$ based at $\ast$  can be
 deformed to avoid $\beta$ and to meet  $\alpha$
 transversely in a finite set of simple points.  A direct
 application of \eqref{eta} shows that the contribution of each of these points
  to $\eta(x^{n(X)}, \gamma)=\eta(\alpha\beta^{n(X)} \alpha^{-1}, \gamma)$ has the
 form $\pm (a- x^{n(X)} a) $ for some $a\in \pi$. Therefore $p\eta(x^{n(X)},
 \gamma) =0$. A similar argument shows that $p\eta(A, \Ker
 p )=0$. This
 implies the existence of $\eta'$.

The properties of $\eta$ recalled in Section~\ref{PMT}
 imply that $\eta'$ is an F-pairing and $\eta'+\overline{\eta'}=-\rho_1$.
 The double bracket $\double{-}{-}'$  in  $A'$ defined by the skew-symmetric F-pairing
 $\eta'-\overline{\eta'} = 2\eta' + \rho_{1}$ in $A' $ makes
 the following diagram commute:
  \begin{equation}\label{eq-braiding1++}
\begin{split}
\xymatrix@R=0.8cm @C=3cm {
A \otimes A \ar[r]^-{\double{-}{-}^s}\ar[d]_{p \otimes p} & A \otimes A \ar[d]^{p \otimes p  } \\
A'\otimes A'   \ar[r]^-{\double{-}{-}'} & A' \otimes A '\, .
}
\end{split}
\end{equation}
Lemma~\ref{homotopy triple bracket} implies   that $\double{-}{-}'$
is quasi-Poisson. It induces a   quasi-Poisson
 bracket $\{ -, -\}'$ in  $A'_N$  by  Lemma~\ref{fdbtopb}.
  The commutativity
 of the diagram \eqref{eq-braiding1++} implies the commutativity
 of the diagram \eqref{eq-braiding1+-+}.

 The second claim of the theorem is proved following the lines
of Section~\ref{endofmain}.
\end{proof}

The next theorem is an analogue of   Theorem \ref{th:quotient}  
and is proved similarly. This theorem covers
all choices of $\ast$ in $\Sigma$.

\begin{theor}\label{th:quotient++}
For any $N\geq 1$, the    subalgebra  $(A')^t_N \subset A'_N$
generated by $\tr (\check A')$ has a unique Poisson bracket
$\{-,-\}'$  such that $\bracket{\tr a}{\tr b }' = 2  \tr ( \langle a
,    b \rangle'_{\operatorname{G}})$ for all $ a, b\in \check A'$.
\end{theor}

 All the brackets discussed in this section are natural with respect to    the action of
 orientation-preserving    weight-preserving
self-homeomorphisms of the pair $(\Sigma, \ast)$.

\subsection{Examples} 1.  Let $\Sigma$ be   a compact connected
oriented surface of
 genus $g\geq 0$ with   $m+1\geq 2$   boundary components. Picking a base point    on
 one of these components and assigning   integers $n_1, \dots,
 n_m$ to the other components we obtain a  weight on~$\Sigma$. The  group, $\pi'$, of the resulting weighted surface
 is
 presented by $ 2g+m$  generators $p_1, q_1, \dots, p_g, q_g, z_{1}, \dots , z_m $ and the
 relations
 $z_1^{n_1}= \cdots = z_m^{n_m}=1$. This   is a Fuchsian group with at least one parabolic or hyperbolic generator.
Theorem \ref{gaws} yields a natural quasi-Poisson bracket in the
algebra $(\kk\pi')_N$   for all $N\geq 1$.  For   the generators
 chosen as on Figure \ref{surface}, 
the product  $ [p_1, q_1] \cdots [p_{g }, q_{g }] z_1\cdots z_m\in \pi'$ is a moment map for
the double bracket   in $\kk\pi'$ defined in the proof of Theorem~\ref{gaws}.

  2. Let $\Sigma$ be a compact connected oriented surface of
 genus $g\geq 0$ with $m\geq 1$ boundary components. Picking   $\ast \in \Sigma\setminus \partial \Sigma$
  and assigning   integers $n_1, \dots,
 n_m$ to the boundary components we obtain a  weight on~$\Sigma$. The  group, $\pi'$, of the resulting weighted surface   is a
   Fuchsian group
 with  generators $p_1, q_1, \dots, p_g, q_g, z_{1}, \dots , z_m $ subject to the
 relations
 $z_1^{n_1}= \cdots = z_m^{n_m}=1$ and 
 $ [p_1, q_1] \cdots [p_{g }, q_{g }] z_1\cdots z_m=1$.  
Theorem~\ref{th:quotient++} yields a natural Poisson bracket in the
algebra   $(\kk\pi')^t_N$   for all $N\geq 1$.  When $\kk$ is a
field of characteristic zero, this gives a natural Poisson bracket
in the algebra    $(\kk\pi')^{\GL_N( \kk)}_N$.

\appendix

\section{The actions of $G_N$ and $\g_N$ on $A_N$ re-examined}\label{appendix_schemes}

Given an algebra $A$ and an integer
$N\geq 1$, denote by $\mathcal{X}_N^A$ the affine scheme (over
$\kk$) whose set of $B$-points    $\mathcal{X}_N^A(B)$   is
$\Hom_{\Alg} (A, \Mat_N(B))$ for any commutative algebra $B$.   The
coordinate algebra of $\mathcal{X}_N^A$ is the commutative algebra
$A_N$ introduced in Section \ref{AMT0}.   In this appendix, we
analyze from the viewpoint of group schemes   the actions of
$G_N=\GL_N(\kk)$ and $\g_N=\mathfrak{gl}_N(\kk)$ on $A_N$ defined in
Section \ref{AMT0++}.  For the language of group schemes, we refer
the reader to \cite{Ja}.

For any commutative algebra $B$, the group $\GL_N(B)$ acts on
$\Mat_N(B)$ by the conjugation $M\mapsto gMg^{-1}$ where  $M\in
\Mat_N(B)$ and $g\in \GL_N(B)$. This induces an action of $\GL_N(B)$
on $\Hom_{\Alg} (A, \Mat_N(B))$ which is natural in $B$.   These
  actions constitute an action of the group scheme $\GL_N$  on the affine
scheme $\mathcal{X}_N^A$.

\begin{lemma}\label{schemes}
The action \eqref{GN_on_AN} of $G_N=\GL_N(\kk)$ on $A_N$ is induced
by the action of   $\GL_N$  on   $\mathcal{X}_N^A$. The action
\eqref{gN_on_AN} of $\g_N\simeq \operatorname{Lie}(\GL_N)$ on $A_N$
is   the infinitesimal version of the action of $\GL_N$  on
$\mathcal{X}_N^A$.
\end{lemma}

\begin{proof}
For any commutative algebra $K$, the action of the group scheme
$\GL_N$ on  $\mathcal{X}_N^A$ induces
   an action of $\GL_N(K)$   on $A_N \otimes K$ by $K$-algebra
automorphisms  as follows. Consider the affine scheme over $K$
obtained from $\mathcal{X}_N^A$ by changing the coefficients from
$\kk$ to $K$. We can view $A_N \otimes K$ as the coordinate algebra
of this scheme through the canonical isomorphism
\begin{equation}\label{change_coef}
\Hom_{K-\CAlg}(A_N \otimes K, B) \simeq \Hom_{\CAlg}(A_N, B)
\end{equation}
for any commutative $K$-algebra $B$. Then
\begin{equation}\label{unique_action_bis}
r(gx)=( g^{-1} r)(x)
\end{equation}
for all  $r\in\Hom_{K-\CAlg}(A_N \otimes K, B)$, $g\in \GL_N(K)$ and $x\in A_N\otimes K$.
Here, the action of $\GL_N(K)$ on $\Hom_{K-\CAlg}(A_N \otimes K, B)$
  is induced by the action of $\GL_N(B)$ on $\mathcal{X}_N^{A}(B)$ 
using  \eqref{change_coef}, \eqref{adjunction}, and the canonical map $\GL_N(K) \to \GL_N(B)$.  
This yields an action of $\GL_N(K)$ on   $A_N
\otimes K$. Applying \eqref{unique_action_bis} to $B=A_N\otimes K$,
$r=\id_{A_N \otimes K}$ and $x=a_{ij}\otimes 1$ (with $a\in A$ and
$i,j\in \{1,\dots,N\}$), we obtain
\begin{equation}\label{beurk}
g (a_{ij}\otimes 1)  = (g^{-1} \id_{A_N\otimes K})(a_{ij}\otimes 1) =   a_{kl} \otimes (g^{-1})_{i,k} g_{l,j}
\end{equation}
for any $g\in \GL_N(K)$. For $K=\kk$, we recover the action
\eqref{GN_on_AN} of $G_N$ on   $A_N$.

The action of $\GL_N(K)$ on $A_N\otimes K$  is natural in $K$ 
and determines thus an action of the group  scheme $\GL_N$ on $A_N$.  
Recall that an action  of a group scheme is equivalent to a
comodule over its coordinate  Hopf  algebra, see \cite[\S I.2.8]{Ja}. 
 Hence we obtain   a comodule structure in $A_N$, which is given by a linear map
$\Delta_N: A_N \to A_N \otimes \kk[\GL_N]$ where $\kk[\GL_N]$ is the
coordinate algebra of $\GL_N$. Since   $\GL_N$ acts on $A_N$ by
algebra automorphisms, $\Delta_N$ is an algebra homomorphism. We now
compute   $\Delta_N$ on the generator $a_{ij}\in A_N$ for any $a\in
A$ and $i,j\in\{1,\dots,N\}$. We have $\Delta_N(a_{ij}) =
\id_{\kk[\GL_N]} (a_{ij}\otimes 1)$ where the action of
$\id_{\kk[\GL_N]}\in  \Hom_{\CAlg}(\kk[\GL_N], \kk[\GL_N])=  \GL_N(\kk[\GL_N])$ 
is given by \eqref{beurk}   with $K=\kk[\GL_N]$.  
The matrix $\id_{\kk[\GL_N]}$ is equal to $(x_{ij})_{i,j}$ where
$x_{ij}\in \kk[\GL_N]$ maps any $M\in \GL_N(B)$ to the $(i,j)$-th
entry of $M$ for any commutative algebra $B$, and
$(\id_{\kk[\GL_N]})^{-1}$ is the matrix $(\overline x_{ij})_{i,j}$
where $\overline x_{ij}\in \kk[\GL_N]$ maps any $M\in \GL_N(B)$ to
the $(i,j)$-th entry of $M^{-1}$. Therefore
\begin{equation}\label{DeltaN}
\Delta_N(a_{ij}) = a_{kl} \otimes \overline x_{ik} x_{lj}.
\end{equation}

Recall that the Lie algebra of a group scheme is its tangent space
at the unit point,  and any representation of a group scheme carries
the structure of a module over the corresponding Lie algebra, see
\cite[\S I.7]{Ja}.  The latter structure is the \lq\lq infinitesimal version" of the representation of the group scheme. We have
$$
\Lie(\GL_N) = \left\{\mu \in \Hom(\kk[\GL_N], \kk): \mu(1)=0, \mu(I_1^2)=0\right\} \simeq \Hom(I_1/I_1^2,\kk)
$$
where $I_1\subset \kk[\GL_N]$ is the ideal   consisting of all $x
\in \kk[\GL_N]$ such that $x(1)=0$.  Viewing $\kk$   as a
$\kk[\GL_N]$-module via the map $\kk[\GL_N] \to \kk, x\mapsto x(1)$,
we can identify elements of $\Lie(\GL_N)$ with derivations
$\kk[\GL_N]\to  \kk$.  The action of $\Lie(\GL_N)$ on $A_N$ induced
by the action of $\GL_N$ on $A_N$
  carries  any $\mu
\in \Lie(\GL_N)$ to   the composite map
$$
A_N \stackrel{\Delta_N}{\longrightarrow} A_N \otimes \kk[\GL_N]
\stackrel{\id \otimes \mu}{\longrightarrow} A_N \otimes \kk  \simeq A_N.
$$
Since $\Delta_N$ is an algebra homomorphism and $\mu$   is a
derivation, the composite map is   a derivation, and we  obtain an
action of $\Lie(\GL_N)$ on $A_N$ by
 derivations.

   The Lie algebra $\g_N=\mathfrak{gl}_N(\kk)$ can be
identified with $\Lie(\GL_N)$ by sending any $w\in \g_N$ to the
linear map $\mu_w: \kk[\GL_N]\to \kk$ defined by
$$
\mu_w(x) =  q(x(1+ \varepsilon w)) \quad \hbox{for all $x\in \kk[\GL_N]$}
$$
 where $\varepsilon$ is the generator of the algebra of dual numbers $\kk[\varepsilon]/(\varepsilon^2)$, 
$x(1+ \varepsilon w)$ is the evaluation of $x$
at  $1+ \varepsilon w\in \GL_N(\kk[\varepsilon]/(\varepsilon^2))$ and $q:
\kk[\varepsilon]/(\varepsilon^2) \to \kk$ is defined by $q(k+l\varepsilon)=l$ for
all $k,l\in \kk$.  Through this identification, the action of
$\Lie(\GL_N)$ on $A_N$ determines an action of $\g_N$ on $A_N$. For
any $w\in \g_N$,  $a\in A$ and $i,j\in \{1,\dots,N\}$, 
\begin{eqnarray*}
w a_{ij}=(\id \otimes \mu_w) \Delta_N(a_{ij}) & \stackrel{\eqref{DeltaN}}{=} & \mu_w( \overline x_{ik} x_{lj}) a_{kl} \\
&=& q( \overline x_{ik}(1+\varepsilon w) x_{lj}(1+ \varepsilon w))  a_{kl}\\
&=& q((\delta_{ik} - \varepsilon w_{i,k})(\delta_{lj}+ \varepsilon w_{l,j})) a_{kl}\\
&=& (-w_{i,k} \delta_{lj} + \delta_{ik} w_{l,j}) a_{kl}
\ =\  -w_{i,k} a_{kj}  + a_{il} w_{l,j}.
\end{eqnarray*}
Thus, this action of $\g_N$ coincides with the action
\eqref{gN_on_AN}.
\end{proof}

Lemma \ref{schemes} has the following useful consequence. If $\kk$
is an algebraically closed field, then $A_N^{G_N}$ coincides with
the $\GL_N$-invariant part of $A_N$ as a representation of the group
scheme $\GL_N$, see \cite[\S I.2.8]{Ja}. Therefore $A_N^{G_N}\subset A_N^{\g_N}$.

\section{Comparison of    quasi-Poisson structures}\label{appendix_comparison}

 We compare our quasi-Poisson structure on  representation manifolds  with those defined in \cite{AKsM}.
We assume that $\kk=\RR$ and use   notations of Section \ref{QPMRS}.
  In order to make the computations of \cite{AKsM} compatible with our conventions, 
we multiply the quasi-Poisson bivector fields appearing in \cite{AKsM}  by~$2$.

\subsection{Preliminaries}

 In this subsection, $G$ is a Lie group whose Lie algebra $\g$ is endowed with a $G$-invariant
non-degenerate symmetric bilinear form $\cdot: \g \times \g \to \RR$.
The construction of \cite{AKsM} uses three main ingredients:  
 a so-called \emph{fusion product} of quasi-Poisson manifolds,
a canonical quasi-Poisson structure on  $G$ and a   preferred   quasi-Poisson structure on  $G \times G$.
We briefly recall the relevant definitions.  

The  fusion product is defined as follows.   
Let $\mathfrak{d} = \g \oplus \g$ be the Lie algebra of the Lie group $D=G \times G$.
Pick a  basis $(e_i)_i$ of $\g$, let  $(e_i^\sharp)_i$ be the basis of  $\g$ dual to $(e_i)_i$ with respect to the form $\cdot$ and set
\begin{equation}\label{psi}
\psi = \sum_{i} (e_i^\sharp,0) \wedge (0,e_i)  \ \in \Lambda^2 \mathfrak{d}.
\end{equation}
The bivector $\psi$ is independent of the choice of the basis $(e_i)_i$ 
because  it corresponds to the skew-symmetric bilinear pairing  
$$
\mathfrak{d} \times \mathfrak{d} \longrightarrow \RR,
\big((v_1,v_2),(w_1,w_2)\big)\longmapsto v_1\cdot w_2 - v_2 \cdot w_1
$$
through the canonical isomorphism $\Lambda^2 \mathfrak{d} \simeq \Lambda^2 \mathfrak{d}^*$ 
induced by the form  $\cdot: \g \times \g \to \RR$. 
 
 Consider $G$-manifolds $M_1$ and $M_2$ equipped with quasi-Poisson bivector fields $P_1$ and $P_2$, respectively.
Let $\operatorname{pr}_i: M_1 \times M_2 \to M_i$ be the $i$-th projection for $i=1,2$.
We  identify $T(M_1\times M_2)$ with $\operatorname{pr}_1^*(T(M_1))\oplus \operatorname{pr}_2^*(T(M_2))$ 
where $T$ stands for the tangent bundle of a manifold.  
The Lie group $D =G\times G$ acts on $M_1 \times M_2$ and, under this action, the bivector 
$\psi\in \Lambda^2 \mathfrak{d}$ generates a bivector field $\psi_{M_1\times M_2}$ of $M_1 \times M_2$.
If we endow $M_1 \times M_2$ with the diagonal $G$-action, then
\begin{equation}\label{fusion}
P_1 \circledast P_2 = \operatorname{pr}_1^*(P_1) + \operatorname{pr}_2^*(P_2) - \psi_{M_1\times M_2}
\end{equation}
is a quasi-Poisson bivector field on $M_1 \times M_2$, see \cite[\S 5]{AKsM}.

For any $v\in \g$,   denote by $v_G^{L}$ the vector field on $G$ generated by $v$
through the left action of  $G$   on   itself by  $(g,m)\mapsto mg^{-1}$.
Similarly,   denote by $v_G^R$ the  vector field on $G$ generated by $v$
through the left action of  $G$   on   itself by  $(g,m)\mapsto gm$.
Note that the vector field $v_G^{L}$ is left-invariant while $v_G^{R}$ is right-invariant. 
(Our notation differs from that of  \cite{AKsM} where $v_G^L$ is denoted    by $v^L$ and   $v_G^R$ is denoted    by $-v^R$.)
Then
\begin{equation} \label{P_G}
P_G = \sum_i (e_i)_G^L \wedge (e_i^\sharp)_G^R
\end{equation}
is a quasi-Poisson bivector field on the underlying manifold of $G$ endowed with left action of $G$ by conjugations, see \cite[\S 3]{AKsM}. 
Note that $P_G$  does not depend on the choice of the   basis $(e_i)_i$ of $\g$
since it is the bivector field generated by $-\psi$ when 
the Lie group $D$ acts on $G$ by $((g_1,g_2),g) \mapsto g_1 g g_2^{-1}$.

Let $G$ act  diagonally on the left of $D=G\times G$ by conjugations. Then
\begin{eqnarray}
\label{P_D} P_D &=&
-\sum_{i} \pr_1^*(e_i)^L_G \wedge \pr_2^*(e_i^\sharp)^R_G -\sum_{i} \pr_1^*(e_i)^R_G \wedge \pr_2^*(e_i^\sharp)^L_G\\
\notag &&-\sum_{i} \left(\pr_1^*(e_i)_G^R + \pr_2^*(e_i)_G^L \right)
\wedge \left(\pr_1^*(e_i^\sharp)_G^L +  \pr_2^*(e_i^\sharp)^R_G\right)
\end{eqnarray}
is a quasi-Poisson bivector field on $D$, see \cite[Examples \ 5.3 \&  5.4]{AKsM}.
The idea behind the definition of  $P_D$ is as follows.
There is a general procedure  which transforms a quasi-Poisson $(G\times G)$-manifold  into a quasi-Poisson $G$-manifold 
and which generalizes the fusion product of quasi-Poisson $G$-manifolds, see \cite[Proposition 5.1]{AKsM}.
Applying this procedure three times to the manifold $D =G\times G$ 
(where $G^4$ acts by the left/right multiplication on the two factors), one obtains  $P_D$.

\subsection{Comparison}

Assume   that $G =\GL_N(\RR)$ with $N\geq 1$.
 The Lie algebra $\g =\mathfrak{gl}_N(\RR)$ of $G$ is equipped with
the trace form defined by $v\cdot w = \tr(vw)$ for any   $v, w\in \g$.  
Let $\Sigma$ be a compact connected oriented  surface of genus $g\geq 0$ with  $m+1\geq 1$ boundary components.
We choose a base point $\ast \in \partial \Sigma$,
and   fix  a basis $(p_1,q_1,\dots,p_g,q_g,z_1,\dots,z_m)$ of   the free group $\pi=\pi_1(\Sigma,\ast)$ as shown on Figure \ref{surface}.
We use this basis to identify  $\mathcal H= \Hom(\pi,G)$  with $D^{g} \times G^m$ where $D=G \times G$.
By \cite{AKsM}, we have the following  quasi-Poisson bivector field on $\mathcal H$:
$$
P'= \underbrace{P_D \circledast \cdots \circledast P_D}_{\hbox{\scriptsize $g$ times}} \circledast
\underbrace{P_G \circledast \cdots \circledast P_G}_{\scriptsize \hbox{$m$ times}}.
$$
The following theorem shows that  the  resulting quasi-Poisson manifold $(\mathcal H,P')$ 
coincides with  the quasi-Poisson manifold  $(\mathcal H,P)$ produced by Theorem \ref{CQPS}.  

\begin{theor}\label{P'_P}
The   bivector field $P'$  
is equal to the   bivector field $P$ of  Theorem~\ref{CQPShahaha}.
\end{theor}

\begin{proof} We first verify the equality $P'=P$ in the case where 
  $g=0$ and $m=1$.
Then the group $\pi$ is freely generated by $z=z_1$
and   $\mathcal{H}=G$.  
Since the elementary matrices $f_{ij}$ and $f_{ji}$ (with $i,j\in \{1,\dots,N\}$) provide dual bases of $\g$, 
\eqref{P_G} gives
$$
P'=P_G =  (f_{rs})_G^L \wedge (f_{sr})_G^R
$$ 
so that, for any $i,j,k,l \in \{1,\dots, N\}$,
\begin{eqnarray*}
\bracket{\tilde z_{ij}}{\tilde z_{kl}}_{P'}& =&  (f_{rs})_G^L(\tilde z_{ij})\ (f_{sr})_G^R(\tilde z_{kl})
- (f_{rs})_G^L(\tilde z_{kl})\ (f_{sr})_G^R(\tilde z_{ij}).
\end{eqnarray*}
Moreover we have
\begin{equation}
\label{R_and_L} (f_{sr})_G^R(\tilde  z_{ij}) = - \delta_{is} \tilde z_{rj}
\quad \hbox{and} \quad
(f_{rs})_G^L(\tilde  z_{ij}) =  \delta_{js} \tilde z_{ir}
\end{equation}
since, at any point $m\in G$,
\begin{eqnarray*}
 (f_{sr})_G^R(\tilde  z_{ij}) (m) &=&
 \Big. \frac{d}{dt} \Big\vert_{ t=0} \tilde  z_{ij}(e^{-tf_{sr}}m)
\ = \ - \tilde z_{ij}(f_{sr}m) \ = \ -\delta_{is} m_{rj}\\
(f_{rs})_G^L(\tilde  z_{ij}) (m) &=&
 \Big. \frac{d}{dt} \Big\vert_{ t=0} \tilde  z_{ij}(me^{tf_{rs}})
 \ =  \ \tilde z_{ij}(mf_{rs}) \ = \ \delta_{js} m_{ir}.
\end{eqnarray*}
We deduce that
\begin{eqnarray*}
\bracket{\tilde z_{ij}}{\tilde z_{kl}}_{P'} & =&
( \delta_{js} \tilde z_{ir})\  (- \delta_{ks} \tilde z_{rl})
-  (\delta_{ls} \tilde z_{kr})\, (- \delta_{is} \tilde z_{rj})\\
&=& - \delta_{jk} \tilde z_{ir}  \tilde z_{rl}
+ \delta_{il} \tilde z_{kr} \tilde z_{rj}
\quad  \stackrel{\eqref{values_2}}{=} \quad \bracket{\tilde z_{ij}}{\tilde z_{kl}}_P.
\end{eqnarray*}

Next, we verify the equality $P'=P$ in the case where  $g=1$ and $m=0$.
Then the group $\pi$ is freely generated by $(p,q)=(p_1,q_1)$
and  $\mathcal{H}=D$.
We have 
\begin{eqnarray*}
P'=P_D &=& - \pr_1^*(f_{rs})^L_G \wedge \pr_2^*(f_{sr})^R_G -\pr_1^*(f_{rs})^R_G \wedge \pr_2^*(f_{sr})^L_G\\
&& - \pr_1^*(f_{rs})^R_G \wedge \pr_2^*(f_{sr})^R_G + \pr_1^*(f_{rs} )^L_G \wedge \pr_2^*(f_{sr})^L_G \\
 &&+  \pr_1^*(f_{rs})_G^L \wedge \pr_1^*(f_{sr})^R_G -    \pr_2^*(f_{rs})_G^L \wedge \pr_2^*(f_{sr})^R_G.
\end{eqnarray*}
By computations similar to the previous case, we have for any $i,j,k,l\in\{1,\dots,N\}$
\begin{eqnarray*}
\bracket{\tilde p_{ij}}{\tilde{p}_{kl}}_{P'} &=& - \delta_{jk} \tilde p_{ir}  \tilde p_{rl}
+ \delta_{il} \tilde p_{kr} \tilde p_{rj} \ \stackrel{\eqref{values_two}}{=} \ \bracket{\tilde p_{ij}}{\tilde{p}_{kl}}_{P}\\
\bracket{\tilde q_{ij}}{\tilde{q}_{kl}}_{P'} &=& -(- \delta_{jk} \tilde q_{ir}  \tilde q_{rl}
+ \delta_{il} \tilde q_{kr} \tilde q_{rj}) \ \stackrel{\eqref{values_three}}{=} \ \bracket{\tilde q_{ij}}{\tilde{q}_{kl}}_{P}.
\end{eqnarray*}
Moreover,  
\begin{eqnarray*}
\bracket{\tilde p_{ij}}{\tilde{q}_{kl}}_{P'} &=& - (f_{rs})^L_G (\tilde p_{ij}) (f_{sr})^R_G(\tilde q_{kl})
-(f_{rs})^R_G(\tilde p_{ij}) (f_{sr})^L_G(\tilde q_{kl})\\
&&
- (f_{rs})_G^R (\tilde p_{ij}) (f_{sr})^R_G (\tilde q_{kl}) + (f_{rs})_G^L (\tilde p_{ij}) (f_{sr})^L_G(\tilde q_{kl})\\
& \stackrel{\eqref{R_and_L}}{=}& -(\delta_{js}\tilde p_{ir})(-\delta_{ks} \tilde q_{rl})-(-\delta_{ir}\tilde p_{sj})(\delta_{lr} \tilde q_{ks})\\
&& -(-\delta_{ir}\tilde p_{sj})(-\delta_{ks} \tilde q_{rl}) + (\delta_{js}\tilde p_{ir})(\delta_{lr} \tilde q_{ks})\\
&=& \delta_{jk}\tilde p_{ir} \tilde q_{rl} + \delta_{il}\tilde p_{sj} \tilde q_{ks}
 - \tilde p_{kj} \tilde q_{il} + \tilde p_{il} \tilde q_{kj} \ \stackrel{\eqref{values_four}}{=} \ \bracket{\tilde p_{ij}}{\tilde{q}_{kl}}_{P}.
\end{eqnarray*}

 Theorem \ref{P'_P} follows now from the two cases  above and  the following  claim.\\

\noindent
\textbf{Claim.}
\emph{Let $\Sigma_0$ be a compact connected oriented  surface with base point $\ast\in \partial \Sigma_0$
which is  the \lq\lq boundary connected sum"
of two compact connected oriented surfaces $\Sigma_1$ and $\Sigma_2$, see Figure \ref{gluing}.
For each $i\in \{0,1,2\}$,
let $P_i$ be the quasi-Poisson bivector field on
$\mathcal H_i = \Hom(\pi_1(\Sigma_i,\ast), G)$ produced by Theorem \ref{CQPShahaha}.
Then  $\mathcal H_0 = \mathcal H_1 \times \mathcal H_2$ and  $P_0=P_1 \circledast P_2$.}\\

\begin{figure}
\labellist \small \hair 2pt
\pinlabel { $\circlearrowleft$} at 25 25
\pinlabel {$\circlearrowleft$} at 585 25
\pinlabel {$\Sigma_1$} at 135 107
\pinlabel {$\Sigma_2$} at 500 107
\pinlabel {\large $\ast$}  at 310 2
\endlabellist
\centering
\includegraphics[scale=0.25]{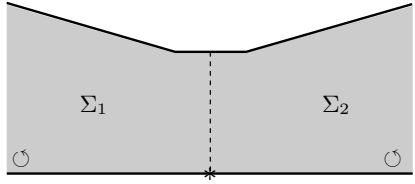}
\caption{The boundary connected  sum $\Sigma_0$  of $\Sigma_1$ and $\Sigma_2$}
\label{gluing}
\end{figure}

For   $x\in \pi_1(\Sigma_0,\ast)$ and  $i,j\in \{1,\dots,N\}$,
let $\tilde x_{ij} \in  C^{\infty}(\mathcal H_0)$ be the function
carrying any  $h \in \mathcal H_0$ to the $(i,j)$-th coefficient of the matrix $h(x)$.
In order to verify the claim,  it is enough to check that
\begin{equation}\label{the_same}
\bracket{\tilde x_{ij}}{\tilde y_{kl}}_{P_0} = \bracket{\tilde x_{ij}}{\tilde y_{kl}}_{P_1 \circledast P_2}
\end{equation}
for any $x, y\in \pi_1(\Sigma_0,\ast)$ and $i,j,k,l\in \{1,\dots,N\}$.
If both $x,y$ belong to $\pi_1(\Sigma_j,\ast)$ for some $j\in \{1,2\}$,
then \eqref{the_same} follows from the equality
$\bracket{\tilde x_{ij}}{\tilde y_{kl}}_{P_0} = \bracket{\tilde x_{ij}}{\tilde y_{kl}}_{P_j}$
(since $\double{x}{y}^s$ can be fully  computed in $\Sigma_j$)
and the fact that $\bracket{\tilde x_{ij}}{\tilde y_{kl}}_{P_1 \circledast P_2} = \bracket{\tilde x_{ij}}{\tilde y_{kl}}_{P_j}$
(since the projection $\mathcal H_1 \times \mathcal H_2 \to \mathcal H_j$
carries $P_1 \circledast P_2$ to $P_j$). It remains to consider the case where  $x\in \pi_1(\Sigma_1,\ast)$ and $y\in \pi_1(\Sigma_2,\ast)$.

Since $\double{x}{y}^\eta=0$, we have
$
\double{x}{y}^s
= 1\otimes x y + y x \otimes 1 - x \otimes y - y \otimes x
$
so that
\begin{equation}\label{P_0}
\bracket{\tilde x_{ij}}{\tilde y_{kl}}_{P_0} =
\delta_{kj}  \tilde x_{ir} \tilde y_{rl} + \tilde y_{ks} \tilde x_{sj} \delta_{il} - \tilde x_{kj} \tilde y_{il} - \tilde y_{kj}  \tilde x_{il}.
\end{equation}
Besides we have  
$$
\bracket{\tilde x_{ij}}{\tilde y_{kl}}_{P_1 \circledast P_2} \stackrel{\eqref{fusion}}{=}
 -\bracket{\tilde x_{ij}}{\tilde y_{kl}}_{\psi_{\mathcal H_1 \times \mathcal H_2}}
\stackrel{\eqref{psi}}{=}  - (f_{rs})_{\mathcal{H}_1}(\tilde x_{ij})\ (f_{sr})_{\mathcal{H}_2}(\tilde y_{kl}).
$$  
At any point $m\in \mathcal H_1$,
\begin{eqnarray*}
(f_{rs})_{\mathcal{H}_1}(\tilde x_{ij})(m) &=&
 \Big. \frac{d}{dt} \Big\vert_{ t=0} \tilde  x_{ij}\left(e^{-tf_{rs}}\, m\, e^{tf_{rs}}\right) \\
 &=& -\tilde x_{ij}(f_{rs}m) + \tilde x_{ij}(mf_{rs}) \ = \ - \delta_{ir}\tilde{x}_{sj}(m) + \delta_{js} \tilde x_{ir}(m)
\end{eqnarray*}
and, by a similar computation, at any point $n\in \mathcal H_2$,
$$
(f_{sr})_{\mathcal{H}_2}(\tilde y_{kl})(n) =
- \delta_{ks}\tilde{y}_{rl}(n) + \delta_{lr} \tilde y_{ks}(n).
$$
Therefore
\begin{eqnarray*}
\bracket{\tilde x_{ij}}{\tilde y_{kl}}_{P_1 \circledast P_2} &=&
- \left(- \delta_{ir}\tilde{x}_{sj} + \delta_{js} \tilde x_{ir}\right)\
\left(- \delta_{ks}\tilde{y}_{rl}+ \delta_{lr} \tilde y_{ks}\right)\\
&=&
- \tilde{x}_{kj} \tilde{y}_{il} + \delta_{jk} \tilde x_{ir} \tilde{y}_{rl}
+ \delta_{il} \tilde{x}_{sj} \tilde y_{ks} -  \tilde x_{il}  \tilde y_{kj}
 \stackrel{ \eqref{P_0}}{=}  \bracket{\tilde x_{ij}}{\tilde y_{kl}}_{P_0}\!.
\end{eqnarray*}

\vspace{-0.5cm}
\end{proof}


\begin{thebibliography}{CJKLS}



\bibitem[AKsM]{AKsM}
A. Alekseev, Y. Kosmann-Schwarzbach, E. Meinrenken,
\emph{Quasi-Poisson manifolds.}
Canad. J. Math. 54 (2002), no. 1, 3--29.

\bibitem[AMM]{AMM}
A. Alekseev, A. Malkin,  E. Meinrenken,
\emph{Lie group valued moment maps.} 
J. Differential Geom. 48 (1998), no. 3, 445--495.

 
\bibitem[Au]{Au}
M. Audin, 
\emph{Lectures on gauge theory and integrable systems.}
 Gauge theory and symplectic geometry (Montreal, PQ, 1995), 1--48, 
 NATO Adv. Sci. Inst. Ser. C Math. Phys. Sci., 488, Kluwer Acad. Publ., Dordrecht, 1997.
  
 
\bibitem[CS]{CS}
M. Chas, D. Sullivan,
\emph{String Topology.} 
Preprint (1999) \texttt{arXiv:}\texttt{math/9911159}.
 

\bibitem[Cb]{Cb}
W. Crawley-Boevey,
\emph{Poisson structures on moduli spaces of representations.}
J. Algebra 325 (2011), 205--215.

\bibitem[Do]{Do}
I. V. Dolgachev,
\emph{Introduction to algebraic geometry.}
Lecture notes available at
\texttt{http://www.math.lsa.umich.edu/$\sim$idolga/lecturenotes.html}

\bibitem[FR]{FR}
V. V. Fock, A. A. Rosly,
\emph{Poisson structure on moduli of flat connections on Riemann surfaces and the $r$-matrix.}
(Russian) Moscow Seminar in Math. Physics.
 English translation: Amer. Math. Soc. Transl. Ser. 2, 191, 67--86 (1999).

\bibitem[Go1]{Go1}
W. M. Goldman,
\emph{ The symplectic nature of fundamental groups of surfaces.}
Adv. in Math. 54 (1984), no. 2, 200--225.

\bibitem[Go2]{Go2}
W. M. Goldman,
\emph{Invariant functions on Lie groups and Hamiltonian flows of surface group representations.}
Invent. Math. 85 (1986), no. 2, 263--302.

\bibitem[Go3]{Go3}
W. M. Goldman,
\emph{Mapping class group dynamics on surface group representations.}
Problems on mapping class groups and related topics, 189--214, Proc. Sympos. Pure Math.,
74, Amer. Math. Soc., Providence, RI, 2006.

\bibitem[GHJW]{GHJW}
K. Guruprasad, J. Huebschmann, L. Jeffrey, A. Weinstein,
\emph{Group systems, groupoids, and moduli spaces of parabolic bundles.}
Duke Math. J. 89 (1997), no. 2, 377--412.

\bibitem[Hu]{Hu}
J. Huebschmann,
\emph{Poisson geometry of certain moduli spaces.}
Rend. Circ. Mat. Palermo (2) Suppl. No. 39 (1996), 15--35.

\bibitem[Ja]{Ja}
J. C. Jantzen,
\emph{Representations of algebraic groups.}
Second edition. Mathematical Surveys and Monographs, 107.
American Mathematical Society, Providence, RI, 2003.

\bibitem[KK1]{KK}
N. Kawazumi, Y. Kuno,
\emph{The logarithms of Dehn twists.}
Preprint (2010) \texttt{arXiv:}\texttt{1008.5017}.

\bibitem[KK2]{KK_intersections}
N. Kawazumi, Y. Kuno,
\emph{Intersections of curves on surfaces and their applications to mapping class groups.}
Preprint (2011) \texttt{arXiv:}\texttt{1112.3841}.

\bibitem[La]{La}
S. Lawton,
\emph{Poisson geometry of ${\rm SL}(3,\CC)$-character varieties relative to a surface with boundary.}
 Trans. Amer. Math. Soc. 361 (2009), no. 5, 2397--2429.

\bibitem[LbP]{LbP}
L. Le Bruyn, C. Procesi,
\emph{Semisimple representations of quivers.}
Trans. Amer. Math. Soc. 317 (1990), no. 2, 585--598.

\bibitem[Mar]{Mar}
C.--M. Marle, \emph{The Schouten--Nijenhuis bracket and interior
products.} J. Geom. Phys. 23 (1997), no. 3-4, 350--359.

\bibitem[Mas]{Mas}
G. Massuyeau, \emph{Infinitesimal Morita homomorphisms and the
tree-level of the LMO invariant.} Bull. Soc. Math. France 140:1
(2012), 101--161.

\bibitem[MT]{MT}
G. Massuyeau, V. Turaev,
\emph{Fox pairings and generalized Dehn twists.}
Preprint (2011) \texttt{arXiv:}\texttt{1109.5248}.

 
\bibitem[Pa]{Pa}
C. D. Papakyriakopoulos, 
\emph{Planar regular coverings of orientable closed surfaces.} 
Knots, groups, and $3$-manifolds (Papers dedicated to the memory of R. H. Fox), 
261--292. Ann. of Math. Studies, No. 84, Princeton Univ. Press, Princeton, N.J., 1975.

\bibitem[Pe]{Pe}
B. Perron,
\emph{A homotopic intersection theory on surfaces: applications to mapping class group and braids.}
Enseign. Math. (2) 52 (2006), no. 1-2, 159--186.
 

\bibitem[Tu]{Tu}
V. G. Turaev,
\emph{Intersections of loops in two-dimensional manifolds.}
(Russian) Mat. Sb. 106(148) (1978),   566--588.
English translation: Math. USSR, Sb. 35 (1979), 229--250.

\bibitem[VdB]{VdB}
M. Van den Bergh,
\emph{Double Poisson algebras.}
Trans. Amer. Math. Soc. 360 (2008), no. 11, 5711--5769.

\bibitem[Wo]{Wo}
S. Wolpert,
\emph{On the symplectic geometry of deformations of a
hyperbolic surface.} Ann. of Math. (2) 117 (1983), 207--234.





                     \end{thebibliography}
                     \end{document}